\documentclass[12pt,a4paper]{article}
\usepackage{amssymb,latexsym}
\usepackage{bm, amsmath, amssymb, amsthm,color}

\def\<{\langle}
\def\>{\rangle}
\def\eps{\varepsilon}
\def\RR{\mathbb{R}}

\def\tr{\operatorname{Tr\,}}
\def\id{\operatorname{id}}
\def\Div{\operatorname{div}}
\def\I{\mathfrak{T}}

\def\tA{{\tilde A}}
\def\tT{{\tilde T^\sharp}}
\def\T{T^\sharp}
\def\tH{{\tilde H}}

\def\Hn{{H}/{n}}
\def\tHp{{{\tilde H}}/{p}}

\def\Ric{\operatorname{Ric}}
\def\vol{\operatorname{vol}}

\def\eq{\hspace*{-2.mm}&=&\hspace*{-2.mm}}
\def\plus{\hspace*{-1.5mm}&+&\hspace*{-1.5mm}}

\def\dt{\partial_t}

\newcommand\mD{{\cal D}}
\newcommand\Ht{H}
\newcommand\Hb{{\tilde H}}

\newcommand\Tb{T}

\newtheorem{theorem}{Theorem}[section]
\newtheorem{corollary}{Corollary}[section]
\newtheorem{lemma}{Lemma}[section] 
\newtheorem{proposition}{Proposition}[section]
\newtheorem{example}{Example}[section]
\newtheorem{definition}{Definition}[section]
\newtheorem{remark}{Remark}[section]

\title{The Einstein-Hilbert type action \\ on metric-affine almost-product manifolds}

\author{
       Vladimir Rovenski\footnote{Mathematical Department, University of Haifa, Mount Carmel, Haifa, 31905, Israel
       \newline e-mail: {\tt vrovenski@univ.haifa.ac.il}}\ \
        \ and \
       Tomasz Zawadzki
       \footnote{Uniwersytet \L\'{o}dzki, ul. Banacha 22, 90-238 \L\'{o}d\'{z}, Poland
        \newline e-mail: {\tt tomasz.zawadzki@wmii.uni.lodz.pl}}
      }

\begin{document}

\date{}

\maketitle

\begin{abstract}
We continue our study of the mixed Einstein-Hilbert action as a functional of a pseudo-Riemannian metric and a linear connection.
Its geometrical part is the total mixed scalar curvature on a smooth manifold endowed with a~distribution or a foliation.
We develop variational formulas for quantities of extrinsic geometry of a distribution on a metric-affine space
and use them to derive Euler-Lagrange equations (which in the case of space-time are analogous to those in Einstein-Cartan theory) and to characterize critical points of this action on vacuum space-time.
Together with arbitrary variations of metric and connection, we consider also variations that
partially preserve the metric, e.g., along the distribution, and also variations among distinguished classes of connections
(e.g., statistical and metric compatible, and this is expressed in terms of restrictions on contorsion tensor).
One of Euler-Lagrange equations of the mixed Einstein-Hilbert action is an analog of the Cartan spin connection equation,
and the other can be presented in the form similar to the Einstein equation, with Ricci curvature replaced by the new Ricci type tensor. This tensor generally has a complicated form, but is given in the paper explicitly for variations among semi-symmetric~connections.
\end{abstract}




\section{Introduction}

We study the mixed Einstein-Hilbert action as a functional of two variables:
a pseudo-Riemannian metric and a linear connection.
Its geometrical part is the total mixed scalar curvature on a smooth manifold endowed with a~distribution or a foliation.
Our goals are to obtain the Euler-Lagrange equations of the action, present them in the classical form of Einstein equations and find their solutions for the vacuum case.

\smallskip

\textbf{1.1. State-of-the-art}.
The {Metric-Affine~Geo\-metry} (founded by E.\,Cartan) generalizes pseudo-Riemannian Geometry:
it uses a linear connection $\bar\nabla$ with torsion, instead of the Levi-Civita connection $\nabla$ of metric $g=\<\cdot,\,\cdot\>$
on a manifold $M$,
e.g.,~\cite{mikes}, and appears in such context as almost Hermitian and Finsler manifolds and theory of gravity.
To~describe geometric properties of $\bar\nabla$, we use the~difference $\I=\bar\nabla-\nabla$ (called the \textit{contorsion tensor}) and also auxiliary (1,2)-tensors $\I^*$ and $\I^\wedge$~defined by
\[
 \<\I^*_X Y,Z\> = \<\I_X Z, Y\>,\quad \I^\wedge_X Y = \I_Y X,
 \quad X,Y,Z\in\mathfrak{X}_M .
\]
The~following distinguished classes of metric-affine manifolds $(M,g,\bar\nabla)$ are considered important.

$\bullet$~\textit{Riemann-Cartan manifolds}, where the $\bar\nabla$-parallel transport along the curves preserves the metric, i.e., $\bar\nabla g =0$, e.g., \cite{gps,rze-27b},
This condition is equivalent to $\I^*=-\I$ and $\bar \nabla$ is then called a metric compatible (or: metric) connection. 
e.g.,~\cite{cb19}, where the torsion tensor
is involved in the {Cartan spin connection equation}, see \eqref{Eq-EC-nabla}.
More specific types of metric connections
(e.g., the \textit{semi-symmetric connections} \cite{FR, Yano} and \textit{adapted metric connections} \cite{bf})
also find applications in geometry and theoretical physics.

$\bullet$~\textit{Statistical manifolds}, where the tensor $\bar\nabla g$ is symmetric in all its entries and connection $\bar\nabla$ is  torsion-free, e.g., \cite{ch12,op2016,pss-2020}.
These conditions are equivalent to $\I^\wedge=\I$ and $\I^* = \I$.
The theory of affine hypersurfaces in $\RR^{n+1}$ is a natural source of such manifolds; they also find applications in theory of probability and statistics.

\smallskip

The above classes of connections admit a natural definition of the sectional curvature: in case of metric connections by the same formula as for the Levi-Civita connection, and for statistical connections by the analogue introduced in \cite{op2016}.
For the curvature tensor $\bar R_{X,Y}=[\bar\nabla_Y,\bar\nabla_X]+\bar\nabla_{[X,Y]}$ of an affine connection $\bar\nabla$,
we have
\begin{equation}\label{E-RC-2}
 \bar R_{X,Y} -R_{X,Y} = (\nabla_Y\,\I)_X -(\nabla_X\,\I)_Y +[\I_Y,\,\I_X],
\end{equation}
where $R_{X,Y}=[\nabla_Y,\nabla_X]+\nabla_{[X,Y]}$ is the~Riemann curvature tensor of $\nabla$.
Similarly as in Riemannian geometry, one can also consider the scalar curvature $\overline{\rm S}$ of $\bar R$.

Many notable examples of pseudo-Riemannian metrics come (as critical points) from variational problems, a particularly famous of which is the \textit{Einstein-Hilbert action}, e.g., \cite{besse}. Its Einstein-Cartan genera\-lization
in the framework of metric-affine geometry, given (on a smooth manifold $M$) by
\begin{equation}\label{Eq-EH}
 \bar J: (g,\I) \to \int_M \big\{\frac1{2\mathfrak{a}}\,(\overline{\rm S}-2{\Lambda})+{\cal L}\big\}\,{\rm d}{\rm vol}_g ,
\end{equation}
extends the original formulation of general relativity and provides interesting examples of metrics as well as connections.
 Here, $\Lambda$ is a constant (the ``cosmological constant"), ${\cal L}$ is Lagrangian describing the matter contents,
and
$\mathfrak{a}=8\pi G/c^{4}$ -- the coupling constant involving the gravitational constant $G$ and the speed of light $c$.
To deal also with non-compact manifolds, 
it is assumed that the 
integral above is taken over $M$ if it converges;
otherwise, one integrates over arbitrarily large, relatively compact domain $\Omega\subset M$, which also contains supports
of variations of $g$ and $\I$. The Euler-Lagrange equation for \eqref{Eq-EH} when $g$ varies is
\begin{subequations}
\begin{equation}\label{Eq-EC-R}
 \overline\Ric -\,(1/2)\,\overline{\rm S}\cdot g +\,\Lambda\,g  = \mathfrak{a}\,\Xi
\end{equation}
(called the Einstein equation) with the non-symmetric Ricci curvature $\overline\Ric$
and the asymmetric ener\-gy-momentum tensor $\Xi$ (generalizing the stress tensor of Newtonian physics), given in a coordinates by
$\Xi_{\mu\nu}=-2\,{\partial{\cal L}}/{\partial g^{\mu\nu}} +g_{\mu\nu}{\cal L}$.
The Euler-Lagrange equation for \eqref{Eq-EH} when $\I$ varies is an algebraic constraint with the torsion tensor ${\cal S}$
of $\bar\nabla$ and the spin tensor $s_{\mu\nu}^c=2\,{\partial{\cal L}}/{\partial \I_{\mu\nu}^c}$
(used to describe the intrinsic angular momentum of particles in spacetime, e.g.,~\cite{tr}):
\begin{equation}\label{Eq-EC-nabla}
 {\cal S}(X,Y) +\tr({\cal S}(\cdot,Y) - {\cal S}(X,\cdot)) = \mathfrak{a}\,s(X,Y),\quad X,Y\in\mathfrak{X}_M.
\end{equation}
\end{subequations}
Since ${\cal S}(X,Y)=\I_XY-\I_YX$, \eqref{Eq-EC-nabla} can be rewritten using the contorsion tensor.
The solution of (\ref{Eq-EC-R},b) is a pair $(g,{\mathfrak T})$, satisfying this system, where the pair of tensors $(\Xi,s)$ (describing a specified type of matter) is given. In vacuum space-time, Einstein and Einstein-Cartan theories coincide.
The~classification of solutions of (\ref{Eq-EC-R},b) is a deep and largely unsolved problem~\cite{besse}.

\smallskip

\textbf{1.2. Objectives}.
On a manifold equipped with an additional structure
(e.g., almost product, complex or contact), one can consider an analogue of \eqref{Eq-EH} adjusted to that structure. In pseudo-Riemannian geometry, it may mean restricting $g$ to a certain class of metrics (e.g., conformal to a given one, in the Yamabe problem \cite{besse}) or even constructing a new, related action (e.g., the Futaki functional on a Kahler manifold \cite{besse}, or several actions on contact manifolds \cite{Blairsurvey}), to cite only few examples.
The~latter approach was taken in authors' previous papers, where the scalar curvature in the Einstein-Hilbert action on a pseudo-Riemannian manifold was replaced by the mixed scalar curvature of a given distribution or a foliation.

In this paper, a similar change in \eqref{Eq-EH} will be considered on
a connected smooth $(n+p)$-dimensional manifold $M$ endowed with an affine connection
and a smooth $n$-dimensional distribution $\widetilde\mD$ (a subbundle of the tangent bundle $TM$).
Distributions and foliations (that can be viewed as integrable distributions) on 
manifolds appear in various situations, e.g., \cite{bf,rov-m}. 
When a pseudo-Riemannian metric $g$ on $M$ is non-degenerate along $\widetilde\mD$, it defines the orthogonal $p$-dimensional distribution $\mD$ such that both distributions span the tangent bundle: $TM=\widetilde\mD\oplus\mD$
and define a Riemannian almost-product structure on $(M,g)$, e.g., \cite{g1967}.
From a mathematical point of view, a \textit{space-time} of general relativity is a 
$(n+1)$-dimensional 
time-oriented (i.e., with a given timelike vector field) Lorentzian manifold, see~\cite{bee}.
A~space-time admits a global \textit{time function} (i.e., increasing function along each future directed nonspacelike curve) if and only if it is stable causal; in particular, a {glo\-bally hyperbolic} spacetime is naturally
endowed with a codimension-one foliation (the level hypersurfaces of a given time-function), see \cite{bs,fs}.

The~\textit{mixed Einstein-Hilbert action} on $(M,\widetilde\mD)$,
\begin{equation}\label{Eq-Smix}
 \bar J_{{\mD}}: (g,{\mathfrak T})\mapsto\!\int_{M} \Big\{\frac1{2\mathfrak{a}}\,
 ( \overline{\rm S}_{{\rm mix}} -2\,{\Lambda}) +{\cal L}\Big\}\,{\rm d}\vol_g,
\end{equation}
is an analog of \eqref{Eq-EH}, where $\overline{\rm S}\,$ is replaced by the mixed scalar curvature $\overline{\rm S}_{\,\rm mix}$, see~\eqref{eq-wal2}, for the affine connection $\bar\nabla=\nabla+\I$.
The physical meaning of \eqref{Eq-Smix} is discussed in~\cite{bdrs}  
for the case of $\I=0$. 

In view of the formula $\overline{\rm S}=\overline{\rm S}_{\rm mix}+\overline{\rm S}^{\,\top}\!+\overline{\rm S}^{\,\bot}$,
where $\overline{\rm S}^{\,\top}$ and $\overline{\rm S}^{\,\bot}$ are the scalar curvatures along the distributions $\widetilde\mD$ and $\mD$,
one can combine the actions \eqref{Eq-EH} and \eqref{Eq-Smix} to obtain the new \textit{perturbed Einstein-Hilbert action} on $(M,\widetilde\mD)$:  
$\bar J_{\eps}: (g,{\mathfrak T})\mapsto\!\int_{M} \big\{\frac1{2\mathfrak{a}}\,(\overline{\rm S}+\eps\,\overline{\rm S}_{\rm mix}
 -2\,{\Lambda}) +{\cal L}\big\}\,{\rm d}\vol_g$ with $\eps\in\RR$,
whose critical points may describe geometry of the space-time in an extended theory of gravity.

The mixed scalar curvature (being an averaged mixed sectional curvature) is one of the simplest curvature invariants of a pseudo-Riemannian almost-product structure. If~a distribution is spanned by a unit vector field $N$, i.e., $\<N,N\>=\eps_N\in\{-1,1\}$,
then $\overline{\rm S}_{\rm mix} = \eps_N\overline\Ric_{N,N}$, where $\overline\Ric_{N,N}$ is the Ricci curvature in the $N$-direction.
If $\dim M=2$ and $\dim\mD=1$, then obviously $\overline{\rm S}_{\rm mix}=\overline{\rm S}$.
If~${\mathfrak T}=0$ then $\overline{\rm S}_{{\rm mix}}$ reduces to the mixed scalar curvature ${\rm S}_{{\rm mix}}$ of $\nabla$,
see~\eqref{eq-wal2-0}, which can be defined as a sum of sectional curvatures of planes that non-trivially intersect with both of the distributions.
Investigation of ${\rm S}_{{\rm mix}}$ led to multiple results regarding the existence of foliations and submersions with interesting geometry, e.g., integral formulas and splitting results, curvature prescribing and variational problems, see survey \cite{rov-5}.
The~trace of the partial Ricci curvature (rank 2) tensor $r_{\cal D}$ is ${\rm S}_{\rm mix}$, see Section~\ref{sec:prel}.
The understanding of the mixed curvature, especially, $r_{\cal D}$ and ${\rm S}_{\rm mix}$, is a fundamental problem of extrinsic geometry of foliations, see~\cite{rov-m}.

Varying \eqref{Eq-Smix} with fixed $\I=0$, as a functional of $g$ only, 
we obtain the Euler-Lagrange equations in the form similar to \eqref{Eq-EC-R},
see \cite{bdrs} for space-times, and for $\widetilde{\mD}$ of any~dimension, see \cite{rz-1,rz-2}, i.e., 
\begin{equation}\label{E-gravity}
\Ric_{\,{\mD}} -\,(1/2)\,{\rm S}_{\,\mD}\cdot g +\,\Lambda\,g = \mathfrak{a}\,\Xi,
\end{equation}
where the Ricci and scalar curvature are replaced by the \textit{mixed Ricci curvature} $\Ric_{\,{\mD}}$, see \eqref{E-main-0ij},
and its~trace ${\rm S}_{\mD}$.
In~\cite{RZconnection}, we obtained the Euler--Lagrange equations for \eqref{Eq-Smix} with fixed $g$ and variable $\I$,
see (\ref{ELconnection1}-h),
and examined critical contorsion tensors (and corresponding connections) in general and in distinguished classes of (1,2)-tensors.
We have shown that $\I$ is critical for \eqref{Eq-Smix}
with fixed $g$ if and only if $\I$ obeys certain system of algebraic equations,
however, unlike \eqref{Eq-EC-nabla}, these equations heavily involve also the pseudo-Riemannian geometry of the distributions.
 In the article we generalize these results, considering variations of \eqref{Eq-Smix} with respect to both $g$ and $\I$, at their arbitrary values.
As we are less inclined to discuss particular physical theories, we basically confine ourselves to studying
the total mixed scalar curvature -- the geometric part of the mixed Einstein-Hilbert action,
i.e., we set $\Lambda={\cal L}=0$ in \eqref{Eq-Smix}, which in physics correspond to vacuum space-time and no ``cosmological~constant":
\begin{equation}\label{actiongSmix}
\bar J_{{\rm mix}} : (g, \I) \mapsto \int_M \overline{\rm S}_{\rm mix}\,{\rm d}\vol_g.
\end{equation}
Considering variations of the metric that preserve the volume of the manifold, we can also obtain the Euler-Lagrange equations for \eqref{actiongSmix}, that coincide with those for \eqref{Eq-Smix} with ${\cal L}=0$ and $\Lambda \neq 0$.

The terms of $\bar{\rm S}_{\,\rm mix}$ without covariant derivatives of $\I$ make up the \textit{mixed scalar $\I$-curvature},
see Section~\ref{sec:prel}, which we find interesting on its own. In particular, ${\rm S}_{\,\I}$ can be viewed as the Riemannian mixed scalar curvature of a distribution with all sectional curvatures of planes replaced by their $\I$-curvatures (see \cite{op2016}), and for statistical connections we have $\bar{\rm S}_{\,\rm mix} = {\rm S}_{\,\rm mix} + {\rm S}_{\,\I}$.
Thus, we also study (in Section~\ref{sec: 2-1}) the following, closely related to \eqref{actiongSmix}, action on $(M,\widetilde{\mD})$:
\begin{equation}\label{actiongISmix}
I: (g, \I) \mapsto \int_M {{\rm S}}_{\,\I}\,{\rm d} \vol_g .
\end{equation}
 For each of the examined actions \eqref{actiongSmix} and \eqref{actiongISmix}, we obtain the Euler-Lagrange equations and formulate results about existence and examples of their solutions, that we describe in more detail further below.
In~particular, from \cite{RZconnection} we know that if 
$\I$ is critical for the action \eqref{actiongSmix}, then $\mD$ and $\widetilde\mD$
are totally umbilical with respect to $\nabla$ -- and to express this together with other conditions, a pair of equations like
(\ref{Eq-EC-R},b) is not sufficient.
Due to this fact, only 
in the special case of semi-symmetric connections we present the 
Euler-Lagrange equation in the form, which directly generalizes \eqref{E-gravity}:
\begin{equation}\label{E-gravity-gen}
\overline\Ric_{\,{\mD}} -\,(1/2)\,\overline{\rm S}_{\,\mD}\cdot g +\Lambda\,g = \mathfrak{a}\,\Xi
\end{equation}
and a separate condition (\ref{ELconnection1}-h), similar to \eqref{Eq-EC-nabla}, for the vector field parameterizing this type of connection.
In~the paper we study solutions of \eqref{E-gravity-gen} and (\ref{ELconnection1}-h)
for the vacuum case.

\smallskip

\textbf{1.3. Structure of the paper}.
The article has the Introduction and three other Sections.

Section~\ref{sec:prel} contains background definitions and necessary results from \cite{bdr,rz-1,rz-2,RZconnection},
among them the notions of the mixed scalar curvature and the mixed and the partial Ricci tensors are~central.

Section~\ref{sec:main} contains the main results, described in detail below.

Section~\ref{sec:aux} contains auxiliary lemmas with necessary, but lengthy computations, and the References include 30 items.

In~Section~\ref{sec:main}, we derive the Euler--Lagrange equations for \eqref{actiongSmix} and \eqref{actiongISmix}
and find some of their solutions
-- critical pairs $(g,\I)$ 
for different kinds of variations of metric and connection. Apart from varying among all metrics that are non-degenerate on $\widetilde\mD$,
we also restrict to the case when metric remains fixed on the distribution, and the complementary case when metric varies only on the distribution -- preserving its orthogonal complement and the metric on it.
This approach (first, applied in \cite{RWa-1} for codimension one foliations) can be used to finding an optimal extension of a metric given only on the distribution -- which is the problem of the relationship between sub-Riemannian and Riemannian geometries. Moreover, in analogy to the Einstein-Hilbert action, all variations are considered in two kinds: with and without
preserving the volume of the manifold, see~\cite{besse}. In~addition, together with arbitrary variations of connection, we consider variations among such distinguished classes as statistical and metric connections, and this is expressed in terms of constraints on $\I$.

Section~\ref{sec:main} is divided into four subsections, according to additional conditions we impose on connections (e.g., metric, adapted and statistical) or actions we consider (defined by the mixed scalar curvature $\overline{\rm S}_{\,\rm mix}$ and
the algebraic curvature-type invariant of a contorsion tensor ${\rm S}_{\,\I}$).

In Section \ref{sec: 2-1}, we vary functional \eqref{actiongISmix} with respect to metric  $g$.
Compared to its variation with fixed $g$, which was considered in \cite{RZconnection}, we obtain additional conditions for general and metric connections.
On the other hand, a metric-affine doubly twisted product is critical for \eqref{actiongISmix} if and only if it is critical for the action with fixed~$g$. Similarly, restricting \eqref{actiongISmix} to pairs of metrics and statistical connections also does not give any new Euler-Lagrange equations than those obtained in \cite{RZconnection}.

In Section~\ref{sec: 2-2},
for arbitrary variations of $(g,\I)$
we show that statistical connections critical for \eqref{actiongSmix} on a closed $M$ are exactly those that are critical for \eqref{actiongSmix} with fixed $g$, and for $n+p>2$ these exist only on metric products.
On the other hand, for every $g$ critical for \eqref{actiongSmix} with fixed $\I=0$, there exist statistical connections,
satisfying algebraic conditions {\rm(\ref{ELSmixIstat1},b}),
such that $(g,\I)$ is critical for \eqref{actiongSmix} restricted to all metrics, but only statistical connections. Note that \eqref{ELSmixIstat2} is equivalent to $\I$ acting invariantly on each distribution, i.e., with only components $\I:\widetilde{\mD}\times\widetilde{\mD}\rightarrow\widetilde{\mD}$ and $\I:{\mD}\times{\mD}\rightarrow{\mD}$.
Equations {\rm(\ref{ELSmixIstat1},b}) imply also that the traces $\tr^\top\I$ and $\tr^\perp\I$ vanish, and these are the only restrictions for $\I$ critical among statistical connections.

In Section \ref{sec: 2-3} we show that for $n,p>1$ the critical value of \eqref{actiongSmix} attained by $(g, \I)$, where $\I$ corresponds to a metric connection, depends only on $g$ and is non-negative on a Riemannian manifold. In other words, pseudo-Riemannian geometry determines the  mixed scalar curvature of any critical metric connection.
For general metric connections, we consider only adapted variations of the metric (see Definition \ref{defintionvariationsofg}) due to complexity of the
variational formulas. Compared to \eqref{actiongSmix} with fixed $g$, we get a new condition \eqref{ELmetric}, involving the symmetric part of $\I|_{{\mD}\times{\mD}}$ and of $\I|_{\widetilde{\mD}\times\widetilde{\mD}}$ in the dual equation.
Under some assumptions, trace of \eqref{ELmetric} depends only on the pseudo-Riemannian geometry of $(M,g,\widetilde\mD)$ and thus gives a necessary condition for the metric to admit a critical point of \eqref{actiongSmix} in a large class of connections (e.g., adapted), or for integrable distributions $\mD$.
On the other hand, in the case of adapted variations, antisymmetric parts of $(\I|_{{\mD}\times{\mD}})^\perp$ and $(\I|_{\widetilde{\mD}\times\widetilde{\mD}})^\top$ remain free parameters of any critical metric connection, as they do not appear in Euler-Lagrange equations (note that these components define part of the critical connection's torsion).
Thus, for a given metric $g$ that admits critical points of \eqref{actiongSmix}, one can expect to have multiple critical metric connections, and examples in Section~\ref{sec: 2-3} confirm~that.

Section~\ref{sec:2-4} deals with a semi-symmetric connection (parameterized by a vector field), as a simple case of a metric connection. Although such connections are critical for the action \eqref{actiongSmix} and arbitrary variations of connections only on metric-affine products, when we restrict variations of the mixed scalar curvature to semi-symmetric connections, we obtain meaningful Euler-Lagrange equations (in Theorem~\ref{propUconnectionEL}), which allow us to explicitly present the mixed Ricci tensor -- analogous to the Ricci tensor in the Einstein equation.

\section{Preliminaries}
\label{sec:prel}

Here, we recall definitions of some functions and tensors, used also in \cite{bdr,rz-1,rz-2,RZconnection,wa1}, and introduce several new notions related to geometry of $(M,g,\bar\nabla)$  endowed with a non-degenerate distribution.

\smallskip

\textbf{2.1. The mixed scalar curvature}.
Let ${\rm Sym}^2(M)$ be the space of symmetric $(0,2)$-tensors tangent to a smooth connected manifold~$M$.
A~\textit{pseudo-Riemannian metric} $g=\<\cdot,\cdot\>$ of index $q$ on $M$ is an element $g\in{\rm Sym}^2(M)$
such that each $g_x\ (x\in M)$ is a {non-degenerate bilinear form of index} $q$ on the tangent space $T_xM$.
For~$q=0$ (i.e., $g_x$ is positive definite) $g$ is a Riemannian metric and for $q=1$ it is called a Lorentz metric.
Let~${\rm Riem}(M)\subset{\rm Sym}^2(M)$ be the subspace of pseudo-Riemannian metrics of a given signature.

 A smooth subbundle $\widetilde{\mD}\subset TM$ (that is a regular distribution) is \textit{non-degenerate},
if $g_x$ is non-degenerate on $\widetilde{\mD}_x\subset T_x M$ for $x\in M$;
in this case, the orthogonal complement ${\mD}$ of~$\widetilde{\mD}$ is also non-degenerate,
and we have $\widetilde{\mD}_x\cap\,{\mD}_x=0$, $\widetilde{\mD}_x\oplus\,{\mD}_x=T_xM$ for all $x \in M$.
Let~$\mathfrak{X}_M, \mathfrak{X}^\bot,\mathfrak{X}^\top$ be the modules over $C^\infty(M)$ of sections (vector fields)
of $TM,{\mD}$ and $\widetilde{\mD}$, respectively. 

Let ${\rm Riem}(M,\widetilde{\mD},{\mD})\subset {\rm Riem}(M)$ be the subspace of pseudo-Riemannian metrics making
$\widetilde{\mD}$ and ${\mD}$ (of ranks $\dim\widetilde{\mD}=n\ge1$ and $\dim{\mD}=p\ge1$) orthogonal and non-degenerate.
Given $g\in{\rm Riem}(M,\widetilde{\mD},{\mD})$, a local adapted orthonormal frame $\{E_a,\,{\cal E}_{i}\}$, where $\{E_a\}\subset\widetilde{\mD}$ and $\eps_i=\<{\cal E}_{i},{\cal E}_{i}\>\in\{-1,1\}$, $\eps_a=\<E_a,E_a\>\in\{-1,1\}$, always exists on $M$.
The~following convention is adopted for the range of~indices:
\begin{equation*}
 a,b,c\ldots{\in}\{1\ldots n\},\quad i,j,k\ldots{\in}\{1\ldots p\}.
\end{equation*}
All the  quantities defined below  with the use of an adapted orthonormal frame do not depend on the choice of this frame.
We~have $X=\widetilde{X} + X^\perp$, where $\widetilde{X} \equiv X^\top$ is the $\widetilde{\mD}$-component of $X\in\mathfrak{X}_M$ (respectively, $X^\perp$ is the  ${\mD}$-component of $X$) with respect to $g$. Set~$\id^\top(X)=X^\top$ and $\,\id^\bot(X)=X^\bot$.

\begin{definition}\rm
 The function on $(M,g,\bar\nabla)$ endowed with a non-degenerate distribution $\widetilde{\mD}$,
\begin{equation}\label{eq-wal2}
 \bar{\rm S}_{\,\rm mix} = \frac{1}{2}\sum\nolimits_{a,i} \eps_a \eps_i \big(\<{\bar R}_{\,E_a, {\cal E}_i} E_a, {\cal E}_i\>
 +\<{\bar R}_{\,{\cal E}_i, E_a}\, {\cal E}_i, E_a\> \big),
\end{equation}
is called the \textit{mixed scalar curvature with respect to connection} $\bar\nabla$.
In particular case of the Levi-Civita connection $\nabla$, the function on $(M,g)$,
\begin{equation}\label{eq-wal2-0}
 {\rm S}_{\rm mix} = \tr_{g}{r}_{\mD} =\sum\nolimits_{\,a,i}\eps_a \eps_i\,\<R_{\,E_a, {\cal E}_{i}}\,E_a,\, {\cal E}_{i}\>
\end{equation}
is called the \textit{mixed scalar curvature} (with respect to $\nabla$).
 The symmetric $(0,2)$-tensor
\begin{equation}\label{E-Rictop2}
 {r}_{{\mD}}(X,Y) = \sum\nolimits_{a} \eps_a\, \<R_{\,E_a,\,X^\perp}\,E_a, \, Y^\perp\>, \quad X,Y\in \mathfrak{X}_M,
\end{equation}
is called the \textit{partial Ricci tensor} related to $\widetilde\mD$.
\end{definition}

Remark that on $(M,\widetilde\mD)$, the ${\rm S}_{\rm mix}$ and $g$-orthogonal complement to $\widetilde{\mD}$ are determined by the choice of metric~$g$. In particular,
if $\dim\widetilde\mD=1$ then $r_\mD=\eps_N\,R_N$,
where $R_N=R_{N,\,^\centerdot}\,N$ is the Jacobi operator,
and if $\dim\mD=1$ then $r_\mD=\Ric_{N,N}g^\bot$,
where the symmetric (0,2)-tensor $g^\perp$ is defined by
 $g^\perp (X,Y) = \< X^\perp, Y^\perp\>$
 for $X,Y \in \mathfrak{X}_M$.

We use the following convention for components of various $(1,1)$-tensors in an adapted orthonormal frame $\{E_a , {\cal E}_i \}$:
$\I_a = \I_{E_a} ,\ \I_i = \I_{{\cal E}_i}$,
etc.
Following the notion of $\I$-{sectional curvature} of a symmetric $(1,2)$-tensor $\I$ on a vector space endowed with a scalar product and a cubic form, see~\cite{op2016}, we define the {mixed scalar $\I$-curvature} by \eqref{E-SK},
as a sum of $\I$-sectional curvatures of planes that non-trivially intersect with both of the distributions,
\begin{equation}\label{E-SK}
 {\rm S}_{\,\I} = \sum\nolimits_{\,a,i} \eps_a \eps_i ( \<[\I_i,\, \I_a] E_a, {\cal E}_i\> +\<[\I_a,\, \I_i]\, {\cal E}_i, E_a\> ).
\end{equation}
The definitions \eqref{E-SK}, \eqref{eq-wal2}--\eqref{eq-wal2-0} do not depend on the choice of an adapted local orthonormal frame.
Thus, we can consider $\bar{\rm S}_{\,\rm mix}$ and ${\rm S}_{\I}$
on $(M,\widetilde\mD)$ as functions of $g$ and $\I$.
 If ${\cal T}$ is either symmetric or anti-symmetric then \eqref{E-SK} reads as
 ${\rm S}_{\,\I} = \sum\nolimits_{a,i}\eps_a\eps_i\,\<[\I_i,\,\I_a]\, E_a, {\cal E}_i\>$.
As was mentioned in the Introduction,
the~\textit{mixed scalar $\I$-curvature}
(for the contorsion tensor $\I$) is a part of $\bar{\rm S}_{\,\rm mix}$, in~fact we have \cite[Eq.~(6)]{RZconnection}:
\begin{equation}\label{barSmix}
 \bar{\rm S}_{\,\rm mix} = {\rm S}_{\,\rm mix} + {\rm S}_{\,\I} + {\bar Q}/2,
\end{equation}
where ${\bar Q}$ consists of all terms with covariant derivatives of $\I$,
\begin{equation*}
 {\bar Q} = \sum\nolimits_{a,i}\eps_a\eps_i\big(\<(\nabla_i \I)_a E_a, {\cal E}_i\> -\<(\nabla_a \I)_i E_a, {\cal E}_i\>
 +\<(\nabla_a \I)_i\, {\cal E}_i, E_a\> -\<(\nabla_i \I)_a\, {\cal E}_i, E_a\>\big).
\end{equation*}
The formulas for the mixed scalar curvature in the next two lemmas are essential in our calculations.
The~lemmas use tensors defined in \cite{rz-1}, which are briefly recalled below. 
\begin{proposition}
The following presentation of the partial Ricci tensor in \eqref{E-Rictop2} is valid, see \cite{bdr,rz-1}:
\begin{equation}\label{E-genRicN}
 r_{{\mD}}=\Div\tilde h +\<\tilde h,\,\tilde H\>-\widetilde{\cal A}^\flat-\widetilde{\cal T}^\flat-\Psi+{\rm Def}_{\mD}\,H.
\end{equation}
Tracing \eqref{E-genRicN}, we have, see {\rm \cite{wa1}},
 \begin{equation}\label{eq-ran-ex}
 {\rm S}_{\rm mix} = \<H,H\> +\<{\tilde H}, {\tilde H}\> -\<h,h\> -\< {\tilde h},{\tilde h}\>
 +\<T,T\> +\<\tilde T,\tilde T\> +\Div(H+\tilde H)\,.
\end{equation}
For totally umbilical distributions, i.e.,
$h=\frac1nH\,{g^\top}$ and $\tilde h=\frac1p\,\tilde H\,{g^\bot}$, \eqref{eq-ran-ex} reads as
\begin{equation}\label{E-PW-Smix-umb}
 {\rm S}_{\rm mix} = \frac{n-1}{n} \<H,H\> +\frac{p-1}{p} \<\tH,\tH\> +\<T, T\> +\< {\tilde T}, {\tilde T} \> +\Div(H +\tH),
\end{equation}
\end{proposition}

Denote by $\<B,C\>_{|V}$ the inner product of tensors $B,C$ restricted to 
$V=(\tilde{\mD}\times{\mD})\cup({\mD}\times\tilde{\mD})$.

\begin{proposition}[see \cite{r-affine}]\label{L-QQ-first}
We have using \eqref{E-RC-2},
\begin{equation}\label{E-Q1Q2-gen}
 2\,(\bar{\rm S}_{\,\rm mix} -{\rm S}_{\,\rm mix})
 = \Div\big( (\tr^\top(\I -\I^*))^\bot +(\tr^\bot(\I -\I^*))^\top \big) - Q,
\end{equation}
where
\begin{eqnarray}\label{E-defQ}
\nonumber
 && Q = - \<\tr^\bot\I,\, \tr^\top\I^*\> -\< \tr^\top\I,\,\tr^\bot\I^*\> +\<\I^*, \I^\wedge\>_{\,|\,V}\\
 && \hskip-5mm -\,\<\tr^\top(\I- \I^*) -\tr^\bot(\I -\I^*),\, H -{\tilde H}\>
 -\< \I -\I^* +\I^\wedge - \I^{* \wedge} ,\ {\tilde A}-{\tilde T}^\sharp + A-T^\sharp\> .
\end{eqnarray}
and the partial traces of $\,\I$ (similarly, for $\I^*$, etc.) are given~by
\begin{equation}\label{E-defTT}
 \tr^\top\I = \sum\nolimits_a\eps_a \I_a E_a,\quad
 \tr^\bot\I = \sum\nolimits_i\eps_i\, \I_i \,{\cal E}_i.
\end{equation}
\end{proposition}

The tensors used in the above lemmas (and other ones) are defined below
for one of the distributions (say, ${\mD}$; similar tensors for $\widetilde{\mD}$ are denoted using $^\top$ or $\ \widetilde{}\ $ notation).

The integrability tensor and the second fundamental form
$T, h:\widetilde{\mD}\times \widetilde{\mD}\to{\mD}$ of $\widetilde{\mD}$ are given by
\begin{equation*}
 T(X,Y)=(1/2)\,[X,\,Y]^\perp,\quad h(X,Y) = (1/2)\,(\nabla_X Y+\nabla_Y X)^\perp, \quad X, Y \in \mathfrak{X}^\top.
\end{equation*}
The mean curvature vector field of $\widetilde{\mD}$ is given by
 $H=\tr_{g} h=\sum\nolimits_a\eps_a h(E_a,E_a)$.
We call $\widetilde{\mD}$ {totally umbilical}, {minimal}, or {totally geodesic}, if
$h=\frac1nH\,{g^\top},\ H =0$, or $h=0$, respectively.

The ``musical" isomorphisms $\sharp$ and $\flat$ will be used for rank one and symmetric rank 2 tensors.
For~example, if $\omega \in\Lambda^1(M)$ is a 1-form and $X,Y\in {\mathfrak X}_M$ then
$\omega(Y)=\<\omega^\sharp,Y\>$ and $X^\flat(Y) =\<X,Y\>$.
For arbitrary (0,2)-tensors $A$ and $B$ we also have $\<A, B\> =\tr_g(A^\sharp B^\sharp)=\<A^\sharp, B^\sharp\>$.

The Weingarten operator $A_Z$ of $\widetilde{\mD}$ with $Z\in\mathfrak{X}^\bot\mathfrak{\mathfrak{}}$, and the operator $T^\sharp_Z$ are defined~by
\[
 \<A_Z(X),Y\>= \,h(X,Y),Z\>,\quad \<T^\sharp_Z(X),Y\>=\<T(X,Y),Z\>, \quad X,Y \in \mathfrak{X}^\top .
\]
The norms of tensors are obtained using
\begin{equation*}
 \<h,h\>=\sum\nolimits_{\,a,b}\eps_a\eps_b\,\<h({E}_a,{E}_b),h({E}_a,{E}_b)\>, \quad
 \<T,T\>=\sum\nolimits_{\,a,b}\eps_a\eps_b\,\<T({E}_a,{E}_b),T({E}_a,{E}_b)\>,\quad {\rm etc}.
\end{equation*}
The \textit{divergence} of a vector field $X\in\mathfrak{X}_M$ is given by
\begin{equation}\label{eq:div}
 (\Div X)\,{\rm d}\vol_g = {\cal L}_{X}({\rm d}\vol_g),
\end{equation}
where ${\rm d} \vol_g$ is the volume form of $g$. One may show that
\[
 \Div X=\sum\nolimits_{i}\eps_i\,\<\nabla_{i}\,X, {\cal E}_i\> +\sum\nolimits_{a}\eps_a\,\<\nabla_{a}\,X, {E}_a\>.
\]
The~${\mD}$-\textit{divergence} of a vector field $X\in\mathfrak{X}_M$ is given by
$\Div^\perp X=\sum\nolimits_{i} \eps_i\,\<\nabla_{i}\,X, {\cal E}_i\>$.
Thus,
$\Div X=\tr(\nabla X) = \Div^\perp X +\widetilde{\Div}\,X$.
Observe that for $X\in\mathfrak{X}^\bot$ we have
\begin{equation}\label{E-divN}
 {\Div}^\bot X = \Div X +\<X,\,H\>.
\end{equation}
For a $(1,2)$-tensor $P$ define a $(0,2)$-tensor ${\Div}^\bot P$ by
\[
 ({\Div}^\bot P)(X,Y) = \sum\nolimits_i \eps_i\,\<(\nabla_{i}\,P)(X,Y), {\cal E}_i\>,\quad X,Y \in \mathfrak{X}_M.
\]
For a~${\mD}$-valued $(1,2)$-tensor $P$, similarly to \eqref{E-divN}, we have
\begin{eqnarray*}
 && ({\Div}^\top P)(X,Y) =\sum\nolimits_a \eps_a\,\<(\nabla_{a}\,P)(X,Y), E_a\>
  =-\<P(X,Y), H\>,\\
 && {\Div}^\bot P = \Div P+\<P,\,H\>\,,
\end{eqnarray*}
where $\<P,\,H\>(X,Y)=\<P(X,Y),\,H\>$ is a $(0,2)$-tensor.
For example, $\Div^\perp h = \Div h+\<h,\,H\>$.
For~a~function $f$ on $M$, we use the notation $\nabla^{\perp} f = (\nabla f)^{\perp}$ of the projection of $\nabla f$ onto $\mD$.

The ${\mD}$-\textit{deformation tensor} ${\rm Def}_{\mD}\,Z$ of $Z\in\mathfrak{X}_M$
is the symmetric part of $\nabla Z$ restricted to~${\mD}$,
\begin{equation*}
 2\,{\rm Def}_{\mD}\,Z(X,Y)=\<\nabla_X Z, Y\> +\<\nabla_Y Z, X\>,\quad X,Y\in \mathfrak{X}^\bot.
\end{equation*}
The self-adjoint $(1,1)$-tensors: ${\cal A}$ (the \textit{Casorati type operator})
and ${\cal T}$ and the symmetric $(0,2)$-tensor $\Psi$, see \cite{bdr,rz-1}, are defined by
\begin{eqnarray*}
 && {\cal A}=\sum\nolimits_{\,i}\eps_i A_{i}^2,\quad
 {\cal T}=\sum\nolimits_{\,i}\eps_i(T_{i}^\sharp)^2,\\
 && \Psi(X,Y) = \tr (A_Y A_X+T^\sharp_Y T^\sharp_X), \quad X,Y\in\mathfrak{X}^\bot.
\end{eqnarray*}
 For readers' convenience, we gather below also definitions of all other basic tensors that will be used in further parts of the paper.
We define a self-adjoint $(1,1)$-tensor ${\cal K}$ by the formula
\[
 {\cal K} = \sum\nolimits_i \eps_{\,i} [\T_i, A_i] = \sum\nolimits_{\,i} \eps_i (\T_i A_i - A_i \T_i),
\]
and the $(1,2)$-tensors $\alpha,\theta$ and ${\tilde\delta}_{Z}$ (defined for a given vector field $Z \in \mathfrak{X}_M$) on $(M, \widetilde{\mD}, g)$:
\begin{eqnarray*}
	&& \alpha(X,Y) = \frac{1}{2}\,(A_{X^{\perp}} (Y^{\top}) + A_{Y^{\perp}} (X^{\top})), \quad
	\theta(X,Y) = \frac{1}{2}\,(T^{\sharp}_{X^{\perp}}(Y^{\top}) + T^{\sharp}_{Y^{\perp}}(X^{\top})),\\
	&& {\tilde\delta}_{Z}(X,Y) = \frac{1}{2}\,\big(\<\nabla_{X^{\top}} Z,\, Y^{\perp}\> +\<\nabla_{Y^{\top}} Z,\, X^{\perp}\>\big),
	\quad X,Y \in \mathfrak{X}_M.
\end{eqnarray*}
For any $(1,2)$-tensors $P,Q$ and a $(0,2)$-tensor $S$ 
on $TM$, define the following $(0,2)$-tensor $\Upsilon_{P,Q}$:
\[
\<\Upsilon_{P,Q}, S\> =  \sum\nolimits_{\,\lambda, \mu} \eps_\lambda\, \eps_\mu\,
[S(P(e_{\lambda}, e_{\mu}), Q( e_{\lambda}, e_{\mu})) + S(Q(e_{\lambda}, e_{\mu}), P( e_{\lambda}, e_{\mu}))],
\]
where on the left-hand side we have the inner product of $(0,2)$-tensors induced by $g$,
$\{e_{\lambda}\}$ is a local orthonormal basis of $TM$ and $\eps_\lambda = \<e_{\lambda}, e_{\lambda}\>\in\{-1,1\}$. Note that
\[
 \Upsilon_{P,Q} = \Upsilon_{Q,P},\quad
 \Upsilon_{P, fQ_{1} + Q_{2}} = f\Upsilon_{P,Q_1}+\Upsilon_{P,Q_2}.
\]
Finally, for the contorsion tensor and $X \in TM$ we define ${\I}^\top_X : \tilde{\mD} \rightarrow \tilde{\mD}$ by
\[
{\I}^\top_X Y = (\I_X (Y^\top))^\top , \quad Y \in TM.
\]

\begin{remark} \label{remarkepsilons} \rm
From now on, we shall omit factors $\eps_\mu$ in all expressions with sums over an adapted frame (or its part), effectively identifying symbols $\sum_\mu$ with $\sum_\mu \eps_\mu$ etc.
As we assume in this paper that $g$ is non-degenerate on the distribution 
$\widetilde{\mD}$, 
the presence of factors $\eps_\mu$ in the sums is the only difference in formulas 
with adapted frames for a Riemannian and a pseudo-Riemannian metric $g$. With the definitions 
given in this section, all tensor equations that follow look exactly the same in both these cases. In more complicated formulas we shall also omit summation indices, assuming that every sum is taken over all indices that appear repeatedly after the summation sign, and contains appropriate factors $\eps_\mu$. 
\end{remark}

\textbf{2.2. The mixed Ricci curvature}.
Let $(M,g)$ be a pseudo-Riemannian manifold endowed with a non-degenerate distribution $\widetilde{\mD}$.
We consider smooth $1$-parameter variations $\{g_t\in{\rm Riem}(M):\,|t|<\eps\}$ of the metric $g_0 = g$.
Let the infinitesimal variations, represented by a symmetric $(0,2)$-tensor
\[
 {B}_t\equiv\partial g_t/\partial t,
\]
be supported in a relatively compact domain $\Omega$ in $M$,
i.e., $g_t =g$ outside $\Omega$ for all $|t|<\eps$.
We~call a variation $g_t$ \emph{volume-preserving} if ${\rm Vol}(\Omega,g_t) = {\rm Vol}(\Omega,g)$ for all $t$.
 We~adopt the notations $\partial_t \equiv \partial/\partial t,\ {B}\equiv{\dt g_t}_{\,|\,t=0}=\dot g$,
but we shall also write $B$ instead of $B_t$ to make formulas easier to read, wherever it does not lead to confusion.
Since $B$ is symmetric, then $\<C,\,B\>=\<{\rm Sym}(C),\,B\>$ for any  $(0,2)$-tensor $C$.
We denote by $\otimes$ the product of tensors and use the symmetrization operator to define
the symmetric product of tensors: $B\odot C = {\rm Sym}(B\otimes C)=\frac12\,(B\otimes C+ C\otimes B)$.

\begin{definition} \label{defintionvariationsofg} \rm
A family of metrics $\{g_t\in{\rm Riem}(M):\, |t|<\eps\}$ such that $g_0 =g$ will be called

(i) $g^\pitchfork$-\textit{variation} if
 $g_{t}(X,Y)= g_0(X,Y)$ for all $X,Y\in \mathfrak{X}^\top$ and $|t|<\eps$.

(ii) \textit{adapted variation},
if the $g_t$-orthogonal complement ${\mD}_t$ remain $g_0$-orthogonal to $\widetilde{\mD}$ for all~$t$.

(iii) \textit{${{g^\top}}$-variation}, if it is adapted
and $g_{t}(X,Y)=g_0(X,Y)$ for all $X,Y\in \mathfrak{X}^\bot$ and $|t|<\eps$.

(iv) \textit{${{g^\perp}}$-variation}, if it is adapted $g^\pitchfork$-variation.
\end{definition}

In other words, for $g^\pitchfork$-variations the metric on $\widetilde{\mD}$ is preserved.
For adapted variation we have $g_t\in{\rm Riem}(M,\widetilde{\mD},{\mD})$ for all~$t$.
For ${{g^\top}}$-variations only the metric on $\widetilde{\mD}$ changes,
and for ${g^\bot}$-\textit{variations} only the metric on ${\mD}$ changes, and ${\mD}$ remains to be $g_t$-orthogonal to~$\widetilde{\mD}$.

The symmetric tensor $B_t=\dot g_t$ (of any variation) can be decomposed into the sum of derivatives of
$g^\pitchfork$- and ${g^\top}$-variations, see \cite{rz-2}. Namely,
$B_t=B_t^\pitchfork + {\tilde B}_t$, where
\[
 B^\pitchfork_t =\bigg(\begin{array}{cc}
   B_{ t \,|\,{\cal D}\times{\cal D}} &
   {B}_{ t \,|\,{\cal D}\times\widetilde{\cal D}} \\
   {B}_{ t \,|\,\widetilde{\cal D}\times{\cal D}} & 0
 \end{array}\bigg),
 \quad
 {\tilde B}_t =\bigg(\begin{array}{cc}
   0 & 0 \\ 0 & {B}_{ t \,|\,\widetilde{\cal D}\times\widetilde{\cal D}}
 \end{array}\bigg).
\]
Thus, for $g^\pitchfork$-variations $B(X,Y) =0$ for all $X,Y \in \mathfrak{X}^\top$.
Denote by $^\top$ and $^\perp$ the $g_t$-orthogonal projections of vectors onto $\widetilde{\mD}$
and ${\mD}(t)$ (the $g_t$-orthogonal complement of $\widetilde{\mD}$), respectively.

\begin{proposition}[see \cite{rz-2}]\label{prop-Ei-a}
Let $g_t$ be a $g^\pitchfork$-variation of $g\in{\rm Riem}(M,\widetilde{\mD},{\mD})$.
Let $\{E_a,\,{\cal E}_{i}\}$ be a local $(\widetilde{\mD},\,{\mD})$-adapted and orthonormal for $t=0$ frame, that evolves according to
\begin{equation}\label{E-frameE}
 \dt E_a = 0,\qquad
 \dt {\cal E}_{i}=-(1/2)\, ({B}_t^\sharp({\cal E}_{i}))^{\perp} -({B}_t^\sharp({\cal E}_{i}))^{\top}.
\end{equation}
Then, for all $\,t, $ $\{E_a(t),{\cal E}_{i}(t)\}$ is a $g_t$-orthonormal frame adapted to $(\widetilde{\mD},{\mD}(t))$.
\end{proposition}

For any $g^\pitchfork$-variation of metric
the evolution of ${\mD}(t)$ gives rise to the evolution of both $\widetilde{\mD}$- and ${\mD}(t)$-components of any $X\in\mathfrak{X}_M$:
\begin{equation*}
 \dt (X^{\top}) = (\dt X)^{\top} + (B^{\sharp} (X^{\perp}))^{\top},\quad
 \dt (X^{\perp}) = (\dt X)^{\perp} -(B^{\sharp} (X^{\perp}))^{\top}.
\end{equation*}

The Divergence Theorem (with $X\in\mathfrak{X}_M$) states that
\begin{equation}\label{E-DivThm}
 \int_{M} (\Div X)\,{\rm d}\vol_g =0,
\end{equation}
when $M$ is closed (compact and without boundary); this is also true if $M$ is open and $X$ is supported in a relatively compact domain
$\Omega\subset M$.
 For any variation $g_t$ of metric $g$ on $M$ with $B=\dt g$ we have
\begin{equation}\label{E-dotvolg}
 \partial_t\,\big({\rm d}\vol_{g}\!\big) = \frac12\,(\tr_{g} B)\,{\rm d}\vol_{g},
\end{equation}
e.g., \cite{topp}.
By Lemma~\ref{L-divX} and \eqref{E-DivThm}--\eqref{E-dotvolg},
\begin{equation}\label{E-DivThm-2}
 \frac{d}{dt}\int_M (\Div X)\,{\rm d}\vol_g  =\int_M \Div\big(\dt X+\frac12\,(\tr_g B) X\big)\,{\rm d}\vol_g = 0
\end{equation}
for any variation $g_t$ of metric with ${\rm supp}\,(\dt g)\subset\Omega$,
and $t$-dependent $X\in\mathfrak{X}_M$ with ${\rm supp}\,(\dt X)\subset\Omega$.

Let ${\rm V}$ be the linear subspace of $TM\times TM$ spanned by
$({\cal D}\times\widetilde{\cal D})\cup(\widetilde{\cal D}\times{\cal D})$.
Thus, the product $TM\times TM$ is the sum of three subbundles, $\widetilde{\mD}\times\widetilde\mD$, ${\mD}\times\mD$ and ${\rm V}$.
Using this decomposition, we define the tensor in \eqref{E-gravity}.

\begin{definition}[see \cite{r2018}]\label{D-Ric-D}\rm
The symmetric $(0,2)$-tensor $\Ric_{\,\mD}$ in \eqref{E-gravity}, defined by its restrictions on three complementary subbundles of $TM\times TM$,
is referred to as the \textit{mixed Ricci curvature}:
\begin{equation}\label{E-main-0ij}
 \left\{\begin{array}{c}
 \Ric_{\,\mD\,|\,\mD\times\mD} = {r}_{{\mD}}
 -\<\tilde h,\,\tilde H\> +\widetilde{\cal A}^{\,\flat} -\widetilde{\cal T}^{\,\flat} +\Psi
 -{\rm Def}_{\cal D}\,H +\widetilde{\cal K}^{\,\flat} \\
 \hskip10mm +\,H^\flat\otimes H^\flat -\frac{1}{2}\,\Upsilon_{\,h,h} -\frac12\,\Upsilon_{\,T,T}
 -\frac{n-1}{p+n-2}\,\Div(\tilde H-{H})\,g^\perp, \\
 \Ric_{\,\mD\,|\,V}
 =  -4\<\theta,\, {\tilde H}\> -2(\Div(\alpha - \tilde\theta))_{\,|{\rm V}} -2\<{\tilde\theta} - {\tilde\alpha}, H\> \\
 \hskip10mm -\,2\,{\rm Sym}(H^{\flat}\otimes{\tilde H}^{\flat}) +2\,\tilde\delta_{H} - 4\,\Upsilon_{{\tilde\alpha}, \theta} - 2\,\Upsilon_{\alpha, {\tilde\alpha}} - 2\,\Upsilon_{\theta, {\tilde\theta}},\\
  \Ric_{\,\mD|\,\widetilde\mD\times\widetilde\mD} = {r}_{\widetilde{\mD}}-\<h,\,H\>+{\cal A}^\flat-{\cal T}^\flat
  +\widetilde\Psi -{\rm Def}_{\widetilde{\cal D}}\,\tilde H +{\cal K}^\flat \\
  \hskip10mm +\,\tilde H^\flat\otimes \tilde H^\flat -\frac{1}{2}\,\Upsilon_{\,\tilde h,\tilde h} -\frac12\,\Upsilon_{\,\tilde T,\tilde T} +\frac{p-1}{p+n-2}\,\Div(\tilde H-{H})\,g^\top.
 \end{array} \right.
\end{equation}
Here \eqref{E-main-0ij}$_3$ is dual to \eqref{E-main-0ij}$_1$
with respect to interchanging distributions $\widetilde{\cal D}$ and $\cal D$,
and their last terms vanish if $n=p=1$. 
Also,
 $\,{\rm S}_{\mD} := \tr_g\Ric_{\,\mD} = {\rm S}_{\rm mix} + \frac{p-n}{n+p-2}\,\Div(H-\tilde{H})$.
\end{definition}

The following theorem, which allows us to restore the partial Ricci curvature \eqref{E-main-0ij},
is based on calculating the variations with respect to $g$ of components in \eqref{eq-ran-ex}
and using \eqref{E-DivThm-2} for divergence~terms.
According to this theorem and Definition~\ref{D-Ric-D} we conclude that
a metric $g\in{\rm Riem}(M,\widetilde{\mD})$ is critical for the action~\eqref{actiongSmix} with fixed $\I =0$ (i.e., considered as a functional of $g$ only), with respect to volume-preserving variations of metric if and only if \eqref{E-gravity} holds.

\begin{theorem}[see \cite{rz-2}]
\label{T-main00}
A metric $g\in{\rm Riem}(M,\widetilde{\mD})$ is critical for the action~\eqref{actiongSmix} with fixed $\I =0$,
with respect to volume-preserving ${g}^\pitchfork$-variations if and only if
\begin{subequations}
\begin{eqnarray}\label{E-main-0i}
\nonumber
 &&\hskip-12mm {r}_{\mD} -\<\tilde h,\,\tilde H\> +\widetilde{\cal A}^\flat -\widetilde{\cal T}^\flat
 +\Psi -{\rm Def}_{\mD}\,H + \widetilde{\cal K}^\flat
 + H^\flat\otimes H^\flat -\frac{1}{2}\,\Upsilon_{h,h} -\frac12\,\Upsilon_{T,T} \\
 && -\frac12\,\big({\rm S}_{\rm mix} +\Div(\tilde H -H)\big)\,g^\perp = \lambda\,g^\perp, \\
 &&\hskip-15mm -4 \<\theta,\,{\tilde H}\> - 2(\Div(\alpha -\tilde \theta))_{\,|{\rm V}} -2\<{\tilde\theta} -{\tilde\alpha}, H\>
 - 2\,H^{\flat}\odot{\tilde H}^{\flat}
 +2\,{\tilde\delta}_{H} - 4\Upsilon_{{\tilde\alpha},\theta} -2\,\Upsilon_{\alpha,{\tilde\alpha}} -2\,\Upsilon_{\theta,{\tilde\theta}} = 0
\end{eqnarray}
\end{subequations}
for some $\lambda\in\RR$.
The Euler-Lagrange equation for volume-preserving ${g}^\top$-variations is dual to \eqref{E-main-0i}.
\end{theorem}

\begin{example}\label{Ex-2-1}\rm
For a \textit{space-time} $(M^{p+1},g)$ endowed with ${\widetilde\mD}$ spanned by a timelike unit vector field $N$, the tensor $\Ric_{\mD}$, see \eqref{E-main-0ij} with $n=1$, and its trace have the following particular form:
\begin{eqnarray}\label{E-RicD-flow}
 && \left\{\begin{array}{c}
 \Ric_{\,\mD\,|\,\mD\times\mD}
  = \eps_{N}(R_N +(\widetilde A_N)^2 -(\widetilde T^{\sharp}_N)^2 +[\,\widetilde T^{\sharp}_N,\,\widetilde A_N])^\flat
  +H^\flat\otimes H^\flat -\tilde\tau_1\,\widetilde h_{sc} -{\rm Def}_{\mD}\,H,\\
 {\Ric_{\,\mD}(\cdot\,,N)}_{\,|\,\mD} = {\Div}^{\perp}\widetilde T^{\sharp}_{N}|_{\,\cal D} +2\,(\widetilde T^{\sharp}_{N}({H}))^{\flat}, \\
 \Ric_{\,\mD}(N,N) = \eps_{N}\Ric_{N,N} -2\,\|\widetilde{T}\|^2 -\Div H,
\end{array} \right. \\
\label{E-RicD-flow-S}
 &&\quad\ {\rm S}_{\,\mD} = \eps_N\Ric_{N,N}+\Div(\eps_N\,\tilde\tau_1 N - {H}).
\end{eqnarray}
Here $\tilde\tau_i=\tr((\widetilde A_N)^i)$, $\widetilde A_N$ is the shape operator,
$\widetilde T$ is the integrability tensor
and $\widetilde h_{sc}$ is the scalar second fundamental form of $\mD$.
Note that the right-hand side of \eqref{E-RicD-flow}$_2$ 
vanishes when $\mD$ is integrable.
\end{example}

\textbf{2.3. Variations with respect to $\I$}.
The next theorem is based on calculating the variations with respect to $\I$ of components ${\rm S}_{\,\I}$ and ${\bar Q}/2$ in \eqref{barSmix} and using \eqref{E-DivThm-2} for divergence~terms.
Here $\{ e_\lambda \}$ are vectors of an adapted frame, without distinguishing distribution to which they belong.

\begin{theorem}
The Euler-Lagrange equation for \eqref{Eq-Smix} with fixed $g$, considered as a functional of an arbitrary $(1,2)$-tensor $\I$,
for all variations of \,$\I$, is the following algebraic system with spin tensor
$s_{\mu\nu}^c=2\,{\partial{\cal L}}/{\partial {\mathfrak T}_{\mu\nu}^c}$
$($hence $s_{\alpha\beta}^\gamma=\<s(e_\alpha,e_\beta),e_\gamma\>\,)$:
\begin{subequations}
\begin{eqnarray}\label{ELconnection1}
 \<\tr^\bot\I^*-\Hb, Z\>\,\<X,Y\> + \<\tr^\bot\I+\Hb, Y\>\,\<X,Z\> =-(\mathfrak{a}/2)\,\<s(X,Y),Z\>,\\
 \label{ELconnection2}
 \<\tr^\top\I^*-\Ht, W\>\,\<U,V\> + \<\tr^\top\I+\Ht, V\>\,\<U,W\>=-(\mathfrak{a}/2)\,\<s(U,V),W\>,\\
 \label{ELconnection3}
 \<\tr^\bot\I^*+\Ht,\, U\>\,\<X,Y\> -\<({A}_U - {T}^\sharp_U + 
  \I_U) X,\, Y\> = -(\mathfrak{a}/2)\,\<s(X,Y),U\>,\\
 \label{ELconnection4}
 \<\tr^\top\I^*+\Hb,\, X\>\,\<U,V\> -\<( \tA_X - \tT_X + 
  \I_X) U,\, V\> = -(\mathfrak{a}/2)\,\<s(U,V),X\>,\\
 \label{ELconnection5}
 \<\tr^\bot\I-\Ht,\, U\>\,\<X,Y\> + \<({A}_U + {T}^\sharp_U-\I_U) Y,\, X\> = -(\mathfrak{a}/2)\,\<s(X,U),Y\>,\\
 \label{ELconnection6}
 \<\tr^\top\I-\Hb,\, X\>\,\<U,V\> +\<({{\tA}}_X + {{\tT}}_X-\I_X) V,\, U\> =-(\mathfrak{a}/2)\,\<s(U,X),V\>,\\
 \label{ELconnection7}
 2\,\<{{\tT}}_X\,U,\, V\> + \<\I_U\, V + \I_V^*\, U,\, X\> = (\mathfrak{a}/2)\,\<s(X,U),V\>,\\
 \label{ELconnection8}
  2\,\<{T}^\sharp_U X,\, Y\> + \<\I_X Y + \I_Y^* X,\, U\> = (\mathfrak{a}/2)\,\<s(U,X),Y\>,
\end{eqnarray}
\end{subequations}
for all $X,Y,Z\in\widetilde\mD$ and $U,V,W\in\mD$,
see {\rm \cite[Eqs.~(15a-h)]{RZconnection}}, where variations of Lagrangian ${\cal L}$, i.e., spin tensor in {\rm (\ref{ELconnection1}-h)},
are omitted.
Here, {\rm (\ref{ELconnection2},d,f,h)} are dual to {\rm (\ref{ELconnection1},c,e,g)}.
\end{theorem}

\begin{proof}
Set $S=\dt\I^t_{\,|\,t=0}$ for a one-parameter family $\I^t\ (|t|<\varepsilon)$ of $(1,2)$-tensors.
Using Proposition~\ref{L-QQ-first} and removing integrals of divergences of compactly supported (in a domain $\Omega$) vector fields,
we~get
\begin{eqnarray*}
 &&\quad {\frac{\rm d}{\rm dt}\int_M \bar{\rm S}_{\,\rm mix}(\I^t)\,{\rm d} \vol_g}\,|_{\,t=0} \\
 && =\frac12\int_M \sum\Big\{
 \<S_a E_b,E_c\>\big(\<\tr^\bot{\mathfrak T}^*-\Hb, E_c\>\,\<E_a, E_b\> +\<\tr^\bot{\mathfrak T}+\Hb, E_b\>\,\<E_a, E_c\>\big) \\
 && +\,\<S_a E_b, {\cal E}_i\>\big( \<\tr^\bot{\mathfrak T}^*+\Ht, {\cal E}_i\>\,\<E_a, E_b\>
 -\<({A}_i - {T}^\sharp_i)E_a, E_b\> - \<\I_i E_a, E_b\> \big) \\
 && +\,\<S_a {\cal E}_i, E_b\>\big( \<\tr^\bot{\mathfrak T}-\Ht, {\cal E}_i\>\,\<E_a, E_b\>
 + \<({A}_i + {T}^\sharp_i)E_b, E_a\> - \<\I_i E_b, E_a\>\big) \\
 && +\,\<S_a {\cal E}_i, {\cal E}_j\> \big( \<(\tilde{A}_a - \tilde{T}^\sharp_a) {\cal E}_i, {\cal E}_j\>
 -\<(\tilde{A}_a + \tilde{T}^\sharp_a) {\cal E}_i, {\cal E}_j\> - \<\I_i{\cal E}_j + \I^*_j {\cal E}_i, E_a\> \big) \\
 && +\,\<S_i {\cal E}_j, {\cal E}_k\>\big( \<\tr^\top{\mathfrak T}^*-\Ht, {\cal E}_k\>\,\<{\cal E}_i, {\cal E}_j\>
 +\<\tr^\top{\mathfrak T} +\Ht, {\cal E}_j\>\,\<{\cal E}_i, {\cal E}_k\> \big) \\
 && +\,\<S_i {\cal E}_j, E_a\>\big(\<\tr^\top{\mathfrak T}^*+\Hb, E_a\>\,\<{\cal E}_i, {\cal E}_j\>
 -\<(\tilde{A}_a +\tilde{T}^\sharp_a){\cal E}_j, {\cal E}_i\> -\<\I_a{\cal E}_i, {\cal E}_j\> \big)  \\
 && +\,\<S_i E_a, {\cal E}_j\>\big(\<\tr^\top{\mathfrak T}-\Hb, E_a\>\,\<{\cal E}_i, {\cal E}_j\>
 +\<(\tilde{A}_a + \tilde{T}^\sharp_a){\cal E}_j, {\cal E}_i\> - \<\I_a {\cal E}_j, {\cal E}_i\> \big)  \\ \nonumber
 && +\,\<S_i E_a, E_b\>(\<({A}_i - {T}^\sharp_i) E_a, E_b\>
 {-}\<({A}_i + {T}^\sharp_i) E_a, E_b\>  {-} \<\I_a E_b + \I^*_b E_a, {\cal E}_i\> ) \Big\}\,{\rm d}\vol_g.
\end{eqnarray*}
Since no further assumptions are made about $S$ or $\I$, all the components $\<S_\mu e_{\lambda}, e_{\rho}\>$ are independent and the above formula
gives rise to (\ref{ELconnection1}-h), where $X,Y,Z\in\widetilde\mD$ and $U,V,W\in\mD$
are any vectors from an adapted frame. 
Observe that in every equation from (\ref{ELconnection1}-h) each term contains the same set of those vectors and is trilinear in them, so all these equations hold in fact for all vectors $X,Y,Z\in\widetilde\mD$ and $U,V,W\in\mD$.
Further below, we obtain many other formulas from computations in adapted frames, in the same way.
\end{proof}

Taking difference of symmetric 
(in $X,Y$) parts of (\ref{ELconnection3},e) with $s=0$ yields that $\widetilde{\cal D}$ is totally umbilical -- similar result for $\mD$ follows from dual equations (e.g., \cite{RZconnection}).
For vacuum space-time (${\cal L}=0$), the (\ref{ELconnection1}-h) are simplified to the following equations (\ref{ELconnectionNew1}-j).

\begin{corollary}[see Theorem~1 in \cite{RZconnection}]\label{T-main1}
Let a metric-affine manifold $(M,g,\bar\nabla=\bar\nabla-\nabla)$ be endowed with a non-degenerate distribution~$\widetilde{\mD}$.
Then $\I$ is critical for the action \eqref{actiongSmix} with fixed $g$ for all variations of $\I$ if and only if
$\,\widetilde{\mD}$ and $\mD$ are {totally umbilical} and $\I$ satisfies the following linear algebraic system
for all $X,Y\in\widetilde\mD$ and $U,V\in\mD$: 
\begin{subequations}
\begin{eqnarray}\label{ELconnectionNew1}
 && (\I_U\, V +\I^{*}_V\, U)^\top = -2\, {\widetilde{T}}(U, V), \\
\label{ELconnectionNew2}
 && (\tr^\bot\I^*)^\top = \Hb = -(\tr^\bot\I)^\top \quad {\rm for~} n>1, \\
\label{ELconnectionNew4}
 && \I^\top_U -\I^{*\top}_U = 2\, {T}^\sharp_U, \\
\label{ELconnectionNew5}
 && \I_U^\top + \I_U^{*\top} = \<\tr^\bot(\I +\I^*),\,U\>\id^\top , \\
\label{ELconnectionNew7}
 && (\tr^\bot(\I -\I^*))^\bot = (2-2/n)\,\Ht, \\
\label{ELconnectionNew8}
 && (\I_X\, Y +\I^{*}_Y\, X)^\bot = -2\,\Tb (X, Y), \\
\label{ELconnectionNew9}
 && (\tr^\top\I^*)^\bot = \Ht = -(\tr^\top\I)^\bot \quad {\rm for~} p>1, \\
\label{ELconnectionNew11}
 && \I_X^\bot -\I_X^{*\bot} = 2\, {{\tilde T}}^\sharp_X, \\
\label{ELconnection1ab}
 && {\I}_X^\bot + \I_X^{*\bot} = \<\tr^\top(\I +\I^*),\,X\>\id^\bot , \\
\label{ELconnectionNew14}
 && (\tr^\top(\I -\I^*))^\top = (2-2/p)\,\Hb .
\end{eqnarray}
\end{subequations}
\end{corollary}

\begin{example}\rm
For our $(M^{p+1},g,{\widetilde\mD})$, see Example~\ref{Ex-2-1}, the system (\ref{ELconnection1}-h) reduces to
\begin{eqnarray*}
 \<\tr^\bot({\mathfrak T}^*+{\mathfrak T}), N\> =-(\mathfrak{a}/2)\,\<s(N,N),N\>,\\
 \<\tr^\top{\mathfrak T}^* -\Ht,\, W\>\,\<U,V\> +\<\tr^\top{\mathfrak T} +\Ht,\, V\>\,\<U,W\> =-(\mathfrak{a}/2)\,\<s(U,V),W\>,\\
 \<\tr^\bot{\mathfrak T}^*,\, U\> - 
  \<{\mathfrak{T}}_U\,N, N\> = -(\mathfrak{a}/2)\,\<s(N,N),U\>,\\
 (\<\tr^\top{\mathfrak T}^*, N\>+\tilde\tau_1)\,\<U,V\> -\<(\tilde{A}_N -\tilde{T}^\sharp_N
 + 
  {\mathfrak{T}}_N) U,\, V\> = -(\mathfrak{a}/2)\,\<s(U,V),N\>,\\
 \<\tr^\bot{\mathfrak T},\, U\> - \<{\mathfrak{T}}_U\,N, N\> = -(\mathfrak{a}/2)\,\<s(N,U),N\>,\\
 (\<\tr^\top{\mathfrak T}, N\>-\tilde\tau_1)\,\<U,V\> +\<(\tilde{A}_N +\tilde{T}^\sharp_N-{\mathfrak{T}}_N) V,\, U\> =-(\mathfrak{a}/2)\,\<s(U,N),V\>,\\
 \<\,2\,\tilde{T}   
 (U, V) + {\mathfrak{T}}_U\, V + {\mathfrak{T}}_V^*\,U,\, N\> = (\mathfrak{a}/2)\,\<s(N,U),V\>,\\
 \<({\mathfrak{T}} + {\mathfrak{T}}^*)_N\,N,\, U\> = (\mathfrak{a}/2)\,\<s(U,N),N\>,
\end{eqnarray*}
where $U,V,W\in\mD$.
\end{example}

\section{Main results}
\label{sec:main}

In Section~\ref{sec: 2-1} we consider the total mixed scalar curvature of contorsion tensor for general and particular connections,
 e.g., metric and statistical.
In Section~\ref{sec: 2-2} we consider the total mixed scalar curvature of statistical manifolds endowed with a distribution and metric-affine doubly twisted products.
In Section~\ref{sec: 2-3} we consider the total mixed scalar curvature of Riemann-Cartan manifolds endowed with a distribution.
In~Section~\ref{sec:2-4}, we derive the Euler-Lagrange equations for semi-symmetric connections and present the mixed Ricci tensor explicitly in \eqref{E-Ric-D-semi-sym}.
Our aims are to find out which metrics admit critical points of examined functionals and which components of $\I$ in these particular cases determine whether or not its mixed scalar curvature is critical in its class of connections. 
This might help to achieve better understanding of both mixed scalar curvature invariant and the role played by some components of contorsion tensor.

\subsection{Variational problem with contorsion tensor}
\label{sec: 2-1}

By Proposition~\ref{L-QQ-first} and \eqref{E-SK}, we have the following decomposition \cite{r-affine} (note that these are terms of $-Q$ in the first line of \eqref{E-defQ}):
\[
 2\,{\rm S}_{\,\I} = \<\tr^\top\I,\,\tr^\bot\I^*\> +\<\tr^\bot\I,\,\tr^\top\I^*\> -\< \I^\wedge, \I^* \>_{| V} .
\]
We consider arbitrary variations $\I(t),\ \I(0)=\I,\ |t|<\eps$,
and variations corresponding to metric and statistical connections,
while $\Omega\subset M$ contains supports of infinitesimal variations $\dt\I(t)$.
In such cases, the Divergence Theorem states that if $X\in\mathfrak{X}_M$ is supported in $\Omega$ then \eqref{E-DivThm} holds.

\begin{theorem}\label{propELSmixI}
A pair $(g,\,\I)$ is critical for the action \eqref{actiongISmix} with respect to all variations of $\,\I$ and $g$ if and only if
$\I$ satisfies the following algebraic systems
(for all $X,Y,Z\in\widetilde\mD$ and $U,V,W\in\mD$): 
\begin{subequations}
\begin{eqnarray}\label{ELSmixIadapted}
 && \tr^\top(\I_V \I^\wedge_U) -\frac{1}{2}\,\<\I_V\,U +\I_U\,V,\ \tr^\top \I^* \> = 0,\\
\label{ELSmixImixed}
 \nonumber
 &&\<\tr^\bot \I -\tr^\top \I,\ \I^*_Y\, U\> -\<\I_Y U +\I_U\, Y,\,\tr^\top \I^*\> - \tr^\bot(\I^*_Y (\I^*)^\wedge_{\,U}) \\
 &&\quad +\tr^\top 
 \big( \I^*_Y (\I^*)^\wedge_{\,U} 
  +\I_U \I^\wedge_{\,Y} +\I_Y \I^\wedge_{\,U}\big)
 = 0 \\
\label{ELSmixIadapteddual}
 &&
 \tr^\top(\I_Y \I^\wedge_X) -\frac{1}{2}\,\<\I_Y\, X +\I_X\, Y,\ \tr^\bot \I^*  \> = 0,
\end{eqnarray}
\end{subequations}
and
\begin{subequations}
\begin{eqnarray}\label{E-34}
 &&(\I^*_X\,Y +\I_Y\,X)^\bot =0,\\
 &&  (\I_U\,V +\I^*_V\,U)^\top =0, \\
\label{E-34c}
 &&\<X,Z\>\<\tr^\bot\I, Y\> +\<X,Y\>\,\<\tr^\bot\I^*, Z\> = 0,\\
 &&\<U,V\>\<\tr^\top\I^*,\, W\> +\<U,W\>\<\tr^\top\I,\, V\>=0,\\
 &&\I_U^\top = \<\tr^\bot\I,\, U\>\id^\top, \\
 &&\I_X^\bot  = \<\tr^\top\I^*,\, X\>\id^\bot, \\
 &&(\tr^\bot(\I -\I^*))^\bot =0,\qquad
 (\tr^\top(\I -\I^*))^\top =0.
\end{eqnarray}
\end{subequations}
Moreover, if $n>1$ and $p>1$ then {\rm (\ref{E-34c},d)} read as
\begin{eqnarray} \label{criticaltrIinlargedimensions}
(\tr^\bot\I)^\top = 0 = (\tr^\bot\I^*)^\top,\quad
(\tr^\top\I^*)^\bot = 0 = (\tr^\top\I)^\bot.
\end{eqnarray}
\end{theorem}

\begin{proof}
From Proposition~\ref{L-QQ-first} and Lemma~\ref{L-dT-3}, for a $g^\pitchfork$-variation $g_t$ of metric $g$ we obtain
\begin{eqnarray}\label{Eq-47}
 && 2 \, 
  \dt {\rm S}_{\,\I} (g_t) = \dt \<\tr^\top\I,\,\tr^\bot\I^*\> +\dt \<\tr^\bot\I,\,\tr^\top\I^*\>
 -\dt\<\I^\wedge, \I^* \>_{| V} \nonumber \\
 && = \frac{1}{2} \sum B({\cal E}_i, {\cal E}_j)
 \big(\<\tr^\top\I, \I^*_i {\cal E}_j - \I^*_j {\cal E}_i\> -\<\I_j {\cal E}_i + \I_i {\cal E}_j 
 , \tr^\top\I^*\>
 +2\,\<\I^*_{j} E_a, \I_a {\cal E}_i\>\big)\nonumber \\
 && +\sum B({\cal E}_i, E_b) \big(\<\tr^\bot\I - \tr^\top\I,\, \I^*_b {\cal E}_i\>
 -\<\I_b {\cal E}_i + \I_i E_b, \tr^\top\I^*\> \nonumber \\
 &&
 +\,\<\I^*_a {\cal E}_i, \I_{ b} E_a\> +\<\I^*_{b} E_a, \I_a {\cal E}_i\>
 +\<\I^*_i E_a, \I_a E_b\> -\<\I^*_j {\cal E}_i, \I_b {\cal E}_j\> \big).
\end{eqnarray}
Thus, $\dt {\rm S}_{\,\I}(g_t) =0$ if and only if the right hand side of \eqref{Eq-47} vanishes for all symmetric tensors $B=\partial_t g$.
For the $({\mD}\times{\mD})$-part of $B$ we get
\begin{equation*}
 \sum B({\cal E}_i, {\cal E}_j)\big(\frac{1}{2}\,\<\tr^\top\I, \I^*_i {\cal E}_j -\I^*_j {\cal E}_i\>
 -\frac{1}{2}\,\<\I_j {\cal E}_i +\I_i {\cal E}_j, \tr^\top\I^*\>
 +\tr^\top(\I_j \I^\wedge_i)\big) = 0,
\end{equation*}
but since $B$ is arbitrary and symmetric and $\I^*_i {\cal E}_j -\I^*_j {\cal E}_i$ is skew-symmetric,
this can be written as \eqref{ELSmixIadapted}.
 For the mixed part of $B$ (i.e., 
 $B$ restricted to the subspace $V$) we get the following Euler-Lagrange equation:
\begin{eqnarray*}
 && \sum B({\cal E}_i, E_b) \big( \<\tr^\bot\I, \I^*_b {\cal E}_i\>
 - \<\tr^\top\I, \I^*_b {\cal E}_i\> - \<\I_b {\cal E}_i + \I_i E_b, \tr^\top\I^*\>  \\
 && +\,\<\I^*_a {\cal E}_i, \I_{ b} E_a\> + \<\I^*_i E_a, \I_a E_b\> + \<\I^*_{b} E_a, \I_a {\cal E}_i\>
 - \<\I^*_j {\cal E}_i, \I_b {\cal E}_j\> \big) =  0.
\end{eqnarray*}
From this we obtain \eqref{ELSmixImixed}.
Taking dual equation to \eqref{ELSmixIadapted} with respect to interchanging distributions $\widetilde{\mD}$ and ${\mD}$, we obtain \eqref{ELSmixIadapteddual}, which is the Euler-Lagrange equation for $g^\top$-variations. The proof of (\ref{E-34}-g), see \cite{RZconnection}, is based on calculation of variations with respect to $\I$ of ${\rm S}_{\,\I}$ and using \eqref{E-DivThm-2}.
\end{proof}

\begin{definition}[see Section~4 in \cite{RZconnection}]\rm
The \textit{doubly twisted product} $B\times_{(v,u)} F$ of metric-affine manifolds $(B,g_B,\I_B)$ and $(F, g_F,\I_F)$
(or the \textit{metric-affine doubly twisted product})
is a mani\-fold $M=B\times F$ with the metric $g = g^\top + g^\bot$ and the affine connection, whose contorsion tensor is
$\I={\I}^\top+{\I}^\bot$,~where
\begin{eqnarray*}
 g^\top(X,Y) \eq v^2 g_B(X^\top,Y^\top),\quad g^\bot(X,Y)=u^2 g_F(X^\bot,Y^\bot),\\
 {\I}^\top_XY \eq u^2(\I_B)_{X^\top}Y^\top,\quad {\I}^\bot_XY=v^2(\I_F)_{X^\bot}Y^\bot,
\end{eqnarray*}
and the warping functions $u,v\in C^\infty(M)$ are positive.
\end{definition}

From Theorem~\ref{propELSmixI} we obtain the following

\begin{corollary}
A metric-affine doubly twisted product $B\times_{(v,u)} F$ with $\sum\nolimits_a\eps_a \neq 0 \ne \sum\nolimits_i\eps_i$ is critical for \eqref{actiongISmix} with respect to all variations of $\,\I$ and $g$ if and only if
\begin{equation} \label{BFtraces}
 \tr\I_B=0=\tr\I_F.
\end{equation}
\end{corollary}

\begin{proof}
It was proven in \cite[Corollary~13]{RZconnection} that a metric-affine doubly twisted product $B\times_{(v,u)} F$
is critical for \eqref{actiongISmix} with fixed $g$ and for variations of $\,\I$ if and only if \eqref{BFtraces} holds.
It can be easily seen that for such doubly twisted product satisfying $\tr\I_B=0=\tr\I_F$ all terms in (\ref{ELSmixIadapted}-c) vanish.
\end{proof}

\begin{corollary} \label{statisticalcritSmixI}
A pair $(g, \I)$, where $\I$ is the contorsion tensor of a statistical connection on $(M, g)$, is critical for the action \eqref{actiongISmix} with respect to all variations of metric, and variations of $\I$
corresponding to statistical connections if and only if $\I$ satisfies the following algebraic system:
\begin{subequations}
\begin{eqnarray}
\label{ELSmixIstat1}
 && (\tr^\top\I)^\top = 0 = (\tr^\bot\I)^\bot, \\
\label{ELSmixIstat2}
 && (\I_X\,Y)^\bot = 0 = (\I_U\,V)^\top,\quad X,Y\in\widetilde\mD,\ \ U,V\in\mD.
\end{eqnarray}
\end{subequations}
\end{corollary}

\begin{proof}
By 
\cite[Corollary~7]{RZconnection}, $\I$ is critical for the action $\I \mapsto \int_M {\rm S}_{\,\I}\,{\rm d}\vol_g$, see \eqref{actiongISmix}, with respect to variations of $\I$
corresponding to statistical connections if and only if the following equations hold:
\begin{subequations}
\begin{eqnarray}\label{ELSmixIstatI1}
&& (\tr^\top \I)^\perp = 0 = (\tr^\perp \I)^\top , \\
\label{ELSmixIstatI2}
&& (\I_U V)^\top = \frac{1}{2}\,\<U,V\> (\tr^\top \I)^\top , \\
&& (\I_X Y)^\perp = \frac{1}{2}\,\<X,Y\> (\tr^\perp \I)^\perp ,
\end{eqnarray}
\end{subequations}
for all $X,Y\in\widetilde\mD$ and $U,V\in\mD$.
If (\ref{ELSmixIstat1},b)
hold, then also (\ref{ELSmixIstatI1}-c) hold, moreover if \eqref{ELSmixIstat2} is satisfied and $\I$ corresponds to a statistical connection, then all terms in equations (\ref{ELSmixIadapted}-c) vanish.

On the other hand, if (\ref{ELSmixIstatI1}-c) hold, then \eqref{ELSmixIadapted} becomes
\begin{equation} \label{ELSmixIadaptedstat}
\frac{n}{4} (\tr^\perp \I)^{ \perp \flat} \otimes (\tr^\perp \I)^{ \perp \flat}
- \frac{3}{4} \<  (\tr^\top \I)^\top , (\tr^\top \I)^\top \>\, g^\perp =0,
\end{equation}
and \eqref{ELSmixIadapteddual} becomes dual to the above.
If $p>1$ and $\< (\tr^\perp \I)^{ \perp} ,(\tr^\perp \I)^{ \perp}\> \ne 0$, then there is $W \in \mD$ such that $\<W ,W \> \neq 0$ and $\< W , (\tr^\perp \I)^{ \perp} \> =0$, and evaluating \eqref{ELSmixIadaptedstat} on $W \otimes W$ we obtain $(\tr^\top \I)^\top =0$ and then it also follows from \eqref{ELSmixIadaptedstat} that $(\tr^\perp \I)^{ \perp } =0$.
If $p>1$ and $(\tr^\perp \I)^{ \perp} =0$, then we obtain $(\tr^\top \I)^\top =0$ from \eqref{ELSmixIadapted}, as $g^\perp$ is non-degenerate.
If $p>1$ and $\< (\tr^\perp \I)^{\perp}, (\tr^\perp \I)^{\perp} \> = 0$ but $(\tr^\perp \I)^{ \perp} \ne 0$, then \eqref{ELSmixIadapted} evaluated on $(\tr^\perp\I)^{\perp}\otimes W$, where $W\in\mD$, implies that
\[
 \<(\tr^\top\I)^\top, (\tr^\top\I)^\top\>\<(\tr^\perp\I)^{\perp}, W\> =0
\]
and since $W$ here is arbitrary, it follows that $\<(\tr^\top\I)^\top, (\tr^\top\I)^\top\> = 0$, and then it also follows from \eqref{ELSmixIadaptedstat} that $(\tr^\perp \I)^{\perp}=0$.

Equalities $(\tr^\top \I)^\top = 0 = (\tr^\perp \I)^{ \perp }$ together with (\ref{ELSmixIstatI2},c) yield \eqref{ELSmixIstat2}.

If $n>1$ we can similarly use \eqref{ELSmixIadapteddual} for the same effect, and if $n=p=1$ then \eqref{ELSmixIadaptedstat} becomes
\[
 \<(\tr^\perp \I)^{ \perp} , (\tr^\perp \I)^{ \perp} \> = 3 \<  (\tr^\top \I)^{ \top} , (\tr^\top \I)^{ \top} \>,
\]
which together with its dual imply $(\tr^\top \I)^\top =0 = (\tr^\perp \I)^{ \perp }$, and again we obtain \eqref{ELSmixIstat2} from (\ref{ELSmixIstatI2},c).
\end{proof}

Next we consider metric connections. Using (\ref{E-34}-g), we obtain the following.

\begin{corollary}[see \cite{RZconnection}]
A contorsion tensor $\I$ corresponding to a {metric connection} is critical for the action \eqref{actiongISmix}
with fixed $g$ for all variations of $\I$ corresponding to metric connections if and only if
$\I$ satisfies the following linear algebraic system
(for all $X,Y\in\widetilde\mD$ and $U,V\in\mD$): 
\begin{subequations}
\begin{eqnarray}
 \label{Tmetriccrit1}
 && (\I_Y\, X +\I^*_X\, Y)^\bot = 0 = (\I_U\, V +\I^*_V\, U)^\top,\\
 \label{Tmetriccrit3}
 && (\tr^\top\I)^\top = 0 = (\tr^\bot\I)^\bot,\\
 \label{Tmetriccrit5}
 && \I_X^\bot = 0 = \I_U^\top, \\
\label{SmixImetricdimnp}
 && (\tr^\bot\I)^\top=0\quad for\ n>1,\quad
 (\tr^\top\I)^\bot=0\quad for\ p>1.
\end{eqnarray}
\end{subequations}
\end{corollary}

\begin{corollary}
A pair $(g, \I)$, where $\I$ is the contorsion tensor of a metric connection on $(M, g)$, is critical for \eqref{actiongISmix} with respect to all variations of metric, and variation of $\I$
correspon\-ding to metric connections if and only if {\rm (\ref{Tmetriccrit1}-d)} are satisfied and the following algebraic system
(where $X\in\widetilde\mD$ and $U\in\mD$) holds:
\begin{eqnarray*}
\nonumber
 && \tr^\top((\I_U)^\bot(\I_X^\wedge)^\bot +2\,(\I_U^\wedge)^\top(\I_X)^\top) -\tr^\bot((\I_U^\wedge)^\top(\I_X)^\top)
 +\<\tr^\bot\I,\, (\I_X\, U)^\top\> = 0 .
\end{eqnarray*}
\end{corollary}

\begin{proof}
In \eqref{ELSmixIadapted}, by \eqref{Tmetriccrit5} we have
$\<\I_a {\cal E}_i, E_b\> =0 = \<\I_a\, {\cal E}_i,\, {\cal E}_k\>$,
and by \eqref{Tmetriccrit3} also $\<\I^*_a\, E_a,\, E_b\>=0$.
Hence, what remains in \eqref{ELSmixIadapted} is
\[
 \<(\I_j\, {\cal E}_i +\I_i\, {\cal E}_j)^\bot,\, \tr^\top\I^*\> = 0,\quad \forall\, i,j.
\]
By \eqref{SmixImetricdimnp}, this is identity if $p>1$. On the other hand, for $p=1$ it reduces to
\[
 2 \<\I_1\,{\cal E}_1, {\cal E}_1\>\<\tr^\top\I^*, {\cal E}_1\> = 0,
\]
and by \eqref{Tmetriccrit3}, $\<\I_1\, {\cal E}_1,\, {\cal E}_1\>=0$.
Therefore, \eqref{ELSmixIadapted} is satisfied if (\ref{Tmetriccrit1}-c) and the second equation in \eqref{SmixImetricdimnp}
are satisfied. Using dual parts of (\ref{Tmetriccrit1}-d) we obtain analogous result for \eqref{ELSmixIadapteddual}.
From (\ref{Tmetriccrit1}-d) we~have for all $b,c,i,k$, 
\begin{eqnarray*}
 && \sum \<\I_a E_a , E_c\>=0,\quad
 \<\I^*_b {\cal E}_i,\,{\cal E}_k\>=0,\quad
 \sum \<\I^*_a E_a, E_c\>=0,\\
 && \<\I^*_b {\cal E}_i,\,{\cal E}_k\>=0,\quad
 \<\I^*_i E_b, E_c\>=0,\quad
 \<\I_b {\cal E}_i,\,{\cal E}_k\>=0. 
\end{eqnarray*}
Thus, in \eqref{ELSmixImixed} we have only the following terms:
\begin{eqnarray*}
 && \sum \<\I_j {\cal E}_j, E_c\> \<\I^*_b {\cal E}_i, E_c\> + \sum \<\I^*_a {\cal E}_i, E_c\> \<\I_{ b} E_a, E_c\>
 + \sum \<\I^*_i E_a, {\cal E}_k\> \<\I_a E_b, {\cal E}_k\>\\
 && +\, \sum \<\I^*_{b} E_a, E_c\> \<\I_a {\cal E}_i, E_c\>
 - \sum \<\I^*_j {\cal E}_i, E_c\> \<\I_b {\cal E}_j, E_c\> = 0
\end{eqnarray*}
for all $b,i$.
Using $\I^* = -\I$ (metric compatibility of $\I$), we obtain that \eqref{ELSmixImixed} is equivalent to
\begin{eqnarray*}
 && \sum \<\I_j {\cal E}_j, E_c\> \<\I_b {\cal E}_i, E_c\> +2 \sum \<\I_a {\cal E}_i, E_c\> \<\I_{ b} E_a, E_c\> \\
 && + \sum \<\I_i E_a, {\cal E}_j\> \<\I_a E_b, {\cal E}_j\> - \sum \<\I_j {\cal E}_i, E_c\> \<\I_b {\cal E}_j, E_c\> = 0
\end{eqnarray*}
for all $b,i$. This completes the proof.
\end{proof}

The results obtained when considering the action \eqref{actiongISmix} on metric-affine doubly twisted products,
allow us to determine which of these structures are critical for the action \eqref{actiongSmix}.

\begin{proposition}
A metric-affine doubly twisted product $B\times_{(v,u)} F$ 
is critical for \eqref{actiongSmix} with respect to all variations of $g$ and $\I$ if and only if \eqref{BFtraces} holds~and
\begin{equation} \label{BFtotallygeodesic}
\nabla^\top u =0 =\nabla^\perp v.
\end{equation}
\end{proposition}

\begin{proof}
It was proven in \cite{RZconnection} that a metric-affine doubly twisted product $B\times_{(v,u)} F$ 
is critical for action \eqref{actiongISmix} with fixed $g$, 
with respect to all variations of $\,\I$, if and only if \eqref{BFtotallygeodesic} and \eqref{BFtraces} hold.
Note that \eqref{BFtotallygeodesic} means that $TB$ and $TF$ as (integrable) distributions on $B\times_{(v,u)} F$ are totally geodesic.
It can be easily seen that if 
\eqref{BFtraces} holds and the distributions are integrable and totally geodesic, then all terms in all variation formulas obtained in Lemma \ref{L-dT-3} vanish.
\end{proof}

\subsection{Statistical connections}
\label{sec: 2-2}

We define a new tensor $\Theta = \I -\I^* +\I^\wedge - \I^{* \wedge }$,
composed of some terms appearing in \eqref{E-defQ}.

\begin{theorem}\label{propstatcrit}
Let $(g, \I)$ correspond to a statistical connection. Then $(g, \I)$ is critical for
\eqref{actiongSmix} with respect to volume-preserving variations of $g$ and
variations of $\I$ among all $(1,2)$-tensors 
if and only if the following conditions are satisfied:

\smallskip
 1. $\widetilde{\mD}$ and ${\mD}$ are both integrable,

 2. $(\tr^\top \I)^\top = 0 = (\tr^\perp \I)^\perp$, see (\ref{ELSmixIstat1},b),

 3. $\I_X : \widetilde{\mD} \rightarrow \widetilde{\mD}$ for all $X\in\widetilde\mD$,

 4. $\I_U : {\mD} \rightarrow {\mD}$ for all $U\in\mD$,

 5. if $n>1$ then ${\tilde H}=0$,

 6. if $p>1$ then $H=0$,

 7. $\widetilde{\mD}$ and ${\mD}$ are both totally umbilical,
\newline
and the following equations $($trivial when $n>1$ and $p>1$, see 5. and 6. above$)$  hold for some $\lambda\in\RR$:
\begin{subequations}
\begin{eqnarray}\label{E-main-0iumbint}
 && \frac{n-1}{n}\,H^\flat\otimes H^\flat -\frac{1}{2} \big(\frac{n-1}{n}\,\<H,H\> +\frac{p-1}{p}\,\<\tH, \tH\>
 +\frac{2(p-1)}{p}\,\Div\tH\big)\,g^\perp = \lambda\,g^\perp, \\
\label{E-main-0iiumbint}
 && \frac{n-1}{n}\,\big({\tilde \delta}_H -\frac{p-1}{p}\,H^\flat \odot \tH^\flat\big) =0, \\
\label{E-main-0iiiumbint}
 && \frac{p-1}{p}\,\tH^\flat \otimes \tH^\flat -\frac{1}{2}\,\big(\frac{p-1}{p}\,\<\tH,\tH\>
 +\frac{n-1}{n}\,\<H, H\> +\frac{2(n-1)}{n}\,\Div H \big) g^\top = \lambda\,g^\top .
\end{eqnarray}
\end{subequations}
\end{theorem}

\begin{proof}
For any $\I$ that corresponds to a statistical connection, we have $\I^\wedge = \I$ and $\I^* = \I$.
Condition~1 follows from (\ref{ELconnectionNew1},f)
and $\I = \I^\wedge$. Then (\ref{ELconnectionNew1},f), condition 1 and
\[
 \<\I_i {\cal E}_j, E_a\>=\<\I^*_j {\cal E}_i, E_a\>=\<\I_a {\cal E}_i, {\cal E}_j\>,\quad\forall\, i,j,a,
\]
yield condition 3. We get condition 5 from $\I=\I^*$ and \eqref{ELconnectionNew2}.  Conditions 4 and 6 are dual to conditions 3 and 5, and are obtained analogously. Condition 2 follows from $\I= \I^*$, condition 3 (and its dual condition 5) and \eqref{ELconnectionNew4} (and its dual \eqref{ELconnectionNew9}). Condition 7 follows from Corollary~\ref{T-main1}.

Let $g_t$ be a $g^\pitchfork$-variation of $g$.
Although for statistical manifolds, \eqref{E-Q1Q2-gen} reads as
\begin{equation}\label{eqvarstat}
 \bar{{\rm S}}_{\,\rm mix} -{{\rm S}}_{\,\rm mix}
 = {{\rm S}}_{\I} = \<\tr^\top\I,\,\tr^\bot\I\> +\frac12\,\<\I,\,\I\>_{\,|\,V} ,
\end{equation}
we cannot vary this formula with respect to metric with fixed $\I$, because when $g$ changes, $\I$ may no longer correspond to statistical connections (condition $\I = \I^*$ may not be preserved by the variation). Instead, we use Lemma~\ref{L-dT-3} and derive
from \eqref{dtIIproduct}
for $\I$ corresponding to a statistical connection (for which $\I = \I^* = \I^\wedge$ and $\Theta=0$) that
\begin{eqnarray*}
 && \dt \<\I^*, \I^\wedge\>_{\,|\,V}
 =\sum B({\cal E}_i, E_b)\big(\<\I_j {\cal E}_i, \I_b {\cal E}_j\> -3 \<\I_a {\cal E}_i, \I_{ b} E_a\>\big)
 -\sum B({\cal E}_i, {\cal E}_j) \<\I_{j} E_a, \I_a {\cal E}_i\> .
\end{eqnarray*}
From conditions 3-4:
 $\dt \<\I^*, \I^\wedge\>_{\,|\,V} = 0$. 
From \eqref{dtThetaA} with $\Theta=0$ we have
\begin{eqnarray*}
 && \dt \< \Theta, A \>
  = 2\sum B({\cal E}_j, E_b) \big(\< h (E_a, E_b), {\cal E}_i\> \<\I_a {\cal E}_i, {\cal E}_j\>
  -\< h (E_a, E_c), {\cal E}_j\> \<\I_a E_b, E_c\>\big) \nonumber \\
 && -\,2\sum B({\cal E}_i, {\cal E}_j) \< h (E_a, E_b), {\cal E}_i\> \<\I_a {\cal E}_j, E_b\> .
\end{eqnarray*}
For totally umbilical distribution, the last equation further simplifies to
\begin{eqnarray*}
 && \dt \< \Theta, A \> = \frac{2}{n} \sum B({\cal E}_j, E_b) \big(\< H, {\cal E}_i\> \<\I_b {\cal E}_i, {\cal E}_j\>
 -\< H, {\cal E}_j\> \<\I_a E_b, E_a\> \big) \nonumber \\
 && -\frac{2}{n} \sum B({\cal E}_i, {\cal E}_j) \< H, {\cal E}_i\> \<\I_a {\cal E}_j, E_a\> .
\end{eqnarray*}
From conditions 2-4 and we obtain in the above
 $\dt \< \Theta, A \> = 0$.
For integrable distributions, since $\Theta=0$, we have
\begin{eqnarray*}
 \dt \< \Theta, T^\sharp \> \eq 0,\quad
 \dt \< \Theta, {\tilde T}^\sharp \> = 0,
\end{eqnarray*}
and from \eqref{dtThetatildeA}, with $\Theta=0$ and totally umbilical distributions, we have
\begin{eqnarray*}
 \dt \< \Theta, {\tilde A} \> \eq
  \frac{2}{p} \sum B({\cal E}_j, E_b) \big(\<{\tilde H}, E_a\> \<\I_a {\cal E}_j, E_b\>
  -\<{\tilde H}, E_b\> \<\I_j {\cal E}_i, {\cal E}_i\> \big) \\
 \plus\frac{2}{p} \sum B({\cal E}_i, {\cal E}_j) \<{\tilde H}, E_a\> \<\I_a {\cal E}_j, {\cal E}_i\> .
\end{eqnarray*}
From conditions 3-4 and 2 we get in the above
\begin{equation*}
 \dt \< \Theta, {\tilde A}\> = -\frac{2}{p}\sum B({\cal E}_j, E_b) \<{\tilde H}, E_b\> \<\I_i {\cal E}_i, {\cal E}_j\>=0.
\end{equation*}
From conditions 3-4, using \eqref{dttracetopI} and \eqref{dttraceperpI}, we get
\begin{eqnarray*} 
 && \dt \<\tr^\top \I, \tr^\perp \I^*\> = -\sum B({\cal E}_i, E_b) \<\tr^\top\I, E_c\> \<E_c, \I_b {\cal E}_i\> =0 ,\\
 && \dt \<\tr^\top \I^*, \tr^\perp \I\>
 =\!\sum B({\cal E}_j, E_b)
 \<\tr^\bot\I -2\tr^\top\I,\, \I_b {\cal E}_j\>
 -\!\sum B({\cal E}_i, {\cal E}_j) \<\I_j {\cal E}_i, \tr^\top\I\> .
\end{eqnarray*}
From conditions 3-4 and 2 we get
 $\dt \<\tr^\top \I^*, \tr^\perp \I\> = 0$.
From $\I^* = \I$, using \eqref{dtIEaH}, we obtain
\begin{eqnarray*} 
 && \dt \< \tr^\top(\I^* -\I), \tH -H\>
 =\sum B({\cal E}_i, {\cal E}_j) \<\tr^\top\I, {\cal E}_j\> \<{\cal E}_i, H\>\\
 && +\sum B({\cal E}_j, E_b) \big( \<\I_b {\cal E}_j, \tH -H\>
 +\<\tr^\top\I, E_b\> \<{\cal E}_j, H\> -\<\tr^\top\I, {\cal E}_j\> \<E_b, \tH\> \big) .
\end{eqnarray*}
From conditions 3-4 and 2 we get
 $\dt\, \<\, \sum (\I^*_a -\I_a) E_a, \tH -H\> = 0$.
Similarly, from \eqref{dtIeiH} we obtain
\begin{eqnarray*} 
 && \dt \<\tr^\bot(\I^* -\I), \tH -H\>
 = \sum B({\cal E}_i, {\cal E}_j) \big(\<\I_i {\cal E}_j, \tH -H\> + \<\tr^\bot\I, {\cal E}_i\> \< H, {\cal E}_j\> \big)\\
 && +\!\sum B({\cal E}_j, E_b) \big(\<\tr^\bot\I, E_b\> \< H, {\cal E}_j\>
 + \<\I_j E_b, \tH -H\>  -\<\tr^\bot\I, {\cal E}_j\> \<\tH, E_b\> \big) .
\end{eqnarray*}
From conditions 3-4 and 2 we get in the above
\begin{equation*} 
\dt \<\tr^\bot(\I^* -\I), \tH -H\> = - \sum B({\cal E}_i, {\cal E}_j)\<\I_j\, {\cal E}_i, H\> .
\end{equation*}
By condition 6 we have $H=0$ if $p>1$ and if $p=1$ we only have $i=j=k=1$ and by condition~2,
\[
 \<\I_j {\cal E}_i, {\cal E}_k\>=\<\tr^\perp \I, \,{\cal E}_1\>=0.
\]
Hence,
for $\I$ corresponding to a statistical connection satisfying the assumptions,
any variation of $\overline{{\rm S}}_{\rm mix}$ with respect to $g$ is just a variation of ${\rm S}_{\rm mix}$ with respect to $g$.
Thus, remaining (\ref{E-main-0iumbint}-c) are equations of Theorem~\ref{T-main00}
written for both distributions integrable and umbilical.
\end{proof}

\begin{corollary}
Let $M$ be a closed manifold. Then $(g, \I)$, where $\I$ corresponds to a statistical connection on $(M,g)$, is critical for the action
\eqref{actiongSmix} with respect to all variations of $g$ and $\I$
if and only if $(g, \I)$ satisfy conditions 1-7 of Theorem~\ref{propstatcrit}; furthermore, either $n=p=1$ or $H=0=\tH$.
\end{corollary}

\begin{proof}
Clearly, (\ref{E-main-0iumbint}-c) hold when $n=p=1$. If $n,p>1$ then conditions 5 and 6 imply $H=\tH =0$. Suppose that $n>1$, $p=1$ and $H \neq 0$ and let $N \in {\cal D}$ be a local unit vector field. Then, evaluating \eqref{E-main-0iumbint} on $N \otimes N$,
we obtain
\begin{equation} \label{H2const}
\frac{n-1}{2n}\,\<H,H\> = \lambda.
\end{equation}
For $p=1$ we have $H = - (\Div N) N$ 
and $\int_M \tau_1 \,{\rm d} \vol_g =0$ for $\tau_1=\<H,N\>$, e.g., \cite{RWa-1}.
The integral formula shows that $\tau_1$ vanishes somewhere on $M$.
On the other hand, \eqref{H2const} yields that $\<H,H\> = \tau_1^2$ is constant on $M$, hence $H=0$.
Since $n>1$, condition 5 in Theorem~\ref{propstatcrit} implies also $\tH=0$.
\end{proof}

Equation \eqref{eqvarstat} and Corollary \ref{statisticalcritSmixI} imply the following
\begin{corollary}
Let $(g, \I)$ correspond to a statistical connection.
Then $(g, \I)$ is critical for the action \eqref{actiongSmix} with respect to all variations of metric and variations of $\I$
corresponding to statistical connections if and only if {\rm(\ref{ELSmixIstat1},b)} and equations of Theorem~\ref{T-main00} hold.
\end{corollary}

\subsection{Metric connections}
\label{sec: 2-3}


Here, we consider $g$ and $\I$ as independent variables in the action \eqref{actiongSmix}, hence for every pair $(g,\I)$ critical for \eqref{actiongSmix} the contorsion tensor $\I$ must be critical for \eqref{actiongSmix} with fixed $g$, and thus satisfy Corollary~\ref{T-main1}. 
Using this fact, we characterize those critical values of \eqref{actiongSmix}, that are attained on the set of pairs $(g, \I)$, where $\I$ is the contorsion tensor of a metric (in particular, adapted) connection for $g$.

%
\begin{proposition}\label{propQ}
Let the contorsion tensor $\I$ of a metric connection $\bar\nabla$ be critical for the action \eqref{actiongSmix} with fixed $g$.
Then $\widetilde{\mD}$ and $\mD$ are both totally umbilical and for $Q$ given in \eqref{E-defQ} we have
\begin{eqnarray} \label{criticalQmetricconnection}
&& \frac12\,Q = \frac{2n-1}{n}\,\<\tr^\top{\I}, H\> + \frac{2p-1}{p}\,\<\tr^\perp \I , \tH\> \nonumber \\
&& +\,\frac{p-1}{p}\,\<\tH,\tH\> + \frac{n-1}{n}\,\<H, H\>
+
\<T, T\> +
\< {\tilde T}, {\tilde T} \> .
\end{eqnarray}
\end{proposition}

\begin{proof}
By Corollary~\ref{T-main1}, both distributions are totally umbilical.
In this case, using (\ref{ELconnectionNew1}-j), we have
\begin{eqnarray*}
\<\tr^\top({\I}- {\I}^*), H -\tH\>
\eq 2\,\< \tr^\top{\I}, H\> -2\,\frac{p-1}{p}\,\<\tH,\tH\>,\\
\<\tr^\perp({\I} -{\I}^*),\, H -\tH\>
\eq 2\frac{n-1}{n}\,\< H, H\> -2\,\<\tr^\perp \I, \tH\>, \\
-\<\tr^\top{\I},\,\tr^\bot{\I}^*\>
\eq \frac{n-1}{n}\,\<\tr^\top \I, H\> +\frac{p-1}{p}\,\<\tr^\perp\I, \tH\> , \\
-\<\tr^\perp{\I}, \tr^\top{\I}^*\>
\eq \frac{p-1}{p}\,\<\tr^\perp \I, \tH\> +\frac{n-1}{n}\,\<\tr^\top \I, H\> .
\end{eqnarray*}
For totally umbilical distributions and critical metric connection, (\ref{ELconnectionNew1}-j)
yield 
\begin{eqnarray*}
-2 \< \I +\I^\wedge, A  \>
\eq 4
\<\tr^\top \I
, H\>,\\
-2 \< \I +\I^\wedge, \tA  \> \eq 4
\<
\tr^\bot \I , \tH\> , \\
%
\< \I +\I^\wedge, \T \> \eq 2 \sum \<\I_a {\cal E}_i +\I_i E_a, \T_i E_a\>
= 4 \<T, T\>,\\
\< \I +\I^\wedge, \tT \> \eq 4\< {\tilde T}, {\tilde T} \> , \\
\<{\I}^*, \I^\wedge \>_{\,|\,V} \eq  - \< \I, \I^\wedge \>_{\,|\,V}
= - 2 \sum \<\I_i E_a, \I_a {\cal E}_i\> = - 2\< T, T\> - 2\< {\tilde T}, {\tilde T} \>.
\end{eqnarray*}
%
Using the above in \eqref{E-defQ},
and simplifying the expression, completes the proof.
\end{proof}

\begin{remark}\rm
Let $n, p >1$. By \eqref{criticaltrIinlargedimensions}, for critical metric connection 
equation \eqref{criticalQmetricconnection} becomes 
\begin{equation*}
\frac12\,Q = -\< H, H\> -\<\tH, \tH\> +
\<T, T\> +
\< {\tilde T}, {\tilde T} \>.
\end{equation*}
By this and \eqref{E-Q1Q2-gen}, for any critical metric connection on a closed manifold $(M,g)$ we have
\begin{eqnarray*}
\int_M \overline{{\rm S}}_{\rm mix}\,{\rm d} \vol_g
&\overset{\eqref{eq-ran-ex}}=& \int_M \big(\frac{2n-1}{n} \<H,H\> +\frac{2p-1}{p} \<\tH,\tH\>
\big)\,{\rm d} \vol_g .
\end{eqnarray*}
Thus, the right hand side of the above equation is the only critical value of the action \eqref{actiongSmix} (with fixed $g$ on a closed manifold $M$) restricted to metric connections for $g$.
Notice that it does not depend on $\I$, but only on the pseudo-Riemannian geometry of distributions on $(M,g)$. Moreover, on a Riemannian manifold it is always non-negative. 
\end{remark}

Consider pairs $(g, \I)$, where $\I$ corresponds to a metric connection, critical for \eqref{actiongSmix} with respect to $g^\perp$-variations. We apply only adapted variations, as they will allow to obtain the Euler-Lagrange equations without explicit use of adapted frame or defining multiple new tensors. The~case of general variations, mostly due to complicated form of tensor $F$ defined by \eqref{formulaF} that appears in variation formulas, is significantly more involved and beyond the scope of this~paper.
Set
\begin{equation}\label{E-chi}
\chi = \sum\nolimits_{a,j} (\I_j E_a)^{\perp \flat} \odot (\tT_a {\cal E}_j)^{\perp \flat},\qquad
\phi(X,Y) =
(\I +\I^\wedge)_{X^\perp} Y^\perp .
\end{equation}
Define also $\phi^\top$ and $\phi^\perp$ by
$\phi^\top(X,Y) = (\phi(X,Y))^\top$
and
$\phi^\perp(X,Y) = (\phi(X,Y))^\perp$ for $X,Y \in \mathfrak{X}_M$.


\begin{theorem}
A pair $(g, \I)$, where
$\I$ corresponds to a metric connection on $M$, is critical for
\eqref{actiongSmix} with respect to $g^\perp$-variations of metric and arbitrary variations of $\,\I$ if and only if all the following conditions hold:
$\widetilde{\mD}$ and $\mD$ are totally umbilical, 
the following Euler-Lagrange equation~holds:
\begin{subequations}
\begin{eqnarray} \label{ELmetric}
 && -\frac{5n-5}{n} H^\flat \otimes H^\flat -\frac12\,\Upsilon_{T,T} +2\,\widetilde{\cal{T}}^\flat
 \nonumber\\
 && +\big(\frac{3p-3}{p}\Div \tH - \frac{2n-1}{n}\,\<\tr^\top {\I},\, H\>
 -\frac{2p-1}{p}\<\tr^\perp \I,\, \tH\>
 -\Div( (\tr^\perp \I )^\top) \big)\, g^{\perp} \nonumber \\
 && -2 \Div \phi^\top +\<\phi, \frac{3}{2}\,\tH -\frac{1}{2}\,H +\frac{1}{2}\,(\tr^\top \I)^\perp \> + 7 \chi
 + \frac{3n+2}{n}\,H^\flat \odot (\tr^\top \I)^{\perp \flat} =0,
\end{eqnarray}
$\I$ satisfies the following linear algebraic system:
\begin{eqnarray} \label{critcontorsion1}
 && (\I_V\, U -\I_U\, V)^\top = 2\,{\tilde T}(U, V), \\ \label{critcontorsiontTab}
 &&
 \I_U^\top = T^\sharp , \\ \label{trperpIperpH}
 && (\tr^\bot \I)^\perp = \frac{n-1}{n} H, \\ \label{critcontorsionTab}
 && (\I_Y\, X -\I_X\, Y)^\perp = 2\,T(X, Y), \\ \label{critcontorsionTij}
 &&
 \I_X^\bot = \tilde T^\sharp_X, \\
\label{critcontorsionlast}
 && (\tr^\top \I)^\top = \frac{p-1}{p} \tH,
\end{eqnarray}
for all $X,Y\in\widetilde\mD$ and $U,V\in\mD$; 
and
\begin{eqnarray}\label{critcontorsionspec1}
 (\tr^\top \I )^\perp = -H,\quad if\ p>1, \qquad
 (\tr^\perp \I )^\top = -{\tilde H},\quad if\ n>1.
\end{eqnarray}
\end{subequations}
\end{theorem}

\begin{proof}
By Corollary~\ref{T-main1},
$\I$ is critical for \eqref{actiongSmix} (with fixed $g$) if and only if distributions $\widetilde{\mD}$ and $\mD$ are totally umbilical and (\ref{critcontorsion1}-g) (together with \eqref{critcontorsionspec1} if their respective assumptions on $n$ and $p$ hold) are satisfied.
%
Let $\I$ be critical for the action \eqref{actiongSmix} with fixed $g$. We shall prove that
a pair $(g, \I)$ is critical for the action \eqref{actiongSmix} with respect to $g^\perp$-variations of metric if and only if \eqref{ELmetric} holds.

By Proposition~\ref{L-QQ-first},
for any variation $g_t$ of metric such that ${\rm supp}(B)\subset\Omega$, and $Q$ in \eqref{E-defQ}, we have
\[
 \frac{d}{dt}\int_M \big(2(\bar{\rm S}_{\,\rm mix}-{\rm S}_{\,\rm mix})+Q\big)\,{\rm d}\vol_g =
 \frac{d}{dt}\int_M (\Div X)\,{\rm d}\vol_g,
\]
where $X=(\tr^\top(\I -\I^*))^\bot +(\tr^\bot(\I -\I^*))^\top$.
Although $X$ is not necessarily zero on $\partial\Omega$, we have ${\rm supp}\,(\partial_t X)\subset\Omega$,
thus, $\frac{d}{dt}\int_M(\Div X)\,{\rm d}\vol_g=0$, see \eqref{E-DivThm-2}, and hence: 
\begin{eqnarray*}
 && \frac{d}{dt}\int_M (\,\bar{\rm S}_{\,\rm mix} - {\rm S}_{\,\rm mix})\,{\rm d}\vol_g
 = -\frac{1}{2} \int_M (\dt Q)\,{\rm d}\vol_g -\frac{1}{4} \int_M Q\,\< B, g \> \,{\rm d}\vol_g ,
\end{eqnarray*}
where, up to divergence of a compactly supported vector field, $\dt Q$ is given in Lemma~\ref{dtQadapted}.
For $g^\perp$-variations we get (see \cite[Eq.~(29]{rz-2} for more general case of $g^\pitchfork$-variations),
\begin{eqnarray*}
 \frac{d}{dt}\int_M {\rm S}_{\,\rm mix}\ {\rm d}\vol_g &=& \int_M \big\< -\Div{\tilde h} -\widetilde{\cal K}^\flat
 - H^\flat \otimes H^\flat +\frac{1}{2}\,\Upsilon_{h,h} +\frac12\,\Upsilon_{T,T} +2\,\widetilde{\cal{T}}^\flat  \\
\nonumber
 \plus\frac{1}{2}\,\big({{\rm S}}_{\,\rm mix} +\Div({\tilde H} - H)\big)\, g^{\perp},\ B\big\>\,{\rm d}\vol_g.
\end{eqnarray*}
For totally umbilical distributions we have
\[
 \widetilde{\cal K}^\flat =0,\quad
 \Div{\tilde h} = \frac{1}{p}\,(\Div{\tilde H})\, g^\perp,\quad
 \big\< \frac{1}{2}\,\Upsilon_{h,h}, B \big\> = \big\< \frac{1}{n}\,H^\flat \otimes H^\flat, B\big\> .
\]
Hence,
\begin{eqnarray*}
 && \frac{d}{dt}\int_M {\rm S}_{\,\rm mix}\, {\rm d}\vol_g
 = \int_M \big\< \frac12\,\Upsilon_{T,T} +2\,\widetilde{\cal{T}}^\flat -\frac{n-1}{n}\, H^\flat \otimes H^\flat \\
 \plus\frac{1}{2}\,\big({{\rm S}}_{\,\rm mix} +\Div(\frac{p-2}{p} {\tilde H} - H) -\frac{1}{2} Q\,\big)\, g^{\perp}
 +\frac{1}{2}\,\delta Q,\,B\big\>\, {\rm d} \vol_g ,
\end{eqnarray*}
where $\delta Q$ is defined by the equality $\<\delta Q, B \> = -\dt Q$, see Lemma~\ref{dtQadapted}.
Thus, the Euler-Lagrange equation for $g^\perp$-variations of metric and totally umbilical distributions is the following:
\begin{equation}\label{ELQ}
 -\frac{2n-2}{n} H^\flat \otimes H^\flat +\Upsilon_{T,T} + 4\,\widetilde{\cal{T}}^\flat
 +\big({{\rm S}}_{\,\rm mix} +\Div(\,\frac{p-2}{p} {\tilde H} - H) -\frac{1}{2}\,Q\, \big)\, g^{\perp} +\delta Q = 0 .
\end{equation}
Using Lemma~\ref{dtQadapted}, Proposition~\ref{propQ} and \eqref{E-PW-Smix-umb} in \eqref{ELQ}, we obtain
\begin{eqnarray*}
 && -\frac{5n-5}{n} H^\flat \otimes H^\flat -\frac12\,\Upsilon_{T,T} + 2\,\widetilde{\cal{T}}^\flat
 +\big(\,\frac{3p-3}{p} \Div \tH -\frac{2n-1}{n}\<\tr^\top {\I},\, H\> \\
 && -\frac{2p-1}{p}\,\<\tr^\perp \I,\, \tH\> -\Div((\tr^\perp \I )^\top)  \big  )\, g^{\perp} -2 \Div \phi^\top \\
 && +\< \phi,\, \frac{p+2}{p}\,\tH -\frac{1}{2} H +\frac{1}{2} \tr^\top \I \> + 7 \chi
 + \frac{3n+2}{n}\,H^\flat \odot (\tr^\top \I)^{\perp \flat} =0.
\end{eqnarray*}
By \eqref{critcontorsionlast}, from the above we get \eqref{ELmetric}.
\end{proof}

\begin{remark}\label{remarkvolpreserving}\rm
Note that for volume-preserving variations, the right hand sides of \eqref{ELmetric} and \eqref{ELQ} should be $\lambda\,g^\perp$, with $\lambda\in\mathbb{R}$ being an arbitrary constant \cite{rz-2}.
This obviously applies also to the special cases of the Euler-Lagrange equation \eqref{ELmetric} discussed below.

If $p>1$ and $n>1$ then
\eqref{ELmetric} can be written as
\begin{eqnarray}\label{ELmetricnpbig}
 && \frac{3-8n}{n} H^\flat \otimes H^\flat -\frac12\,\Upsilon_{T,T} +2\,\widetilde{\cal{T}}^\flat
 -2 \Div \phi^\top +\< \phi, \frac{3}{2} \tH - H \> + 7 \chi \nonumber \\
 && +\,\big(\frac{4p-3}{p} \Div \tH +\frac{2n-1}{n}\< H, H\> +\frac{2p-1}{p}\<\tH, \tH\>\big)\,g^{\perp} = 0 .
\end{eqnarray}
\end{remark}

Taking trace of \eqref{ELmetricnpbig} and using (\ref{trperpIperpH},g--i) and
equalities
 $\tr_g\Upsilon_{T,T} = 2\,\<T, T\>$
 and
 $\tr_g\widetilde{\cal{T}}^\flat = -\<{\tilde T}, {\tilde T}\>$,
we obtain the following result.

\begin{corollary}\label{corELtrace}
Let a pair $(g, \I)$, where $g$ is a pseudo-Riemannian metric on $M$ and $\I$ corresponds to a metric connection,
be critical for \eqref{actiongSmix} with respect to $g^\perp$-variations of metric and arbitrary variations of $\,\I$. Then
for $n,p>1$ we have 
\begin{equation}\label{ELmetrictracenpbig}
 \frac{(2n-1)(p-5)}{n}\,\< H,H \> -\<T, T\> -2\<{\tilde T}, {\tilde T}\> + (4p-1) \Div \tH + 2(p-2)\<\tH, \tH\> + 7\tr^\perp\chi = 0,
\end{equation}
and for $n=1$ and $p>1$ we get 
\begin{eqnarray*}
 && ({p-5})\<H,H\> -2\<{\tilde T}, {\tilde T}\> + 3(p-1)\Div\tH \nonumber \\
 && -(p+4)\Div((\tr^\perp \I)^\top) +2(2 - p)\<\tr^\perp \I,\, \tH \> + 7 \tr^\perp \chi =0.
\end{eqnarray*}

\end{corollary}

Recall that an \textit{adapted connection} to $(\mD,\widetilde\mD)$, see e.g., \cite{bf}, is defined by
\[
 {\bar \nabla}_Z\, X \in\mathfrak{X}^\bot,\quad
 {\bar \nabla}_Z\, Y \in \mathfrak{X}^\top,\quad
 X\in\mathfrak{X}^\bot,\ Y\in\mathfrak{X}^\top,\ Z \in\mathfrak{X}_M,
\]
and an example is the Schouten-Van Kampen connection with contorsion tensor
\[
 \I_{X} Y = -(\nabla_{X^\top}Y^\bot)^\top -(\nabla_{X^\top}Y^\top)^\bot -(\nabla_{X^\bot}Y^\bot)^\top -(\nabla_{X^\bot}Y^\top)^\bot,
 \quad X,Y \in\mathfrak{X}_M.
\]

%
\begin{proposition}\label{L-3-2old}
Let $\widetilde{\mD}$ and $\mD$ both be totally umbilical.
Then contorsion tensor $\I$ corresponding to an adapted metric connection satisfies {\rm (\ref{ELmetric}-i)} if and only if it satisfies the equations
\begin{subequations}
\begin{eqnarray}\label{adaptedcritconfirst}
&&
\I_U^\top = T^\sharp_U, \\ \label{trperpIperpadapted}
&& (\tr^\perp \I )^\perp = \frac{n-1}{n} H,  \\
&&
\I_X^\bot = \tT_X , \\
\label{adaptedcritconlast}
&& (\tr^\top \I )^\top = \frac{p-1}{p} {\tilde H},\\
\label{ELmetricadapted}
%
&& \frac{3-8n}{n}\,H^\flat \otimes H^\flat - \frac12\,\Upsilon_{T,T} -5\,\widetilde{\cal{T}}^\flat
- \< \phi, H \>  \nonumber \\
&& + \big(\,\frac{4p+1}{p}\, \Div \tH + \frac{2p-4}{p}\, \<\tH, \tH\> + \frac{2n-1}{n}\,\<H, H\>
\big) g^\perp  = 0,
\end{eqnarray}
\end{subequations}
for all $X\in\widetilde\mD$ and $U\in\mD$. 
\end{proposition}

\begin{proof}
For adapted connection and totally umbilical distribution $\mD$ we have
$\phi^\top = - 2 {\tilde h} = - \frac{2}{p} \tH g^\perp$, see \cite[Section~2.5]{RZconnection},
and 
\begin{eqnarray}\label{adaptedcontorsion}
\I_X Y \eq -(\nabla_{X^\top}\, Y^\top)^\bot -(\nabla_{X^\bot}\, Y^\bot)^\top
+\,({{A}}_{Y^\bot} + {{T}}^\sharp_{Y^\bot}) X^\top + ({\tilde A}_{Y^\top} + {\tilde T}^\sharp_{Y^\top})\, X^\bot
\nonumber \\
&& +\, (\I_X \,Y^\top)^\top + (\I_X \,Y^\bot)^\bot.
\end{eqnarray}
Moreover, an adapted connection is critical for \eqref{actiongSmix} with fixed $g$
if and only if (\ref{adaptedcritconfirst}-d) hold, see \cite{RZconnection}.
Note that for adapted connection from \eqref{adaptedcontorsion} we obtain 
$\chi = - \widetilde{\cal T}^\flat$, as for $X,Y \in \mD$ we have
\[
2\chi(X,Y) = \!\sum (2\<\tT_a{\cal E}_j, X\>\<\tT_a{\cal E}_j, Y\> + \<\tA_a{\cal E}_j, X\> \<\tT_a {\cal E}_j, Y\>
+ \<\tA_a{\cal E}_j, Y\> \<\tT_a{\cal E}_j, X\>) {=} -2\sum \<\tT_a\tT_a X, Y\>
\]
for umbilical distributions.
Also \eqref{critcontorsionspec1} hold, in all dimensions $n,p$. 
Thus, for a critical adapted connection, \eqref{ELmetric}
simplifies to \eqref{ELmetricadapted}.
\end{proof}

If $p>1$ then $\phi^\perp$ is not determined by $(\tr^\perp \I)^\perp$ and
by \eqref{adaptedcontorsion} in Proposition~\ref{L-3-2old}
can be set arbitrary for an adapted metric connection. Using this fact and taking trace of \eqref{ELmetricadapted} yield the following.

\begin{corollary}
Let $\widetilde{\mD}$ and $\mD$ both be totally umbilical.
If a contorsion tensor $\I$, corresponding to an adapted metric connection, satisfies {\rm(\ref{ELmetric}-i)} then the metric $g$ satisfies
\begin{equation}\label{trELmetricadaptedcrit}
 \frac{ 5 -10n + 2np -p}{n}\,\< H,H \> - \< T,T\> + 5 \< \tilde T , \tilde T \>
 + (4p+1) \Div \tH  + (2p-4) \<\tH, \tH\> =0.
\end{equation}
If $p>1$ and at every point of $M$ we have $H \ne 0$, then for a given $(M,g)$ satisfying \eqref{trELmetricadaptedcrit} there exists a metric adapted connection such that $(g, \I)$ is critical for the action \eqref{actiongSmix} with respect to all variations of $\,\I$ and $g^\perp$-variations of metric.
\end{corollary}

\begin{corollary}
Let $(M,g)$ be a closed Riemannian 
manifold endowed with $\mD$ integrable and $\widetilde{\mD}$ integrable and totally geodesic, and let $p \neq 2$.
Then there exists a metric compatible adapted connection such that $(g, \I)$ is critical for the action \eqref{actiongSmix} with respect to all variations of $\,\I$ and $g^\perp$-variations of metric if and only if $\mD$ is totally geodesic.
\end{corollary}

\begin{proof}
Under these assumptions we obtain that \eqref{ELmetricadapted} holds if and only if
\[
 (4p+1)\Div {\tilde H} + (2p-4) \<\tH, \tH\> =0.
\]
Integrating this equation on a closed $(M,g)$
and using \eqref{E-DivThm} yields $\tH=0$.
\end{proof}

\begin{example} \rm
In \cite{FriedrichIvanov} it was proved that on a Sasaki manifold $(M, g, \xi , \eta)$
(that is $M$ with a normal contact metric structure)
there exists a unique metric connection with a skew-symmetric, parallel torsion tensor, and its contorsion tensor is given by
 $\<\I_X Y, Z\> = \frac{1}{2}\,(\eta \wedge d \eta)(X,Y,Z)$, where
 $X,Y,Z \in \mathfrak{X}_M$
and $\eta$ is the contact form on $M$. Let $\widetilde{\mD}$ be the one-dimensional distribution spanned by the Reeb field $\xi$.
It follows that for this connection we have $\phi =0$ and for $X,Y \in \mD$
\begin{equation*}
 \chi(X,Y) = -\frac{1}{4}\sum\nolimits_{\,i}\big[\,(\eta\wedge d\eta)(\xi, {\cal E}_i ,X)\cdot\<\,\tT_\xi {\cal E}_i, Y\>
 + (\eta\wedge d\eta)(\xi, {\cal E}_i, Y)\cdot \<\,\tT_\xi{\cal E}_i, X\>\,\big]
 = -\,\widetilde{\cal{T}}^\flat(X,Y),
\end{equation*}
see \eqref{E-chi}, as $d \eta(X,Y) = 2\<X, \tT_\xi\, Y\>$.
Since $g$ is a Sasaki metric, both distributions are totally geodesic, and for volume-preserving variations the Euler-Lagrange equation \eqref{ELmetric} gets $\lambda\,g^\perp$ on the right-hand side (see Remark~\ref{remarkvolpreserving}) and becomes
\begin{equation}\label{ELforFriedrichIvanovConnection}
 -5\,\widetilde{\cal{T}}^\flat
 = \lambda\,g^\perp .
\end{equation}
As on a Sasakian manifolds we have
$\widetilde{\cal{T}}^\flat=-\frac{1}{p}\,\<{\tilde T},{\tilde T}\>g^{\perp}$ and $\<{\tilde T},{\tilde T}\>=p$
(e.g., \cite[Section 3.3]{rz-2}), we see that \eqref{ELforFriedrichIvanovConnection} holds in this case for $\lambda=5$.
We can slightly modify this example to obtain a critical metric connection on any contact manifold $(M, \eta)$ with a contact metric structure $g$, by taking $\I_\xi \xi =0$ and for all $X,Y \in{\cal D}$:
\[
 \<\I_X Y, \xi\> = \frac{1}{2}\,(\eta \wedge d \eta)(X,Y,Z) = -\<\I_X \xi, Y\>,\quad
  \<\I_\xi X, Y\> = -\frac{1}{2}\,(\eta\wedge d\eta)(\xi,X, Y).
\]
For all $X,Y,Z\in{\cal D}$ we can take as $\<\I_X Y, Z\>$ any 3-form.
While no longer with parallel torsion, connection $\nabla + \I$ will then satisfy all Euler-Lagrange equations (\ref{ELmetric}-i).
\end{example}

Corollary \ref{corELtrace} can be viewed as an integrability condition for \eqref{ELmetric}. Below we give examples of $\I$, constructed for metrics $g$ that satisfy \eqref{ELmetrictracenpbig} with particular form of $\chi$, obtaining pairs $(g,\I)$ that are critical points of \eqref{actiongSmix} with respect to variations of $\,\I$ and $g^\perp$-variations of metric.

\begin{proposition} \label{propexample1}
Let $n,p>1$ and $H \neq 0$ everywhere on $M$. For any $g$ such that $\widetilde{\mD}$ and $\mD$ are totally umbilical and \eqref{ELmetrictracenpbig} holds with $\chi=0$, there exists a contorsion tensor $\I$ such that
$\I_X Y \in \mathfrak{X}^\bot$ for all $X,Y\in\mathfrak{X}^\bot$ and $(g, \I)$ is critical for the action \eqref{actiongSmix} with respect to
$g^\perp$-variations of metric and arbitrary variations of $\,\I$. 
\end{proposition}

\begin{proof}
Suppose that $\I_X Y \in \mathfrak{X}^\bot$ for all $X,Y \in \mathfrak{X}^\bot$. Then $\phi^\top = 0$, $\chi =0$,
see definitions \eqref{E-chi} (because $\<\I_j E_a, {\cal E}_i\>=-\<\I_j {\cal E}_i, E_a\>=0$),
$(\tr^\perp \I )^\top=0$, from equations for critical connections it follows that ${\mD}$ is integrable
and \eqref{ELmetric} is an algebraic equation for symmetric (0,2)-tensor $\phi$:
\begin{equation}\label{ELmetricspec1}
 -\frac{8n-3}{n}\,H^\flat \otimes H^\flat
 +\big(\frac{3p-3}{p} \Div \tH + \frac{2n-1}{n}\,\< H, H\>
 \big)\, g^{\perp}
 -\frac12\,\Upsilon_{T,T} -\< \phi, H \>
 =0.
\end{equation}
For $H \ne 0$, we can always find $\phi$ (and then $\I$) satisfying \eqref{ELmetricspec1}. Clearly, such $\phi$ is not unique.
\end{proof}


\begin{proposition} \label{propexample2}
Let $n,p>1$ and $H \neq 0$ everywhere on $M$.
For any $g$ such that $\widetilde{\mD}$ is totally umbilical and $\mD$ is totally geodesic and \eqref{ELmetrictracenpbig} holds with $\chi=-\widetilde{\cal T}^\flat$, there exists a contorsion tensor $\I$ such that
$(\I_X\,\xi)^\perp = \tT_\xi X$ for all $X \in \mathfrak{X}^\bot$, $\xi \in \mathfrak{X}^\top$, and a pair $(g,\, \I)$ is critical for the action \eqref{actiongSmix} with respect to $g^\perp$-variations of metric and arbitrary variations of $\,\I$.
\end{proposition}

\begin{proof}
For $(\I_i E_a)^\perp = \tT_a {\cal E}_i$ we have for $X,Y \in \mathfrak{X}^\bot$:
\begin{equation*}
 \chi(X,Y) =
 \sum\nolimits_{a,j}\<\tT_a {\cal E}_j, X\>\<\tT_a {\cal E}_j,Y\>
 = -\widetilde{\cal T}^\flat (X,Y).
\end{equation*}
Then, since
 $\<\I_{i}\,{\cal E}_i, E_a\> = -\<\I_{i}\,E_a, {\cal E}_i\> = -\<\tT_a {\cal E}_i, {\cal E}_i\> =0$,
we also get $(\tr^\perp \I )^\top=0 = \tH$ and similarly, $\phi^\top=0$. So, \eqref{ELmetric} has the following form:
\begin{equation}
\label{ELmetricspec2a}
 -\frac{8n-3}{n}\,H^\flat \otimes H^\flat
 -\frac12\,\Upsilon_{T,T} - 5\,\widetilde{\cal{T}}^\flat
 +
 \frac{2n-1}{n}\,\<H, H\>
 \, g^{\perp}
 -\< \phi, H \>
 =0,
\end{equation}
Again, we get an algebraic equation for symmetric tensor $\phi$, which admits many solutions.
\end{proof}

Note that in Propostions \ref{propexample1} and \ref{propexample2} instead of condition $H\ne 0$ everywhere on $M$, we can assume that at those points of $M$, where $H=0$ the metric $g$ satisfies \eqref{ELmetricspec1} and \eqref{ELmetricspec2a} with $H=0$ (then these equations do not contain $\phi$).

\begin{example}\rm
Let
$\widetilde{\mD}$ and $\mD$ be totally umbilical, $n,p>1$, ${\cal D}$ integrable and \eqref{ELmetrictracenpbig} hold.
Then $\chi=0$ holds, since ${\cal D}$ is integrable, so \eqref{ELmetrictracenpbig} does not contain any components of
$\I$.
 With these assumptions we can construct a simple example of $\I$ that satisfies the Euler-Lagrange equations (\ref{critcontorsion1}-i)
and \eqref{ELmetricnpbig} in some domain. Let $U$ be a 
neighborhood of $p \in M$; we choose any 
local adapted orthonormal frame $(E_a , {\cal E}_i)$ on~$U$.
Then, due to $\phi(X,Y) = \phi(X^\perp,  Y^\perp)$, we have
\begin{eqnarray*}
 && (\Div \phi^\top)({\cal E}_i, {\cal E}_j) = \sum\nolimits_a\<\nabla_{E_a}(\phi^\top({\cal E}_i, {\cal E}_j)), E_a\>
 +\sum\nolimits_k \<\nabla_{ {\cal E}_k } ( \phi^\top ({\cal E}_i, {\cal E}_j)  ), {\cal E}_k \> \nonumber \\
 && - \sum\nolimits_{a,m} \<\phi^\top( {\cal E}_i , {\cal E}_m ) , E_a \> \< \nabla_{E_a} {\cal E}_j , {\cal E}_m \>
 - \sum\nolimits_{a,m} \< \phi^\top( {\cal E}_m , {\cal E}_j ) , E_a \> \< \nabla_{E_a} {\cal E}_i , {\cal E}_m \> \nonumber \\
 && -\sum\nolimits_{k,m} \< \phi^\top( {\cal E}_i , {\cal E}_m ) , {\cal E}_k \> \< \nabla_{{\cal E}_k} {\cal E}_j , {\cal E}_m \> - \sum\nolimits_{k,m} \< \phi^\top( {\cal E}_m , {\cal E}_j ) , {\cal E}_k \> \< \nabla_{{\cal E}_k} {\cal E}_i , {\cal E}_m \> . 
\end{eqnarray*}
We define components of $\I$ with respect to the adapted frame on $U$. Let $( \I_i {\cal E}_j - \I_j {\cal E}_i )^\top =0$ for $i \neq j$ and let $(\I_i E_a)^\top$, $(\I_a E_b)^\perp$ and $(\I_a {\cal E}_i)^\perp$ be such that (\ref{critcontorsiontTab},e,f,h) hold on $U$.
%
%
For all $( i,j ) \neq (p,p)$, consider \eqref{ELmetricnpbig} evaluated on $({\cal E}_i , {\cal E}_j)$ as a system of linear, non-homogeneous, first-order PDEs for $\{ \phi({\cal E}_i , {\cal E}_j) , ( i,j ) \neq (p,p) \}$, assume in this system that
 $\phi({\cal E}_p , {\cal E}_p) = \frac{n-1}{n}\, H - \tH - \sum\nolimits_{\,i=1}^{p-1} \phi({\cal E}_i , {\cal E}_i)$,
and let $\{ \phi_{ij} , ( i,j ) \neq (p,p) \}$ be any local solution of this 
system of PDEs on (a subset of) $U$. Let $\I_i {\cal E}_j + \I_j {\cal E}_i = \phi_{ij}$ for $(i,j) \neq (p,p)$ and let $\I_p {\cal E}_p = \frac{1}{2} (\frac{n-1}{n} H - \tH - \sum_{i=1}^{p-1} \phi_{ii})$, then 
(\ref{trperpIperpH},i)
hold. By the assumption that \eqref{ELmetrictracenpbig} holds and the fact that \eqref{ELmetricnpbig} is a linear, non-homogeneous equation for $\phi$, \eqref{ELmetricnpbig} evaluated on $({\cal E}_p , {\cal E}_p)$ will also be satisfied. Thus, equations (\ref{critcontorsion1}-i)
and \eqref{ELmetricnpbig} hold on (a subset of) $U$ for $\I$ constructed above.
\end{example}

Note that when we consider adapted variations, we also have the equation dual (with respect to interchanging $\widetilde{\mD}$ and $\mD$) to \eqref{ELmetric}, so we can mix different assumptions from the above examples for different distributions,
e.g., conditions
$(\I_i\, E_a)^\perp = \tT_a\,{\cal E}_i$ and
$\I_X Y \in \mathfrak{X}^\top$ for $X,Y \in \mathfrak{X}^\top$.

\subsection{Semi-symmetric connections}
\label{sec:2-4}

The following connections are metric compatible, see \cite{Yano}.
Using variations of $\I$ in this class, we obtain example with explicitly given tensor $\overline\Ric_\mD$.

\begin{definition}
\rm
An affine connection $\bar\nabla$ on $M$ is \textit{semi-symmetric} if its torsion tensor $S$ satisfies
$S(X,Y)=\omega(Y)X-\omega(X)Y$, where $\omega$ is a one-form on $M$.
For $(M,g)$ we have
\begin{equation}\label{Uconnection}
\bar\nabla_XY=\nabla_XY + \<U , Y\> X -\<X,Y\>U,
\end{equation}
where $U=\omega^\sharp$ is the dual vector field.
\end{definition}

We find Euler--Lagrange equations of \eqref{Eq-Smix} as a particular case of (\ref{ELconnection1}-h),
using variations of $\I$ corresponding to semi-symmetric connections.
Now we consider variations of a semi-symmetric connection 
only among connections also satisfying \eqref{Uconnection} for some $U$. 

\begin{proposition}\label{corUcriticalforI}
A semi-symmetric connection ${\bar\nabla}$ on $(M,g,\mD)$ satisfying \eqref{Uconnection} 
is critical for the action \eqref{Eq-Smix} with fixed $g$ among all 
semi-symmetric connections if and only if
\begin{equation}\label{UcriticalforI}
 2p(n-1)\,U^\top - (n-p) \tH = -(\mathfrak{a}/2)\,s^\top,\quad
 2n(p-1)\,U^\perp - (p-n) H = -(\mathfrak{a}/2)\,s^\bot,
\end{equation}
where $s^\top=(s(\cdot\,,\cdot))^\top$ and $s^\bot=(s(\cdot\,,\cdot))^\bot$.
In particular, if $\,n=p=1$ and $s=0$ (no spin) then every semi-symmetric connection is critical among all such connections,
because $Q=0$ in this case.
\end{proposition}

\begin{proof}
Let $U_t,\ t\in(-\epsilon, \epsilon)$, be a family of compactly supported vector fields on $M$,
and let $U=U_0$ and $\dot U = \dt U_t |_{t=0}$. Then for a fixed metric $g$,
from \eqref{QforUconnection} we obtain
\[
 \dt Q(U_t) |_{t=0} = (p-n)\< \dot U, \tH\> + 2p(n-1) \< U^\top , \dot U \> + \< \dot U , H \> (n-p) + 2n(p-1) \< U^\perp , \dot U \>.
\]
Separating parts with $(\dot U)^\top$ and $(\dot U)^\perp$, we get
\[
 \dt Q(U_t) |_{\,t=0} = \< \dot U,\  (p-n) \tH + 2p(n-1) U^\top \> + \< \dot U,\ (n-p) H + 2n(p-1) U^\perp \>,
\]
from which \eqref{UcriticalforI} follow.
\end{proof}

\begin{remark}\rm
By Lemma \ref{lemmasemisymmetric}, if a semi-symmetric connection ${\bar\nabla}$ on $(M,g,\mD)$ is critical for the action \eqref{actiongSmix} with fixed $g$, then both $\widetilde{\cal D}$ and ${\cal D}$ are integrable and totally geodesic.
Indeed, let ${\bar \nabla}$ be given by \eqref{Uconnection} and satisfy (\ref{critcontorsion1}-g) and conditions \eqref{critcontorsionspec1},
i.e., it is critical for action \eqref{actiongSmix} with fixed $g$. We find from \eqref{UconnectionImixed} that both $\widetilde{\cal D}$ and $\cal D$ are integrable. Moreover, if $n=p=1$ then \eqref{UconnectiontrtopI} and its dual with (\ref{critcontorsion1}-g) yield $H = 0 = \tH$ and $U=0$ (i.e., the connection ${\bar \nabla}$ becomes the Levi-Civita connection).
If $n>2$ and $p>2$ we also have $H = 0 = \tH$ and $U=0$, in this case using also \eqref{critcontorsionspec1}.
If $n=1$ and $p>1$ we obtain from \eqref{trperpIperpH} that $U^\perp =0$ and from \eqref{critcontorsionspec1}$_1$ that $H=0$, moreover 
as both distributions are totally umbilical by Corollary~\ref{T-main1}, it follows that they are totally geodesic.
\end{remark}

\begin{theorem}\label{propUconnectionEL}
A pair $(g, \I)$, where $g\in{\rm Riem}(M,\widetilde{\mD},{\mD})$ and $\I$ corresponds to a semi-symmetric connection on $M$ defined by \eqref{Uconnection}, is critical for \eqref{actiongSmix} with respect to volume-preserving $g^\pitchfork$-variations of metric and variations of $\,\I$ corresponding to semi-symmetric connections
%
if and only if the following Euler-Lagrange equations are satisfied:
\begin{subequations}
\begin{eqnarray}\label{UELD}
 && {r}_{\mD} -\<\tilde h,\,\tilde H\> +\widetilde{\cal A}^\flat -\widetilde{\cal T}^\flat
 +\Psi  +\widetilde{\cal K}^\flat -{\rm Def}_{\mD}\,H
 + H^\flat\otimes H^\flat -\frac{1}{2}\,\Upsilon_{h,h} -\frac12\,\Upsilon_{T,T} \\
\nonumber
 && -\frac12\,\big({\rm S}_{\rm mix} +\Div(\tilde H -H)\big)\,g^\perp
 -\frac14\,(p-n)( \Div U^\top)\,g^\perp
 +\frac12\,n(p-1) U^{\perp \flat} \otimes U^{\perp \flat} =\lambda\,g^\perp, \\
\label{UELmixed}
 && 4\,\<\theta,\,{\tilde H}\>  +2(\Div(\alpha -\tilde\theta) )_{\,| {\rm V}}
 +2\<{\tilde \theta} - {\tilde\alpha}, H\> + 2\,H^{\flat} \odot {\tilde H}^{\flat} -2\,{\tilde \delta}_{H}
  +4\,\Upsilon_{{\tilde\alpha}, \theta} +2\Upsilon_{\alpha, {\tilde\alpha}}
 +2\,\Upsilon_{{\tilde \theta}, \theta} \nonumber \\
 && +\,\frac12\,(n-p){\tilde \delta}_{U^\perp} + \frac12\,(n-p)\< {\tilde \alpha} - {\tilde \theta} , U^\perp \>
 -(p-n) \< {\theta} , U^\top \> - p(n-1) U^{\top \flat} \otimes U^{\perp \flat} = 0,
\end{eqnarray}
\end{subequations}
and
\begin{equation}\label{UcriticalforI-0}
 2p(n-1)\,U^\top - (n-p) \tH = 0,\quad
 2n(p-1)\,U^\perp - (p-n) H = 0.
\end{equation}
\end{theorem}


\begin{proof} By Proposition~\ref{L-QQ-first} and \eqref{dtQgforUconnection}, we obtain
\begin{eqnarray*}
 \dt \int_M (\bar{{\rm S}}_{\,\rm mix} - {\rm S}_{\,\rm mix})\,{\rm d}\vol_g \eq
  \int_M \big\<
  \frac14\,(p-n)(\Div U^\top ) g^\perp -(p-n)\< {\theta}, U^\top \> \nonumber \\
 &-& \!\!\frac12\,n(p-1) U^{\perp \flat} \otimes U^{\perp \flat}
 -p(n-1) U^{\top \flat} \otimes U^{\perp \flat} ,\ B\big\>\,{\rm d}\vol_g .
\end{eqnarray*}
Using (\ref{E-main-0i},b) give rise to (\ref{UELD},b).
Finally, notice that \eqref{UcriticalforI-0} is \eqref{UcriticalforI} for vacuum space-time.
\end{proof}

Although generally $\overline\Ric_{\,\mD}$ in \eqref{E-gravity-gen} has a long expression and is not given here,
for particular case of  semi-symmetric connections,
due to Theorem~\ref{propUconnectionEL}, we present the mixed Ricci tensor explicitly~as
\begin{equation}\label{E-Ric-D-semi-sym}
\left\{\begin{array}{c}
\overline\Ric_{\,\mD\,|\,\mD\times\mD} = \Ric_{\,\mD\,|\,\mD\times\mD}
+\frac12\,n(p-1) U^{\perp \flat} \otimes U^{\perp \flat}
-\frac14\,(p-n)( \Div U^\top)\,g^\perp + \frac{Z}{2-n-p}\,g^\perp, \\
\overline\Ric_{\mD\,|\,V} = \Ric_{\mD\,|\,V}
-\frac12\,(n-p)\big({\tilde \delta}_{U^\perp} +\< {\tilde \alpha} - {\tilde\theta}, U^\perp \>\big)
+(p-n)\<\theta, U^\top\> + p(n-1) U^{\top\flat}\otimes U^{\perp\flat},\\
\overline\Ric_{\,\mD|\,\widetilde\mD\times\widetilde\mD} = \Ric_{\,\mD|\,\widetilde\mD\times\widetilde\mD}
+\frac12\,p(n-1) U^{\top\flat} \otimes U^{\top\flat} -\frac14\,(n-p)(\Div U^\bot)\,g^\top + \frac{Z}{2-n-p}\,g^\top,
\end{array} \right.
\end{equation}
also $\overline{\rm S}_{\mD}=\tr_g\overline\Ric_{\,\mD} = {\rm S}_{\mD} + \frac{2}{2-n-p}\,Z$,
where $\Ric_{\,\mD}$ and ${\rm S}_{\mD}$ as in Definition~\ref{D-Ric-D}, $n+p>2$ and
\begin{equation*}
 Z=\frac12\,n(p-1)\|U^{\perp}\|^2 +\frac12\,p(n-1)\|U^{\top}\|^2 -\frac14\,p(p-n)\Div U^\top -\frac14\,n(n-p)\Div U^\bot.
\end{equation*}
This is because $\overline\Ric_{\,\mD} - \frac{1}{2} \tr (\overline\Ric_{\,\mD}) g = 0$ is equivalent to all three Euler-Lagrange equations for \eqref{actiongSmix}.

\begin{example}\rm
For a \textit{space-time} $(M^{p+1},g)$ endowed with ${\widetilde\mD}$ spanned by a timelike unit vector field $N$,
see Example~\ref{Ex-2-1}, the tensor $\overline\Ric_{\mD}$ has the following particular form
(i.e., \eqref{E-Ric-D-semi-sym} with $n=1$):
\begin{equation*}
\left\{\begin{array}{c}
 \overline\Ric_{\,\mD\,|\,\mD\times\mD} = \Ric_{\,\mD\,|\,\mD\times\mD}
 +\frac12\,(p-1) U^{\perp \flat}\otimes U^{\perp\flat}
 -\frac14\,(p-1)(\Div U^\top)\,g^\perp + \frac{Z}{1-p}\,g^\perp, \\
 \overline\Ric_{\mD\,|\,V} = \Ric_{\mD\,|\,V}
 -\frac12\,(1-p)\big({\tilde \delta}_{U^\perp} +\< {\tilde \alpha} - {\tilde\theta}, U^\perp \>\big)
 ,\\
 \overline\Ric_{\,\mD\,|\,\widetilde\mD\times\widetilde\mD} = \Ric_{\,\mD|\,\widetilde\mD\times\widetilde\mD}
 -\frac14\,\eps_N(1-p)(\Div U^\bot) + \eps_N\frac{Z}{1-p},
\end{array} \right.
\end{equation*}
and $\overline{\rm S}_{\,\mD} = {\rm S}_{\,\mD}
+\frac{2\,\eps_N Z}{1-p}$, see \eqref{E-RicD-flow-S},
where
$Z=\frac14\,(p-1)\big(\,2\,\|U^{\perp}\|^2 -p\,\Div U^\top +\Div U^\bot\big)$.
Note that $\theta=0$ and $2\,{\tilde\delta}_{U^\perp}(N,\cdot)=(\nabla_N\,(U^\bot))^{\bot\flat}$.
\end{example}

\begin{remark}\rm
 By Proposition~\ref{corUcriticalforI}, also \eqref{UcriticalforI} holds, which allows to simplify the Euler-Lagrange equations of Theorem~\ref{propUconnectionEL} as discussed below.
If $n=p=1$ then \eqref{UcriticalforI} does not give any restrictions for $U$ and all terms containing $U$ vanish in (\ref{UELD},b) -- as expected from the last sentence in Proposition~\ref{corUcriticalforI}.

If $n=1$ and $p>1$ then by \eqref{UcriticalforI} we have $\tH=0$ and $U^\perp=\frac{1}{2}H$, while $U^\top$ can be arbitrary.
We~also have $-\frac{1}{2} \Upsilon_{h,h} = - H^\flat \otimes H^\flat$, and \eqref{UELD} becomes
\begin{equation*}
 -\Div{\tilde h} -\widetilde{\cal K}^\flat
 +2\,\widetilde{\cal{T}}^\flat +\frac{1}{2}\,\big({{\rm S}}_{\,\rm mix} +\Div({\tilde H} - H)
  + \frac{p-1}{2}\,\Div U^\top\big)\,g^\perp - \frac{p-1}{4}\,H^{\flat} \otimes H^{\flat} = \lambda\,g^\perp,
\end{equation*}
where we replaced $r_\mD$ by $\Div\tilde h$ (with additional terms) according to \eqref{E-genRicN},
and for \eqref{UELmixed} we have
\begin{equation}\label{UELmixedn1}
 2\,(\Div(\alpha -\tilde\theta) )_{\,| {\rm V}} + \frac{7+p}{4} \<{\tilde \theta} - {\tilde\alpha}, H\>
 -\frac{7+p}{4} \, {\tilde \delta}_{H} + 2\Upsilon_{\alpha, {\tilde\alpha}} = 0.
\end{equation}
Let $N \in \widetilde{\cal D}$ and $X \in {\cal D}$.
Using results and notation from \cite{rz-2}, we have the following:
\begin{eqnarray*}
 2 (\Div {\tilde\theta})(X,N) \eq (\Div \tT_N)(X) + \<\tT_N H, X\>, \\
 2 (\Div \alpha)(X, N) \eq \<\nabla_N H - {\tilde \tau}_1 H, X\>, \\
 2 \Upsilon_{\alpha, {\tilde\alpha}}(X,N) \eq \<\tA_N H, X\>, \\
 2 {\tilde\delta}_H (X,N) \eq \<\nabla_N H, X\>, \\
 2 \<{\tilde \theta} -{\tilde \alpha}, H \> (X,N) \eq -\<\tT_N H + \tA_N H, X\> ,
\end{eqnarray*}
where ${\tilde \tau}_1 = \tr\tA_N$.
Hence, \eqref{UELmixedn1} holds if and only if for unit $N\in\widetilde{\cal D}$ and all $X \in {\cal D}$ we have
\begin{eqnarray*}
 \frac{1-p}{8}\,\<\nabla_N H, X\> - \<{\tilde \tau}_1 H, X\> - (\Div \tT_N)(X) - \frac{15+p}{8}\,\<\tT_N H, X\>
 +\frac{1-p}{8}\,\<\tA_N H, X\>  = 0 .
\end{eqnarray*}
 If $n>1$ and $p>1$, then using \eqref{UcriticalforI} we reduce \eqref{UELD} to the following:
\begin{eqnarray*}
 && -\Div{\tilde h} -\widetilde{\cal K}^\flat
 +\frac{1}{2}\,\Upsilon_{h,h}+\frac12\,\Upsilon_{T,T} +2\,\widetilde{\cal{T}}^\flat
 +\frac{1}{2}\,\big({{\rm S}}_{\,\rm mix} +\Div({\tilde H} - H)\big)\,g^{\perp} \nonumber \\
 && - \frac{(p-n)^2}{8p(n-1)}\,(\Div \tH )\,g^\perp
 - \frac{(p-n)^2 + 8n(p-1)}{8n(p-1)}\,H^\flat \otimes H^\flat = \lambda\,g^\perp,
\end{eqnarray*}
and we reduce \eqref{UELmixed} to the following:
\begin{eqnarray*}
 && 4\,\Upsilon_{{\tilde\alpha}, \theta} +2(\Div(\alpha -\tilde\theta))_{\,|\,{\rm V}} +2\,\Upsilon_{\alpha, {\tilde\alpha}}
 +2\,\Upsilon_{{\tilde \theta}, \theta} -\frac{(p-n)^2 + 8n(p-1)}{4n(p-1)}\,{\tilde \delta}_H \\
 &&\hskip-17mm -\frac{(p-n)^2 + 8n(p-1)}{4n(p-1)}\,\<{\tilde \alpha} - {\tilde\theta}, H \>
 +\frac{(n-p)^2 + 8p(n-1)}{2p(n-1)}\,\< \theta, \tH \> + \frac{(p-n)^2}{4n(p-1)}\,H^\flat \odot \tH^\flat = 0.
\end{eqnarray*}
Note that for vacuum space-time the distributions $\widetilde{\cal D}$ and ${\cal D}$ don't need to be umbilical 
to admit $(g , \I)$ critical for \eqref{actiongSmix} among all metrics and semi-symmetric connections.
\end{remark}

\section{Auxiliary lemmas}
\label{sec:aux}


\begin{lemma}\label{L-divX}
For any variation $g_t$ of metric and a $t$-dependent vector field $X$ on $M$, we have
\begin{equation*}
 \dt\,(\Div X) = \Div (\dt X) +\frac{1}{2}\,X(\tr_{g} B).
\end{equation*}
\end{lemma}

\begin{proof}
Differentiating the formula \eqref{eq:div} and using \eqref{E-dotvolg}, we get
\begin{eqnarray*}
&& \dt\big((\Div X)\,{\rm d}\vol_g\big) = \big(\dt\,(\Div X) +\frac12\,(\Div X)\tr_{g} B\big)\,{\rm d}\vol_{g},\\
&& \dt\big({\cal L}_{X}({\rm d}\vol_g)\big) = \big(\Div (\dt X) +\frac{1}{2}\,X(\tr_{g} B)+\frac12\,(\Div X)\tr_{g} B\big)\,{\rm d}\vol_{g}.
\end{eqnarray*}
From this the
claim follows.
\end{proof}

Define symmetric $(1,2)$-tensors $L,G,F$, 
by the following formulas:
\begin{eqnarray}
\nonumber
L(X,Y) \eq \frac{1}{4} (\Theta^*_{X^\perp} Y^\perp +\Theta^{\wedge*}_{X^\perp} Y^\perp +
\Theta^*_{Y^\perp} X^\perp +\Theta^{\wedge*}_{Y^\perp} X^\perp) ,\\
\nonumber
 G(X,Y) \eq \frac{1}{4}({\Theta}^*_{ X^\perp} Y^\top + {\Theta}^{\wedge*}_{ X^\perp} Y^\top + {\Theta}^{\wedge*}_{Y^\perp} X^\top + {\Theta}^*_{ Y^\perp} X^\top ),\\
\nonumber
 F(X,Y) \eq \frac{1}{4}(\Theta^*_{X^\top} Y^\perp \!+\Theta^{\wedge*}_{X^\top} Y^\perp \!- \Theta_{X^\top} Y^\perp
 \!- \Theta^\wedge_{X^\top} Y^\perp \! \\ \label{formulaF}
 &&  +\Theta^*_{ Y^\top} X^\perp \!+\Theta^{\wedge*}_{Y^\top}X^\perp \!- \Theta_{Y^\top} X^\perp
 \!- \Theta^\wedge_{Y^\top} X^\perp),
\end{eqnarray}
where $\Theta = \I -\I^* +\I^\wedge - \I^{* \wedge }$ and $(\Theta^\wedge)_X Y = \Theta_Y X$ for all $X,Y \in \mathfrak{X}_M$.

The~following equalities
(and similar formulas for $\Upsilon_{\alpha,  {\tilde \alpha}}$, $\Upsilon_{\theta,  {\tilde \alpha}}$, etc.)
will be used (recall Remark~\ref{remarkepsilons} for notational conventions):
\begin{eqnarray*}
 \<\,\<\alpha, {\tilde H}\>, S\> \eq \sum\nolimits_{\,a,i} 
  \<A_{i}(E_a), {\tilde H}\> S(E_a,{\cal E}_i),\quad
 \<\Upsilon_{\alpha, \theta}, S\> = \sum\nolimits_{\,a,i} 
 S(A_{i}(E_a), T^{\sharp}_i(E_a)) ,\\
  \Upsilon_{\alpha, {\tilde \theta}}(X,Y) \eq \frac{1}{2}\sum\nolimits_{a,i} 
  \<X, A_{ i} E_a\>\,
  \< Y, {\tilde T}^\sharp_{a} {\cal E}_i \>,\quad X \in \mathfrak{X}^\top,\ \ Y \in \mathfrak{X}^\bot .
\end{eqnarray*}
The variations of components of $Q$ in \eqref{E-defQ} (used in previous sections) are collected in the following three lemmas;
the results for $g^\top$ variations are dual to $g^\bot$-parts in results for $g^\pitchfork$-variations.

\begin{lemma}\label{L-dT-2}
For any $g^\pitchfork$-variation of metric $g\in{\rm Riem}(M,\,\widetilde{\mD},\,{\mD})$ we have
\begin{eqnarray*}
 && \dt\tr^\top\I = 0,\quad
 \dt\tr^\bot\I = -\sum\nolimits_{\,i} 
 \big(\frac12\,(\I_i+\I^\wedge_i)(B^\sharp{\cal E}_i)^\bot
 + (\I_i+\I^\wedge_i)(B^\sharp{\cal E}_i)^\top\big),\\
 && \dt\tr^\top\I^* = \sum\nolimits_{\,a} 
  [\I^*_a, B^\sharp]\,E_a,\\
 && \dt\tr^\bot\I^* = \sum\nolimits_{\,i} 
 \big( [\I^*_i,B^\sharp]\,{\cal E}_i
 -\frac12\,(\I^*_i +\I^{* \wedge}_i)(B^\sharp{\cal E}_i)^\bot -(\I^*_i +\I^{* \wedge}_i)(B^\sharp{\cal E}_i)^\top\big).
\end{eqnarray*}
\end{lemma}

\begin{proof}
 For any variation $g_t$ of metric and $X,Y\in\mathfrak{X}_M$ we have
\begin{equation*}
 (\dt\I^\wedge)_X Y = (\dt\I)_Y X =0,\quad
 (\dt\I^*)_X = [\I^*_X, B^\sharp]\,,
\end{equation*}
where the first formula is obvious, the second one follows from \eqref{E-defTT}$_1$, equality $\dt\I=0$ and
\begin{eqnarray*}
 \<\I^*_X B^\sharp(Y),Z\> \eq \<\I_X Z,B^\sharp(Y)\> = B(\I_X Z,Y) =\dt \<\I_X Z,Y\> = \dt \<\I^*_X Y,Z\> \\
 \eq B(\I^*_X Y,Z) +\<\dt\I^*_X Y,Z\> = \<B^\sharp\I^*_X Y,Z\> +\<\dt\I^*_X Y,Z\>.
\end{eqnarray*}
Using the above and  \eqref{E-frameE} completes the proof.
\end{proof}

Lemma~\ref{L-dT-2} is used in the proof of the following

\begin{lemma}\label{L-dT-3}
For $g^\pitchfork$-variation $g_t$ of metric on $(M,\widetilde{\mD},g,\bar\nabla=\nabla+\I)$ we have
\begin{eqnarray}\label{dtIIproduct}
 \nonumber
 &&\dt \<\I^*, \I^\wedge\>_{\,|\,V} = -\sum B({\cal E}_i, {\cal E}_j) \<\I^*_{j} E_a, \I_a {\cal E}_i\> \\
 && +\sum B({\cal E}_i, E_b) \big(\<\I^*_j {\cal E}_i, \I_b {\cal E}_j\>
 -\<\I^*_a {\cal E}_i, \I_{ b} E_a\> - \<\I^*_{b} E_a, \I_a {\cal E}_i\> - \<\I^*_i E_a, \I_a E_b \>\big) ,
\end{eqnarray}
\begin{eqnarray}
\label{dtThetaA}
\nonumber
 && \dt \<\Theta, A \>  = -2 \sum B({\cal E}_i, {\cal E}_j) \big(\< h (E_a, E_b), {\cal E}_i \> \<{\cal E}_j, \I_a\,E_b \> \\
 && -\frac{1}{2} \<h(E_a,E_b),{\cal E}_j\> (\< \Theta_a {\cal E}_i +\Theta_i E_a, E_b \>
 +\<\Theta_b {\cal E}_i +\Theta_i E_b, E_a \>)\big) \nonumber \\
 && +\sum B({\cal E}_i, E_b) \big(\<\Theta_a {\cal E}_j +\Theta_j E_a, {\cal E}_i \> \< h (E_a, E_b), {\cal E}_j \> \nonumber \\
 && -\< \Theta_a E_b +\Theta_b E_a, E_c \> \< h (E_a, E_c), {\cal E}_i \>
 + 2\,\< h (E_a, E_b), {\cal E}_j \> \<{\cal E}_j, \I_a\,{\cal E}_i \>\nonumber \\
 && -\frac{1}{2}\,\<({\tilde A}- {\tilde T}^\sharp )_a {\cal E}_i, {\cal E}_j \>
 (\<\Theta_a {\cal E}_j +\Theta_j E_a, E_b \> +\< \Theta_b {\cal E}_j +\Theta_j E_b, E_a \>) \nonumber \\
 && -2\,\< h (E_b, E_a), {\cal E}_j \> \<{\cal E}_i, \I_{j} E_a \>
 + 2\,\< h (E_a, E_b), {\cal E}_j \> \<E_a, \I_j\,{\cal E}_i \> \nonumber\\
 && -2\,\< h (E_a, E_c), {\cal E}_i \> \<E_b, \I_a E_c \> \big)
 +\Div^\top \< B_{| V}, G\> -\< B_{| V}, \Div^\top G \> ,
\end{eqnarray}
\begin{eqnarray}
\label{dtThetaT}
\nonumber
 && \dt \<\Theta, T^\sharp\> = -2 \sum B({\cal E}_i, {\cal E}_j) \< T (E_a, E_b), {\cal E}_i \> \<{\cal E}_j, \I_a E_b \> \\
 && +\sum B({\cal E}_i, E_b) \big(\< \Theta_a {\cal E}_j +\Theta_j E_a, {\cal E}_i \> \< T (E_a, E_b), {\cal E}_j \>
 -2\,\< T (E_a, E_c), {\cal E}_i \> \<E_b, \I_a E_c \> \nonumber \\
 &&  -\< \Theta_a E_b +\Theta_{b} E_a, E_c \> \< T (E_a, E_c), {\cal E}_i \>
 + 2\< T (E_a, E_b), {\cal E}_j \> \<{\cal E}_j, \I_a {\cal E}_i \> \nonumber \\
 && -2\< T (E_b, E_a), {\cal E}_j \> \<{\cal E}_i, \I_{j} E_a \>
 + 2\< T (E_a, E_b), {\cal E}_j \> \<E_a, \I_j {\cal E}_i \>\big),
\end{eqnarray}
\begin{eqnarray}
\label{dtThetatildeT}
\nonumber
 && \dt \<\Theta, {\tilde T}^\sharp\> =
 \sum B({\cal E}_i, {\cal E}_j) \big( 2\,\<{\tilde T}({\cal E}_k, {\cal E}_j), E_a\> \<{\cal E}_k, \I_a {\cal E}_i\>  \\
\nonumber
 && -2\<{\tilde T}({\cal E}_i, {\cal E}_k), E_a\>  \<{\cal E}_j, \I_a {\cal E}_k\>
 + 2\<{\tilde T}({\cal E}_k, {\cal E}_j ), E_a\> \<{E}_a, \I_k {\cal E}_i\>\\
 && -\frac{1}{2}\,\< \Theta_a {\cal E}_j +\Theta_{j} E_a, {\cal E}_k \> \<{\tilde T}({\cal E}_i, {\cal E}_k ), E_a \>
 +\frac{1}{2}\,\< \Theta_a {\cal E}_k +\Theta_k E_a, {\cal E}_i \> \< {\tilde T}({\cal E}_k, {\cal E}_j ), E_a \>
 \nonumber \\
 && -\frac{1}{2}\,(\< \Theta_a {\cal E}_i +\Theta_i E_a, {\cal E}_k \> -\< \Theta_a {\cal E}_k +\Theta_k E_a, {\cal E}_i) \> \<E_a,  {\tilde T}({\cal E}_j, {\cal E}_k) \> \big)\nonumber \\
 && +\sum B({\cal E}_i, E_b)\big( \frac{1}{2}\,(\< \Theta_a {\cal E}_i +\Theta_i E_a, {\cal E}_j \> -\< \Theta_a {\cal E}_j +\Theta_j E_a, {\cal E}_i) \> \<E_a,(A + T^\sharp )_j E_b \> \nonumber \\
 && -2\,\<{\tilde T}({\cal E}_i, {\cal E}_j ), E_a\>  \<E_b, \I_a{\cal E}_j\>
 + 2\,\<{\tilde T}({\cal E}_j, {\cal E}_i ), E_a\> \<{E}_a, \I_j E_b\> \nonumber \\
 && + 2\,\<{\tilde T}({\cal E}_j, {\cal E}_i ), E_a\> \<{\cal E}_j, \I_a E_b\>
 -2\,\<{\tilde T}({\cal E}_k, {\cal E}_j), E_b\> \<{\cal E}_i, \I_{k}{\cal E}_j\> \nonumber \\
 && -\< \Theta_a E_b +\Theta_{b} E_a, {\cal E}_j \> \< {\tilde T}({\cal E}_i, {\cal E}_j ), E_a\> \big)
 +\Div^\perp\< B_{|V}, F \> -\< B_{| V}, \Div^\perp F \> ,
\end{eqnarray}
\begin{eqnarray}
\label{dtThetatildeA}
 && \dt \<\Theta, {\tilde A} \> =  \sum B({\cal E}_i, {\cal E}_j) \big(\frac{1}{2}\<\Theta_k E_a +\Theta_a {\cal E}_k, {\cal E}_i\>
 \<{\tilde h}({\cal E}_k, {\cal E}_j ), E_a\> \nonumber \\
\nonumber
 && -\frac{1}{2}\<\Theta_{j} E_a +\Theta_a{\cal E}_j, {\cal E}_k\>\<{\tilde h}({\cal E}_i, {\cal E}_k), E_a\>
 -2\<{\tilde h}({\cal E}_i, {\cal E}_k ), E_a\> \<{\cal E}_j, \I_a {\cal E}_k\>\\
 && -(\<\Theta_i E_a +\Theta_a{\cal E}_i,{\cal E}_k\> +\<\Theta_k E_a +\Theta_a{\cal E}_k, {\cal E}_i\>)\<{\tilde h}({\cal E}_j, {\cal E}_k), E_a\> \nonumber \\
 && +(\<\Theta_k E_a +\Theta_a {\cal E}_k, {\cal E}_j\> +\<\Theta_j E_a +\Theta_a {\cal E}_j, {\cal E}_k )\> \<({\tilde A}_a -{\tilde T}^\sharp_a) {\cal E}_i, {\cal E}_k\> \nonumber \\
 && + 2\<{\tilde h}({\cal E}_k, {\cal E}_j), E_a\> \<{\cal E}_k, \I_a {\cal E}_i\>
 +2\<{\tilde h}({\cal E}_k, {\cal E}_j ), E_a\>\<{E}_a, \I_k {\cal E}_i \> \big) \nonumber \\
 && +\sum B({\cal E}_i, E_b)\big((\<\Theta_j E_a +\Theta_a {\cal E}_j, {\cal E}_i\>+\<\Theta_i E_a +\Theta_a{\cal E}_i, {\cal E}_j)\>
 \<(A_j + T^\sharp_j) E_b, E_a\> \nonumber \\
 && -(\<\Theta_i E_a +\Theta_a {\cal E}_i, {\cal E}_j\> +\<\Theta_j E_a +\Theta_a {\cal E}_j, {\cal E}_i )\> \< (A_j - T^\sharp_j) E_b, E_a\> \nonumber \\
 && -2\<{\tilde h}({\cal E}_i, {\cal E}_j ), E_a\> \<E_b, \I_a {\cal E}_j\>
 +2\<{\tilde h}({\cal E}_j, {\cal E}_i ), E_a\> \<{E}_a, \I_j E_b\> \nonumber \\
 \nonumber
 && + 2\<{\tilde h}({\cal E}_j, {\cal E}_i ), E_a\> \<{\cal E}_j, \I_a E_b\>
 -2\<{\tilde h}({\cal E}_k, {\cal E}_j), E_b\> \<{\cal E}_i, \I_{k}{\cal E}_j\> \\
  && -\<\Theta_{b} E_a +\Theta_a  E_b, {\cal E}_j\>\<{\tilde h}({\cal E}_i, {\cal E}_j), E_a\> \big)
  -2 \Div^\top \< B, L \> + 2 \< B, \Div^\top L \> ,
\end{eqnarray}
\begin{equation}\label{dttracetopI}
 \dt \<\tr^\top \I, \tr^\perp \I^*\>
 = \frac{1}{2} \sum B({\cal E}_i, {\cal E}_j) \big(\<\tr^\top\I,\, \I^*_i {\cal E}_j {-} \I^*_j {\cal E}_i\>\big)
 -\sum B({\cal E}_i, E_b) \<\tr^\top\I,\, \I^*_b {\cal E}_i\>,
\end{equation}
\begin{eqnarray} \label{dttraceperpI}
\nonumber
 && \dt \<\tr^\top \I^*, \tr^\perp \I\> = -\frac{1}{2} \sum B({\cal E}_i, {\cal E}_j)
 \<\I_j {\cal E}_i + \I_i {\cal E}_j, \tr^\top\I^* \> \\
 && + \sum B({\cal E}_i, E_b)\big(\<\tr^\bot\I, \I^*_b {\cal E}_i\> - \<\I_b {\cal E}_i + \I_i E_b, \tr^\top\I^*\> \big),
\end{eqnarray}
\begin{eqnarray} \label{dtIEaH}
 && \dt\<\,\tr^\top(\I^*-\I),\, \tH - H\>
 = \sum B({\cal E}_i, {\cal E}_j) \big( -\frac{1}{2}\,\delta_{ij} \Div((\tr^\top(\I^*-\I) )^\top) \nonumber \\
 && \nonumber
 -\<H, {\cal E}_j\> \< \tr^\top(\I^*-\I), {\cal E}_i\> -\<H, {\cal E}_j\> \<\tr^\top(\I^*-\I), {\cal E}_i\>
  +\<\tr^\top\I^*,  {\cal E}_j\> \<{\cal E}_i, H\>\big) \nonumber \\
 && +\sum B({\cal E}_i, E_b) \big(\<\tH, E_b\> \< \tr^\top(\I^*-\I), {\cal E}_i\>
 -\<H, {\cal E}_i\> \< \tr^\top(\I^*-\I), E_b\> \nonumber \\
 && +\,\<\I^*_b {\cal E}_i, \tH - H\> +\<\tr^\top\I^*,  E_b\> \<{\cal E}_i, H\> + 2\<\T_i E_b , \tr^\top(\I^*-\I)\>  \nonumber \\
 && +\,\<{\cal E}_i,(\tr^\top(\I^*-\I))^\perp\> \<\tilde H, E_b\> -\< H, {\cal E}_i\> \<\tr^\top(\I^*-\I), E_b\> \nonumber \\
 && -\,\<\tr^\top\I^*,  {\cal E}_i\> \<E_b, \tH\> -\<\nabla_b((\tr^\top(\I^*-\I))^\perp ), {\cal E}_i\>
 -\<\tA_b {\cal E}_i -\tT_b {\cal E}_i, \tr^\top(\I^*-\I)\> \big) \nonumber \\
 && +\Div\big(( B^\sharp((\tr^\top(\I^*-\I))^\perp))^\top -\frac12\, (\tr_{\mD}B)(\tr^\top(\I^*-\I))^\top\big),
\end{eqnarray}
\begin{eqnarray} \label{dtIeiH}
 && \dt \<\,\tr^\bot(\I^*-\I), \tH - H\>
 = \sum B({\cal E}_i, {\cal E}_j) \big(\<\I^*_j {\cal E}_i, \tH - H\> +\<\tr^\bot\I^*, {\cal E}_i\> \< H, {\cal E}_j\>\nonumber \\
 && -\<H, {\cal E}_j\> \< \tr^\bot(\I^*-\I), {\cal E}_i\> -\frac{1}{2}\,\delta_{ij} \Div((\tr^\bot(\I^*-\I))^\top)
 -\<H, {\cal E}_j\> \<\tr^\bot(\I^*-\I), {\cal E}_i\> \big) \nonumber \\
 && +\sum B({\cal E}_i, E_b) \big(\<\tH, E_b\> \<\tr^\bot(\I^*-\I), {\cal E}_i\>
 -\<H, {\cal E}_i\> \<\tr^\bot(\I^*-\I), E_b\>  \nonumber \\
 && +\,\<\tr^\bot\I^*,  E_b\> \< H, {\cal E}_i\> +\<\I^*_i E_b, \tH - H\> -\<\tr^\bot\I^*, {\cal E}_i\> \<\tH, E_b\>
 +2\<\T_i E_b , \tr^\bot(\I^*-\I)\> \nonumber \\
 && +\<{\cal E}_i, \tr^\bot(\I^*-\I)\> \<\tilde H, E_b\>
 -\<\nabla_b((\tr^\bot(\I^*-\I))^\perp), {\cal E}_i\> \nonumber \\
 && -\<\tA_b {\cal E}_i - \tT_b {\cal E}_i, \tr^\bot(\I^*-\I) \>
 -\< H, {\cal E}_i\> \<\tr^\bot(\I^*-\I), E_b\> \big) \nonumber \\
 && + \Div\big(( B^\sharp((\tr^\bot(\I^*-\I))^\perp))^\top
 -\frac12\,(\tr_{\mD}B)(\tr^\bot(\I^*-\I))^\top\big) .
\end{eqnarray}
\end{lemma}

\begin{proof}
To obtain $\dt \Theta$, we compute for $X,Y,Z \in \mathfrak{X}_M$:
\begin{equation*}
 \dt \<\I^{* \wedge}_X Y, Z\> = B(\I^{* \wedge}_X Y, Z) +\<(\dt \I^{* \wedge})_X Y, Z\>.
\end{equation*}
On the other hand,
\begin{equation*}
 \dt \<\I^{* \wedge}_X Y, Z\> = \dt \<\I^*_Y X, Z\> = B(\I^*_Y X, Z) +\<\dt (\I^*_Y X), Z\>
 = \<\I^*_Y B^\sharp X, Z \>,
\end{equation*}
so
\[
 (\dt \I^{*\wedge})_X Y = \I^*_Y B^\sharp X - B^\sharp\,\I^{*}_Y X.
\]
From this we obtain
\begin{equation}\label{Eq-dt-Theta}
 (\dt \Theta)_X Y = -(\dt \I^*)_X Y -\dt (\I^{*\wedge})_X Y
 = -\I^*_X B^\sharp Y + B^\sharp \I^*_X Y -\I^{*}_Y\,B^\sharp X + B^\sharp \I^{*}_Y X .
\end{equation}
We shall use Proposition~\ref{prop-Ei-a} and the fact that for $g^\perp$-variations $B(X,Y)=0$ for $X,Y \in \mathfrak{X}^\top$. 

\textbf{Proof of \eqref{dtIIproduct}}.
We have $\<\I^*, \I^\wedge\>_{\,|\,V} = \sum \<\I^*_a {\cal E}_i, \I_i E_a\> +\sum \<\I^*_i E_a, \I_a {\cal E}_i \>$, so
\begin{eqnarray*}
 && \dt \<\I^*, \I^\wedge\>_{\,|\,V} = \sum\big[ B(\I^*_a {\cal E}_i, \I_i E_a)
 + B(\I^*_i E_a, \I_a {\cal E}_i) + \<\I^*_a \dt {\cal E}_i, \I_i E_a\> \\
 &&\hskip-4mm + \<\I^*_a {\cal E}_i, \I_{\dt {\cal E}_i} E_a\>
 + \<\I^*_{\dt {\cal E}_i} E_a, \I_a {\cal E}_i\>
 + \<\I^*_i E_a, \I_a \dt {\cal E}_i\>
 + \<(\dt \I^*)_a {\cal E}_i, \I_i E_a\>
 + \<(\dt \I^*)_i E_a, \I_a {\cal E}_i\> \big].
\end{eqnarray*}
We compute 8 terms above separately:
\begin{eqnarray*}
 && \sum B(\I^*_a {\cal E}_i, \I_i E_a) = \sum \big[ B({\cal E}_j, E_b) \<\I^*_a {\cal E}_i, E_b\> \<\I_i E_a, {\cal E}_j\> \\
 &&\quad + B({\cal E}_j, E_b) \<\I^*_a {\cal E}_i, {\cal E}_j\> \<\I_i E_a, E_b\>
 + B({\cal E}_j, {\cal E}_k) \<\I^*_a {\cal E}_i, {\cal E}_k\> \<\I_i E_a, {\cal E}_j\> \big],\\
 && \sum B(\I^*_i E_a, \I_a {\cal E}_i) = \sum \big[ B({\cal E}_j, E_b) \<\I^*_i E_a, E_b\> \<\I_a {\cal E}_i, {\cal E}_j\> \\
 &&\quad + B({\cal E}_j, E_b) \<\I^*_i E_a, {\cal E}_j\> \<\I_a {\cal E}_i, E_b\>
 + B({\cal E}_k, {\cal E}_j) \<\I^*_i E_a, {\cal E}_j\> \<\I_a {\cal E}_i, {\cal E}_k\> \big],\\
 &&\sum \<\I^*_a \dt {\cal E}_i, \I_i E_a\> = -\sum \big[ B({\cal E}_i,E_b) \<\I^*_a E_b, \I_i E_a\>
 +\frac{1}{2} B({\cal E}_i, {\cal E}_j)  \<\I^*_a {\cal E}_j, \I_i E_a\> \big],\\
 &&\sum \<\I^*_a {\cal E}_i, \I_{\dt {\cal E}_i} E_a\> = -\sum \big[ B({\cal E}_i, E_b) \<\I^*_a {\cal E}_i, \I_{ b} E_a\>
 +\frac{1}{2} B({\cal E}_i, {\cal E}_j) \<\I^*_a {\cal E}_i, \I_{ j} E_a\> \big],\\
 &&\sum \<\I^*_{\dt {\cal E}_i} E_a, \I_a {\cal E}_i\> = -\sum \big[ B({\cal E}_i, E_b) \<\I^*_{b} E_a, \I_a {\cal E}_i\>
 +\frac{1}{2} B({\cal E}_i, {\cal E}_j) \<\I^*_{j} E_a, \I_a {\cal E}_i\> \big],\\
 &&\sum \<\I^*_i E_a, \I_a \dt {\cal E}_i\> = -\sum \big[ B({\cal E}_i, E_b) \<\I^*_i E_a, \I_a E_b\>
 +\frac{1}{2} B({\cal E}_i, {\cal E}_j) \<\I^*_i E_a, \I_a {\cal E}_j\>\big] ,\\
 && \sum \< (\dt \I^*)_a {\cal E}_i, \I_i E_a\>
 =\sum \big[ B({\cal E}_i, E_b) \<\I^*_a E_b, \I_i E_a\>
 + B({\cal E}_i, {\cal E}_j) \<\I^*_a {\cal E}_j, \I_i E_a\> \\
 &&\, -B({\cal E}_j, E_b) \<\I^*_a {\cal E}_i, E_b\> \<{\cal E}_j, \I_i E_a\>
 {-}B({\cal E}_i, {\cal E}_j) \<\I^*_a {\cal E}_k, {\cal E}_j\> \<{\cal E}_i, \I_k E_a\>
 {-}B({\cal E}_j, E_b) \<\I^*_a {\cal E}_i, {\cal E}_j\> \< E_b, \I_i E_a\> \big] ,\\
 && \sum \< (\dt \I^*)_i E_a, \I_a {\cal E}_i\>
 = \sum \big[ B({\cal E}_j, E_a) \<\I^*_i {\cal E}_j, \I_a {\cal E}_i\>
 - B({\cal E}_j, E_b) \<\I^*_i E_a, E_b\> \<{\cal E}_j, \I_a {\cal E}_i\> \\
 &&\  - B({\cal E}_j, E_b) \<\I^*_i E_a, {\cal E}_j\> \< E_b, \I_a {\cal E}_i\>
 - B({\cal E}_j, {\cal E}_k) \<\I^*_i E_a, {\cal E}_j\> \<{\cal E}_k, \I_a {\cal E}_i\> \big].
\end{eqnarray*}
Summing the 8 terms computed above and simplifying, we obtain \eqref{dtIIproduct}.

\textbf{Proof of \eqref{dtThetaA}}.
We have
\begin{eqnarray*}
%
%
 && \< \Theta, A \>
 = \sum \<\Theta_a {\cal E}_i +\Theta_i E_a, E_b\> \< h (E_a, E_b), {\cal E}_i \>.
\end{eqnarray*}
So
\begin{eqnarray*}
 && \dt \< \Theta, A \> = \sum \big[ B(\Theta_a {\cal E}_i +\Theta_i E_a, E_b) \< h (E_a, E_b), {\cal E}_i\>
 + \<\Theta_a {\cal E}_i +\Theta_i E_a, E_b\> B( h (E_a, E_b), {\cal E}_i) \\
 && +\, \<\Theta_a (\dt {\cal E}_i) +\Theta_{\dt {\cal E}_i} E_a, E_b\> \< h (E_a, E_b), {\cal E}_i\>
 + \<\Theta_a {\cal E}_i +\Theta_i E_a, E_b\> \<\dt h (E_a, E_b), {\cal E}_i\> \\
 && +\, \<\Theta_a {\cal E}_i +\Theta_i E_a, E_b\> \< h (E_a, E_b), \dt {\cal E}_i\>
 + \<(\dt\Theta)_a {\cal E}_i + (\dt \Theta)_i E_a, E_b\> \< h (E_a, E_b), {\cal E}_i \> \big].
\end{eqnarray*}
We start from the fourth term of the 6 terms above.
Then, from \cite{rz-2}, 
\begin{eqnarray*}
 &&\sum \<\Theta_a {\cal E}_i +\Theta_i E_a, E_b\> \<\dt h (E_a, E_b), {\cal E}_i\>
  = \sum \frac{1}{2}\big[ \<\Theta_a {\cal E}_i +\Theta_i E_a, E_b\> \\
 && +\<\Theta_b {\cal E}_i +\Theta_i E_b, E_a\>\big]
 \big(\nabla_a B(E_b, {\cal E}_i) - B( h(E_a, E_b ), {\cal E}_i) + B(\nabla_i E_a, E_b) \big).
\end{eqnarray*}
We have
\begin{eqnarray*}
  \frac{1}{2}\sum \big(\<\Theta_a {\cal E}_i +\Theta_i E_a, E_b\> +\<\Theta_b {\cal E}_i +\Theta_i E_b, E_a \>\big)
 \nabla_a B(E_b, {\cal E}_i)
  = \Div^\top \< B_{| V}, G \> -\< B_{| V}, \Div^\top G \>,
\end{eqnarray*}
because
\[
 \< B_{| V}, \Div G \> = \frac{1}{2} \sum \big(\<\nabla_a {\Theta}^{\wedge*}_i E_b, E_a\>
  + \<\nabla_a {\Theta}^*_i E_b, E_a\>\big) B(E_b, {\cal E}_i) .
\]
We also have
\begin{eqnarray*}
 && -\sum\frac{1}{2}\big( \<\Theta_a {\cal E}_i +\Theta_i E_a, E_b\> +\<\Theta_b {\cal E}_i +\Theta_i E_b, E_a \>\big)
 B( h(E_a, E_b ), {\cal E}_i) \\
 && = -\frac{1}{2} \sum B({\cal E}_i, {\cal E}_j) \<{\cal E}_j, h(E_a,E_b)\>
 ( \<\Theta_a {\cal E}_i +\Theta_i E_a, E_b\>  +\<\Theta_b {\cal E}_i +\Theta_i E_b, E_a \>),\\
 && \sum \frac{1}{2}\big(\<\Theta_a {\cal E}_i +\Theta_i E_a, E_b\> +\<\Theta_b {\cal E}_i +\Theta_i E_b, E_a\>\big) B(\nabla_i E_a, E_b) \\
 && = -\frac{1}{2} \sum B({\cal E}_j, E_b) \<({\tilde A}- {\tilde T}^\sharp )_a {\cal E}_j, {\cal E}_i \>
 (\<\Theta_a {\cal E}_i +\Theta_i E_a, E_b\> +\<\Theta_b {\cal E}_i +\Theta_i E_b, E_a \>).
\end{eqnarray*}
Now we consider other terms of $\dt \<\Theta, A \>$.
 For the fifth term we have
\begin{equation*}
 \sum \<\Theta_a {\cal E}_i +\Theta_i E_a, E_b\> \< h (E_a, E_b), \dt {\cal E}_i\>
 = -\frac{1}{2} \sum  B({\cal E}_i, {\cal E}_j) \<\Theta_a {\cal E}_i +\Theta_i E_a, E_b\> \< h (E_a, E_b), {\cal E}_j\>.
\end{equation*}
For the first, second and third terms we have
\begin{eqnarray*}
 &&\hskip-7mm \sum B(\Theta_a {\cal E}_i +\Theta_i E_a, E_b) \< h (E_a, E_b), {\cal E}_i\>
 = \sum B({\cal E}_j, E_b) \<\Theta_a {\cal E}_i +\Theta_i E_a, {\cal E}_j\> \<h (E_a, E_b), {\cal E}_i\> ,\\
 &&\hskip-7mm \sum \<\Theta_a {\cal E}_i +\Theta_i E_a, E_b\> B( h (E_a, E_b), {\cal E}_i) =
 \sum B({\cal E}_i, {\cal E}_j) \<\Theta_a {\cal E}_i +\Theta_i E_a, E_b\> \< h (E_a, E_b), {\cal E}_j\>,\\
 &&\hskip-7mm \<\Theta_a (\dt {\cal E}_i) {+}\Theta_{\dt {\cal E}_i} E_a, E_b\>
 {=} -\frac{1}{2} \sum B({\cal E}_i, {\cal E}_j) \<\Theta_a {\cal E}_j {+}\Theta_j E_a, E_b\>
 {-}\sum B({\cal E}_i, E_c) \<\Theta_a E_c {+}\Theta_c E_a, E_b\>
 .
\end{eqnarray*}
Using \eqref{Eq-dt-Theta}, we have
\begin{eqnarray*}
 && \sum \<(\dt \Theta)_a {\cal E}_i + (\dt \Theta)_i E_a, E_b\>
 = \sum\big[-2 B({\cal E}_i, E_c) \<\I^*_a E_c, E_b\> -2 B({\cal E}_i, {\cal E}_j) \<\I^*_a {\cal E}_j, E_b\> \\
 && + 2 B({\cal E}_j, E_b) \<\I^*_a {\cal E}_i, {\cal E}_j\>
 -2 B({\cal E}_j, E_a) \<\I^{*}_{i} {\cal E}_j,E_b\>
 + 2 B({\cal E}_j, E_b) \<\I^{*}_i {E}_a, {\cal E}_j\>\big] .
\end{eqnarray*}
Hence, for the sixth term of $\dt \< \Theta, A \>$, we have
\begin{eqnarray*}
 && \sum \< (\dt \Theta)_a {\cal E}_i + (\dt \Theta)_i E_a, E_b\> \< h (E_a, E_b), {\cal E}_i\>
 = \sum[ -2 B({\cal E}_i, E_c) \< h (E_a, E_b), {\cal E}_i\> \<\I^*_a E_c, E_b\> \\
 && -2 B({\cal E}_i, {\cal E}_j) \< h (E_a, E_b), {\cal E}_i\> \<\I^*_a {\cal E}_j, E_b\>
 + 2 B({\cal E}_j, E_b) \< h (E_a, E_b), {\cal E}_i\> \<\I^*_a {\cal E}_i, {\cal E}_j\> \\
 && -2 B({\cal E}_j, E_a) \< h (E_a, E_b), {\cal E}_i\> \<\I^{*}_{i} {\cal E}_j,E_b\>
 + 2 B({\cal E}_j, E_b) \< h (E_a, E_b), {\cal E}_i\> \<\I^{*}_i {E}_a, {\cal E}_j\> \big].
\end{eqnarray*}
So finally
we get \eqref{dtThetaA}.

\textbf{Proof of \eqref{dtThetaT}}.
We have
\begin{equation*}
 \<\Theta, T^\sharp \>
 = \sum \<\Theta_a {\cal E}_i +\Theta_i E_a, E_b\> \< T (E_a, E_b), {\cal E}_i \>,
\end{equation*}
thus
\begin{eqnarray*}
 && \dt \< \Theta, T^\sharp \> = \sum \big[ B(\Theta_a {\cal E}_i +\Theta_i E_a, E_b) \< T (E_a, E_b), {\cal E}_i\> \\
 && + \<\Theta_a {\cal E}_i +\Theta_i E_a, E_b\> B( T (E_a, E_b), {\cal E}_i)
 + \<\Theta_a (\dt {\cal E}_i) +\Theta_{\dt {\cal E}_i} E_a, E_b\> \<T(E_a, E_b), {\cal E}_i\> \\
 && + \<\Theta_a {\cal E}_i +\Theta_i E_a, E_b\> \< T (E_a, E_b), \dt {\cal E}_i\>
 + \<(\dt \Theta)_a {\cal E}_i + (\dt \Theta)_i E_a, E_b\> \< T (E_a, E_b), {\cal E}_i\> \big],
\end{eqnarray*}
because $\dt T =0$.
We compute 5 terms above separately:
\begin{eqnarray*}
 && \sum B(\Theta_a {\cal E}_i +\Theta_i E_a, E_b) =
 \sum B({\cal E}_j, E_b) \<\Theta_a {\cal E}_i +\Theta_i E_a, {\cal E}_j\> \< T (E_a, E_b), {\cal E}_i\> ,\\
 &&\sum \<\Theta_a {\cal E}_i +\Theta_i E_a, E_b\> B( T (E_a, E_b), {\cal E}_i) =
 \sum B({\cal E}_i, {\cal E}_j) \<\Theta_a {\cal E}_j +\Theta_j E_a, E_b\> \< T (E_a, E_b), {\cal E}_i\> ,\\
 && \sum \<\Theta_a (\dt {\cal E}_i) +\Theta_{\dt{\cal E}_i} E_a, E_b\> \< T (E_a, E_b), {\cal E}_i\> =
 -\frac{1}{2} \sum B({\cal E}_i, {\cal E}_j) \<\Theta_a {\cal E}_j +\Theta_j E_a, E_b\> \<T (E_a, E_b), {\cal E}_i\> \\
 && -\sum B({\cal E}_i, E_b) \<\Theta_a E_b +\Theta_b E_a, E_c\> \<T (E_a, E_c), {\cal E}_i\> ,\\
 && \sum \<\Theta_a {\cal E}_i +\Theta_i E_a, E_b\> \< T (E_a, E_b), \dt {\cal E}_i\> =
 -\frac{1}{2}\sum B({\cal E}_i, {\cal E}_j) \<\Theta_a {\cal E}_i +\Theta_i E_a, E_b\>\< T (E_a, E_b), {\cal E}_j\> \big],\\
 && \sum \< (\dt \Theta)_a {\cal E}_i + (\dt \Theta)_i E_a, E_b\>\< T (E_a, E_b), {\cal E}_i\>
 = \sum \big[ -2 B({\cal E}_i, E_c) \<\I^*_a E_c, E_b\>\< T (E_a, E_b), {\cal E}_i\> \\
 && -2 B({\cal E}_i, {\cal E}_j) \<\I^*_a {\cal E}_j, E_b\>\< T (E_a, E_b), {\cal E}_i\>
 + 2 B({\cal E}_j, E_b) \<\I^*_a {\cal E}_i, {\cal E}_j\>\< T (E_a, E_b), {\cal E}_i\> \\
 && -2 B({\cal E}_j, E_a) \<\I^{*}_{i} {\cal E}_j,E_b\>\< T (E_a, E_b), {\cal E}_i\>
 + 2 B({\cal E}_j, E_b) \<\I^{*}_i {E}_a, {\cal E}_j\>\< T (E_a, E_b), {\cal E}_i\> \big].
\end{eqnarray*}
Finally,
we get \eqref{dtThetaT}.

\textbf{Proof of \eqref{dtThetatildeT}}.
We have
 $\< \Theta, {\tilde T}^\sharp \>
 = \sum \<\Theta_a {\cal E}_i +\Theta_i E_a, {\cal E}_j\> \<{\tilde T}({\cal E}_i, {\cal E}_j ), E_a\>$.
Now we compute
\begin{eqnarray}\label{Enew-6terms}
\nonumber
 && \dt \< \Theta, {\tilde T}^\sharp \> =
 \sum \big[ B(\Theta_a {\cal E}_i +\Theta_i E_a, {\cal E}_j) \<{\tilde T}({\cal E}_i, {\cal E}_j ), E_a\>
 + \<\Theta_a {\cal E}_i +\Theta_i E_a, {\cal E}_j\> B({\tilde T}({\cal E}_i, {\cal E}_j ), E_a) \\
\nonumber
 && + \<\Theta_a (\dt {\cal E}_i) +\Theta_{\dt {\cal E}_i} E_a, {\cal E}_j\> \<{\tilde T}({\cal E}_i, {\cal E}_j ), E_a\>
 + \<\Theta_a {\cal E}_i +\Theta_i E_a, \dt {\cal E}_j\> \<{\tilde T}({\cal E}_i, {\cal E}_j ), E_a\> \\
 && + \<\Theta_a {\cal E}_i +\Theta_i E_a, {\cal E}_j\> \<\dt {\tilde T}({\cal E}_i, {\cal E}_j ), E_a\>
 + \< (\dt \Theta)_a {\cal E}_i + (\dt \Theta)_i E_a, {\cal E}_j\> \<{\tilde T}({\cal E}_i, {\cal E}_j ), E_a\> \big].
\end{eqnarray}
Let $U : \mD \times \mD \rightarrow \widetilde{\mD}$ be a $(1,2)$-tensor, given by
$\<U_i {\cal E}_j, E_a\>=\frac{1}{2}\,(\<\Theta_a{\cal E}_i +\Theta_i E_a, {\cal E}_j\>-\<\Theta_a{\cal E}_j +\Theta_j E_a, {\cal E}_i\>)$.
We compute the fifth term in $\dt \< \Theta, {\tilde T}^\sharp \>$:
\begin{eqnarray*}
 &&\hskip-4mm 2 \sum\<\Theta_a {\cal E}_i +\Theta_i E_a, {\cal E}_j\> \<\dt {\tilde T}({\cal E}_i, {\cal E}_j ), E_a\>
 =  2 \sum \<U_i {\cal E}_j, E_a\> \<\dt {\tilde T}({\cal E}_i, {\cal E}_j ), E_a\> \\
 && =\sum \<U_i {\cal E}_j, E_a\> \big( 2\<{\tilde T}(-\frac{1}{2}(B^\sharp {\cal E}_i )^\perp, {\cal E}_j ), E_a\>
 + 2\<{\tilde T}( {\cal E}_i, -\frac{1}{2}(B^\sharp {\cal E}_j )^\perp ), E_a\> \\
 && +\< \nabla_{(B^\sharp {\cal E}_j)^\top} {\cal E}_i -\nabla_{(B^\sharp {\cal E}_i  )^\top} {\cal E}_j , E_a \>
 +\< \nabla_{j}((B^\sharp {\cal E}_i )^\top) -\nabla_{i}((B^\sharp {\cal E}_j )^\top) , E_a \> \big),\\
 &&\hskip-4mm \sum \<U_i {\cal E}_j, E_a\> \<{\tilde T}(-( B^\sharp {\cal E}_i )^\perp, {\cal E}_j ), E_a\>
 = -\sum B({\cal E}_i, {\cal E}_k) \<U_i {\cal E}_j, {\tilde T}({\cal E}_k, {\cal E}_j )\>\\
 && = -\frac{1}{2} \sum B({\cal E}_i, {\cal E}_k) \big(\<\Theta_a {\cal E}_i +\Theta_i E_a, {\cal E}_j\> -\<\Theta_a {\cal E}_j +\Theta_j E_a, {\cal E}_i \>\big) \<E_a,  {\tilde T}({\cal E}_k, {\cal E}_j )\>, \\
 &&\hskip-4mm -\sum \<U_i {\cal E}_j, E_a\> \<\nabla_{(B^\sharp {\cal E}_i  )^\top} {\cal E}_j, E_a\>
 = \sum B({\cal E}_i, E_b) \< U_i {\cal E}_j,(A + T^\sharp )_j E_b\> \\
 && = \frac{1}{2} \sum B({\cal E}_i, E_b)\big(\<\Theta_a {\cal E}_i +\Theta_i E_a, {\cal E}_j\>
 -\<\Theta_a {\cal E}_j +\Theta_j E_a, {\cal E}_i\>\big)\<E_a,(A + T^\sharp )_j E_b\>,\\
 &&\hskip-4mm -\sum \<U_i {\cal E}_j, E_a\> \<\nabla_{i}((B^\sharp {\cal E}_j )^\top), E_a\>
 = \sum \big[\<\nabla_i(B({\cal E}_j, E_a ) U^*_j E_a) , {\cal E}_i\>
 - B({\cal E}_j, E_a ) \<\nabla_i  U^*_j E_a, {\cal E}_i \> \big],
\end{eqnarray*}
where
 $\<U_j {\cal E}_i, E_a\> = \<U^*_j E_a, {\cal E}_i \>$.
Note that
\begin{equation*}
 \<U^*_j E_a, {\cal E}_i\>
 = \frac{1}{2}\,\<\Theta_a^* {\cal E}_j +\Theta^{\wedge*}_a {\cal E}_j -\Theta_a {\cal E}_j -\Theta^\wedge_a {\cal E}_j, {\cal E}_i\>,
\end{equation*}
thus, using
(1,2)-tensor $F$ defined in \eqref{formulaF},
we can write
\begin{equation*}
 -\sum \<U_i {\cal E}_j, E_a\> \<\nabla_{i}((B^\sharp {\cal E}_j )^\top), E_a\>
 = \Div^\perp (\< B_{|V}, F \>) -\< B_{| V}, \Div^\perp F \> .
\end{equation*}
For the first four terms of $\dt \< \Theta, {\tilde T}^\sharp \>$, see \eqref{Enew-6terms}, we obtain:
\begin{eqnarray*}
 &&
 B(\Theta_a {\cal E}_i +\Theta_i E_a, {\cal E}_j)
 = B({\cal E}_j, {\cal E}_k)\<\Theta_a {\cal E}_i +\Theta_i E_a, {\cal E}_k\>
 +B({\cal E}_j, E_b) \<\Theta_a {\cal E}_i +\Theta_i E_a, E_b\>
 ,\\
 && \sum \<\Theta_a {\cal E}_i +\Theta_i E_a, {\cal E}_j\> B({\tilde T}({\cal E}_i, {\cal E}_j ), E_a) =
 \sum B(E_a, E_b)\<\Theta_a {\cal E}_i +\Theta_i E_a, {\cal E}_j\> \<{\tilde T}({\cal E}_i, {\cal E}_j ), E_b\> = 0,\\
 && \<\Theta_a (\dt {\cal E}_i) +\Theta_{\dt {\cal E}_i} E_a, {\cal E}_j\>
 = -\frac{1}{2}\,B({\cal E}_i, {\cal E}_k) \<\Theta_a {\cal E}_k +\Theta_k E_a, {\cal E}_j\>
 -B({\cal E}_i, E_b) \<\Theta_a E_b +\Theta_b E_a, {\cal E}_j\>
 ,\\
 && \<\Theta_a {\cal E}_i +\Theta_i E_a, \dt {\cal E}_j\>
 = -\frac{1}{2}\,B({\cal E}_j, {\cal E}_k) \<\Theta_a {\cal E}_i +\Theta_i E_a, {\cal E}_k\>
 -B({\cal E}_j, E_b) \<\Theta_a {\cal E}_i +\Theta_i E_a, E_b\>
 .
\end{eqnarray*}
Using \eqref{Eq-dt-Theta},
we consider
\begin{eqnarray*}
 && \< (\dt \Theta)_a {\cal E}_i + (\dt \Theta)_i E_a, {\cal E}_j\>
 =\<-\I^*_a B^\sharp{\cal E}_i + B^\sharp\I^*_a {\cal E}_i -\I^{*}_i\,B^\sharp E_a + B^\sharp\I^{*}_i {E}_a, {\cal E}_j\> \\
 && +\<-\I^*_i B^\sharp E_a + B^\sharp\I^*_i E_a -\I^{*}_a\,B^\sharp{\cal E}_i + B^\sharp\I^{*}_a {\cal E}_i, {\cal E}_j \>,
\end{eqnarray*}
which can be simplified to the following:
\begin{eqnarray*}
 && \<(\dt\Theta)_a {\cal E}_i + (\dt \Theta)_i E_a, {\cal E}_j\> =
 -2\sum\nolimits_{\,k} B({\cal E}_i, {\cal E}_k ) \<\I^*_a {\cal E}_k, {\cal E}_j\> \\
 && -2\sum\nolimits_{\,b} B({\cal E}_i, E_b ) \<\I^*_a E_b, {\cal E}_j\>
 + 2\sum\nolimits_{\,k} B({\cal E}_k , {\cal E}_j)\<\I^*_a {\cal E}_i, {\cal E}_k\>
 + 2\sum\nolimits_{\,b} B({\cal E}_j, E_b) \<\I^*_a {\cal E}_i, E_b\> \\
 &&
 -2\sum\nolimits_{\,k} B({\cal E}_k, E_a) \<\I^{*}_{i} {\cal E}_k, {\cal E}_j\>
 + 2\sum\nolimits_{\,b} B({\cal E}_j, E_b) \<\I^{*}_i {E}_a, E_b\>
 + 2\sum\nolimits_{\,k} B({\cal E}_k, {\cal E}_j) \<\I^{*}_i {E}_a, {\cal E}_k\> .
\end{eqnarray*}
Hence, the sixth term in $\dt \< \Theta, {\tilde T}^\sharp \>$ is:
\begin{eqnarray*}
 && \sum \< (\dt \Theta)_a {\cal E}_i + (\dt \Theta)_i E_a, {\cal E}_j\> \<{\tilde T}({\cal E}_i, {\cal E}_j ), E_a\>
 = 
 \sum\big[ -2 B({\cal E}_i, {\cal E}_k ) \<{\tilde T}({\cal E}_i, {\cal E}_j ), E_a\>  \<\I^*_a {\cal E}_k, {\cal E}_j\> \\
 && -2 B({\cal E}_i, E_b ) \<{\tilde T}({\cal E}_i, {\cal E}_j ), E_a\>  \<\I^*_a E_b, {\cal E}_j\>
 + 2 B({\cal E}_k , {\cal E}_j) \<{\tilde T}({\cal E}_i, {\cal E}_j ), E_a\>  \<\I^*_a {\cal E}_i, {\cal E}_k\>  \\
 && +2 B({\cal E}_j, E_b) \<{\tilde T}({\cal E}_i, {\cal E}_j ), E_a\> \<\I^*_a {\cal E}_i, E_b\>
 -2 B({\cal E}_k, E_a) \<{\tilde T}({\cal E}_i, {\cal E}_j ), E_a\> \<\I^{*}_{i} {\cal E}_k, {\cal E}_j\> \\
 && +2 B({\cal E}_j, E_b) \<{\tilde T}({\cal E}_i, {\cal E}_j ), E_a\> \<\I^{*}_i {E}_a, E_b\>
 +2 B({\cal E}_k, {\cal E}_j) \<{\tilde T}({\cal E}_i, {\cal E}_j ), E_a\> \<\I^{*}_i {E}_a, {\cal E}_k\> \big].
\end{eqnarray*}
Finally,
we get \eqref{dtThetatildeT}.

\textbf{Proof of \eqref{dtThetatildeA}}.
We have
\begin{equation*}
 \< \Theta, {\tilde A} \>
 = \sum \<\Theta_i E_a +\Theta_a {\cal E}_i, {\cal E}_j\> \<{\tilde h}({\cal E}_i, {\cal E}_j ), E_a\>.
\end{equation*}
Hence
\begin{eqnarray*}
 && \dt \<\Theta, {\tilde A} \> =
 \sum \big[ B(\Theta_i E_a +\Theta_a {\cal E}_i, {\cal E}_j) \<{\tilde h}({\cal E}_i, {\cal E}_j ), E_a\>
 +\<\Theta_i E_a +\Theta_a {\cal E}_i, {\cal E}_j\> B({\tilde h}({\cal E}_i, {\cal E}_j ), E_a) \\
 && +\<\Theta_{\dt {\cal E}_i} E_a +\Theta_a (\dt {\cal E}_i ), {\cal E}_j\> \<{\tilde h}({\cal E}_i, {\cal E}_j ), E_a\>
 +\<\Theta_i E_a +\Theta_a {\cal E}_i, \dt {\cal E}_j\> \<{\tilde h}({\cal E}_i, {\cal E}_j ), E_a\> \\
 && +\<\Theta_i E_a +\Theta_a {\cal E}_i, {\cal E}_j\> \<\dt {\tilde h}({\cal E}_i, {\cal E}_j ), E_a\>
 +\<(\dt \Theta)_i E_a + (\dt \Theta)_a {\cal E}_i, {\cal E}_j\> \<{\tilde h}({\cal E}_i, {\cal E}_j ), E_a\> \big].
\end{eqnarray*}
We shall denote by (h) the fifth of the above 6 terms, and write it as sum of seven terms (h1) to (h7):
\begin{eqnarray*}
 && \sum \<\Theta_i E_a +\Theta_a {\cal E}_i, {\cal E}_j\> \<\dt {\tilde h}({\cal E}_i, {\cal E}_j ), E_a\> \\
 && = \sum\big[ -\frac{1}{2}(\<\Theta_i E_a +\Theta_a {\cal E}_i, {\cal E}_j\>
 +\<\Theta_j E_a +\Theta_a {\cal E}_j, {\cal E}_i \>)\nabla_{a} B({\cal E}_i, {\cal E}_j)\\
 && -\frac{1}{2}(\<\Theta_i E_a +\Theta_a {\cal E}_i, {\cal E}_j\> +\<\Theta_j E_a +\Theta_a {\cal E}_j, {\cal E}_i \>)
 \<{\tilde h}( B^\sharp {\cal E}_i, {\cal E}_j) + {\tilde h}({\cal E}_i, B^\sharp {\cal E}_j ), E_a\>  \\
 && -\frac{1}{2}\big(\<\Theta_i E_a +\Theta_a {\cal E}_i, {\cal E}_j\> +\<\Theta_j E_a +\Theta_a {\cal E}_j, {\cal E}_i \>\big)
 \<\nabla_{i}((B^\sharp {\cal E}_j )^\top) +\nabla_{j}((B^\sharp {\cal E}_i )^\top ), E_a\> \\
 && -\frac{1}{2}(\<\Theta_i E_a +\Theta_a {\cal E}_i, {\cal E}_j\> +\<\Theta_j E_a +\Theta_a {\cal E}_j, {\cal E}_i \>)
 \<\nabla_{( B^\sharp {\cal E}_j )^\top} {\cal E}_i +\nabla_{(B^\sharp {\cal E}_i )^\top} {\cal E}_j, E_a\> \\
 && +\frac{1}{2}(\<\Theta_i E_a +\Theta_a {\cal E}_i, {\cal E}_j\> +\<\Theta_j E_a +\Theta_a {\cal E}_j, {\cal E}_i\>)
 (\nabla_{i} B({\cal E}_j,E_a) +\nabla_{j} B({\cal E}_i, E_a) )\\
 && -\frac{1}{2}(\<\Theta_i E_a +\Theta_a {\cal E}_i, {\cal E}_j\> +\<\Theta_j E_a +\Theta_a {\cal E}_j, {\cal E}_i \>)
 (B(\nabla_{i} E_a, {\cal E}_j) + B(\nabla_{j} E_a, {\cal E}_i )) \\
 && +\frac{1}{2}(\<\Theta_i E_a +\Theta_a {\cal E}_i, {\cal E}_j\> +\<\Theta_j E_a +\Theta_a {\cal E}_j, {\cal E}_i \>)
 ( B(\nabla_{a} {\cal E}_i, {\cal E}_j ) + B(\nabla_{a} {\cal E}_j, {\cal E}_i )) \big].
\end{eqnarray*}
We have for the term (h1) above:
\begin{eqnarray*}
 && -\frac{1}{2} \sum(\<\Theta_i E_a +\Theta_a {\cal E}_i, {\cal E}_j\> +\<\Theta_j E_a +\Theta_a {\cal E}_j, {\cal E}_i \>)
 \nabla_{a} B({\cal E}_i, {\cal E}_j )\\
 && = \sum B({\cal E}_i, {\cal E}_j) \<\nabla_a(\Theta^*_i {\cal E}_j +\Theta^{\wedge*}_i {\cal E}_j), E_a \>
 -\sum \<\nabla_a\big(B({\cal E}_i, {\cal E}_j) (\Theta^*_i {\cal E}_j +\Theta^{\wedge*}_i {\cal E}_j) \big), E_a\>,
\end{eqnarray*}
which can be written as
\begin{eqnarray*}
 && -\frac{1}{2} \sum(\<\Theta_i E_a +\Theta_a {\cal E}_i, {\cal E}_j\> +\<\Theta_j E_a +\Theta_a {\cal E}_j, {\cal E}_i \>)
 \nabla_{a} B({\cal E}_i, {\cal E}_j )\\
 && = -2 \Div^\top \< B, L \> + 2 \< B, \Div^\top L \> .
\end{eqnarray*}
For (h2):
\begin{eqnarray*}
 && -\frac{1}{2} \sum(\<\Theta_i E_a +\Theta_a {\cal E}_i, {\cal E}_j\> +\<\Theta_j E_a +\Theta_a {\cal E}_j, {\cal E}_i \>)
 \<{\tilde h}( B^\sharp {\cal E}_i, {\cal E}_j) + {\tilde h}({\cal E}_i, B^\sharp {\cal E}_j ), E_a\> \\
 && = -\sum B({\cal E}_i, {\cal E}_k)
 (\<\Theta_i E_a +\Theta_a {\cal E}_i, {\cal E}_j\> +\<\Theta_j E_a +\Theta_a {\cal E}_j, {\cal E}_i \>)
  \<{\tilde h}({\cal E}_k, {\cal E}_j ), E_a\>.
\end{eqnarray*}
Note that for (h3) we can assume $\nabla_X E_a \in \mD$ for all $X \in TM$ at a point, where we compute the formula,
and hence
\begin{eqnarray*}
 && -\frac{1}{2} \sum(\<\Theta_i E_a +\Theta_a {\cal E}_i, {\cal E}_j\> +\<\Theta_j E_a +\Theta_a {\cal E}_j, {\cal E}_i \>)
 \<\nabla_{i}((B^\sharp {\cal E}_j )^\top) +\nabla_{j}((B^\sharp {\cal E}_i )^\top ), E_a\>  \\
 && = -\sum(\<\Theta_i E_a +\Theta_a {\cal E}_i, {\cal E}_j\> +\<\Theta_j E_a +\Theta_a {\cal E}_j, {\cal E}_i \>)
 \nabla_{i} B( E_a, {\cal E}_j).
\end{eqnarray*}
For (h5), analogously,
\begin{eqnarray*}
 && \frac{1}{2} \sum(\<\Theta_i E_a +\Theta_a {\cal E}_i, {\cal E}_j\> +\<\Theta_j E_a +\Theta_a {\cal E}_j, {\cal E}_i \>)
 (\nabla_{i} B({\cal E}_j,E_a) +\nabla_{j} B({\cal E}_i, E_a))\\
 &&  = \sum(\<\Theta_i E_a +\Theta_a {\cal E}_i, {\cal E}_j\> +\<\Theta_j E_a +\Theta_a {\cal E}_j, {\cal E}_i \>)
 \nabla_{i} B({\cal E}_j, E_a),
\end{eqnarray*}
so (h3)+(h5)=0. For (h4) we have
\begin{eqnarray*}
 && -\frac{1}{2} \sum(\<\Theta_i E_a +\Theta_a {\cal E}_i, {\cal E}_j\> +\<\Theta_j E_a +\Theta_a {\cal E}_j, {\cal E}_i \>)
 \<\nabla_{( B^\sharp {\cal E}_j )^\top} {\cal E}_i +\nabla_{(B^\sharp {\cal E}_i )^\top} {\cal E}_j, E_a\> \\
 && = \sum B({\cal E}_j,E_b)
 (\<\Theta_i E_a +\Theta_a {\cal E}_i, {\cal E}_j\> +\<\Theta_j E_a +\Theta_a {\cal E}_j, {\cal E}_i \>)
  \< (A_i + T^\sharp_i) E_b, E_a\>.
\end{eqnarray*}
For (h6) term we have
\begin{eqnarray*}
 && -\frac{1}{2} \sum(\<\Theta_i E_a +\Theta_a {\cal E}_i, {\cal E}_j\> +\<\Theta_j E_a +\Theta_a {\cal E}_j, {\cal E}_i \>)
 ( B(\nabla_{i} E_a, {\cal E}_j) + B(\nabla_{j} E_a, {\cal E}_i ))\\
 && =  \sum B({\cal E}_k, {\cal E}_j)
 (\<\Theta_i E_a +\Theta_a {\cal E}_i, {\cal E}_j\> +\<\Theta_j E_a +\Theta_a {\cal E}_j, {\cal E}_i \>)
  \<({\tilde A}_a -{\tilde T}^\sharp_a) {\cal E}_k, {\cal E}_i\>,
\end{eqnarray*}
and (h7) term can be written as
\begin{eqnarray*}
 && \frac{1}{2}\sum(\<\Theta_i E_a +\Theta_a {\cal E}_i, {\cal E}_j\> +\<\Theta_j E_a +\Theta_a {\cal E}_j, {\cal E}_i \>)
 (B(\nabla_{a} {\cal E}_i, {\cal E}_j ) + B(\nabla_{a} {\cal E}_j, {\cal E}_i )) \\
 && = -\sum B(E_b, {\cal E}_i)
 (\<\Theta_i E_a +\Theta_a {\cal E}_i, {\cal E}_j\> +\<\Theta_j E_a +\Theta_a {\cal E}_j, {\cal E}_i\>)
  \< (A_j - T^\sharp_j) E_b, E_a\> .
\end{eqnarray*}
Now we compute other terms of $\dt \< \Theta, {\tilde A} \>$. Recall that those 6 terms are
\begin{eqnarray*}
 && \dt\<\Theta, {\tilde A}\>=
 \sum\big[ B(\Theta_i E_a+\Theta_a {\cal E}_i,{\cal E}_j)\<{\tilde h}({\cal E}_i,{\cal E}_j),E_a\> \\
 && +\,\<\Theta_i E_a +\Theta_a {\cal E}_i, {\cal E}_j\> B({\tilde h}({\cal E}_i, {\cal E}_j ), E_a)
 +\<\Theta_{\dt {\cal E}_i} E_a +\Theta_a (\dt {\cal E}_i ), {\cal E}_j\> \<{\tilde h}({\cal E}_i, {\cal E}_j ), E_a\> \\
 && +\,\<\Theta_i E_a +\Theta_a {\cal E}_i, \dt {\cal E}_j\> \<{\tilde h}({\cal E}_i, {\cal E}_j ), E_a\>
 +\<\Theta_i E_a +\Theta_a {\cal E}_i, {\cal E}_j\> \<\dt {\tilde h}({\cal E}_i, {\cal E}_j ), E_a\> \\
 && +\,\< (\dt \Theta)_i E_a + (\dt \Theta)_a {\cal E}_i, {\cal E}_j\> \<{\tilde h}({\cal E}_i, {\cal E}_j ), E_a\> \big].
\end{eqnarray*}
For the first and second terms of the above $\dt \< \Theta, {\tilde A} \>$ we have
\begin{eqnarray*}
 && B(\Theta_i E_a +\Theta_a {\cal E}_i, {\cal E}_j)
 =\sum B({\cal E}_j, {\cal E}_k)\<\Theta_i E_a +\Theta_a {\cal E}_i,{\cal E}_k\>
 + \sum B({\cal E}_j, E_b) \<\Theta_i E_a +\Theta_a {\cal E}_i, E_b\> ,\\
 && \sum \<\Theta_i E_a +\Theta_a {\cal E}_i, {\cal E}_j\> B({\tilde h}({\cal E}_i, {\cal E}_j ), E_a) = 0,
\end{eqnarray*}
because $B=0$ on $\widetilde{\mD} \times \widetilde{\mD}$.
 For the third and fourth terms we have:
\begin{eqnarray*}
 && \<\Theta_{\dt {\cal E}_i} E_a {+}\Theta_a (\dt {\cal E}_i ), {\cal E}_j\>
 = \sum[-\frac{1}{2}\,B({\cal E}_i, {\cal E}_k)\<\Theta_k E_a +\Theta_a{\cal E}_k, {\cal E}_j\>
 -B({\cal E}_i, E_b) \<\Theta_b E_a +\Theta_a  E_b, {\cal E}_j\>
 ],\\
 && \<\Theta_i E_a +\Theta_a {\cal E}_i, \dt {\cal E}_j\>
 = \sum\big[-\frac{1}{2}\,B({\cal E}_j, {\cal E}_k)\<\Theta_i E_a +\Theta_a {\cal E}_i, {\cal E}_k\>
 -B({\cal E}_j, E_b) \<\Theta_i E_a +\Theta_a {\cal E}_i,  E_b\>
 \big].
\end{eqnarray*}
For the sixth term, note that
\begin{eqnarray*}
 && \sum \<(\dt\Theta)_a\, {\cal E}_i + (\dt\Theta)_i E_a, {\cal E}_j\> \<{\tilde h}({\cal E}_i, {\cal E}_j), E_a\> =
 \sum\big[ -2 B({\cal E}_i, {\cal E}_k ) \<{\tilde h}({\cal E}_i, {\cal E}_j ), E_a\> \<\I^*_a {\cal E}_k, {\cal E}_j\> \\
 && -\,2 B({\cal E}_i, E_b ) \<{\tilde h}({\cal E}_i, {\cal E}_j ), E_a\> \<\I^*_a E_b, {\cal E}_j\>
 + 2 B({\cal E}_k , {\cal E}_j) \<{\tilde h}({\cal E}_i, {\cal E}_j ), E_a\> \<\I^*_a {\cal E}_i, {\cal E}_k\>  \\
 && +\,2 B({\cal E}_j, E_b) \<{\tilde h}({\cal E}_i, {\cal E}_j ), E_a\> \<\I^*_a {\cal E}_i, E_b\>
 -2 B({\cal E}_k, E_a) \<{\tilde h}({\cal E}_i, {\cal E}_j ), E_a\> \<\I^{*}_{i} {\cal E}_k, {\cal E}_j\> \\
 && +\,2 B({\cal E}_j, E_b) \<{\tilde h}({\cal E}_i, {\cal E}_j ), E_a\> \<\I^{*}_i {E}_a, E_b\>
 + 2 B({\cal E}_k, {\cal E}_j) \<{\tilde h}({\cal E}_i, {\cal E}_j ), E_a\> \<\I^{*}_i {E}_a, {\cal E}_k\> \big].
\end{eqnarray*}
Finally,
we get \eqref{dtThetatildeA}.

\textbf{Proof of \eqref{dttracetopI} and \eqref{dttraceperpI}} is straightforward.

\textbf{Proof of \eqref{dtIEaH} and \eqref{dtIeiH}}.
The variation formulas for these terms appear in the following part of $Q$ in \eqref{E-defQ}:
\begin{eqnarray*}
 && -\<\tr^\top\I -\tr^\bot\I +\tr^\bot\I^* -\tr^\top\I^*,\, \tH - H\>  \\
 && = \< \tr^\top(\I^* -\I), \tH - H\> +\<\tr^\bot(\I^* -\I), \tH - H\> .
\end{eqnarray*}
We have
\begin{eqnarray*}
 && \dt \<\, \tr^\top(\I^* -\I), \tH - H\>
 = B( \tr^\bot(\I^* -\I), \tH - H) \\
 && +\sum \< (\dt \I^*)_k {\cal E}_k, \tH - H\>
 +\< \tr^\bot(\I^* -\I), \dt \tH\> - \< \tr^\bot(\I^* -\I), \dt H\> ,\\
 && B( \tr^\top(\I^* -\I), \tH - H)
 = \sum B({\cal E}_i, E_b) \big(\<\tH, E_b\> \< \tr^\bot(\I^* -\I), {\cal E}_i\> \\
 && - \<H, {\cal E}_i\> \< \tr^\bot(\I^* -\I), E_b\> \big)
 -\sum B({\cal E}_i, {\cal E}_j) \<H, {\cal E}_j\> \< \tr^\bot(\I^* -\I), {\cal E}_i\> .
\end{eqnarray*}
Then we have
\begin{eqnarray*}
 && \sum \< (\dt \I^*)_a E_a, \tH - H\>
%
 = \sum B({\cal E}_i, E_b) \big(\<\I^*_b {\cal E}_i, \tH - H\>
  + \<\tr^\top\I^*, E_b\> \<{\cal E}_i, H\> \\
 && - \<\tr^\top\I^*, {\cal E}_i\> \<E_b, \tH\>\big)
 +\sum B({\cal E}_i, {\cal E}_j) \<\tr^\top\I^*, {\cal E}_j\> \<{\cal E}_i, H\>, \\
 && \sum \< (\dt \I^*)_i {\cal E}_i, \tH - H\>
%
 = \sum B({\cal E}_i, {\cal E}_j) \big( \<\I^*_j {\cal E}_i, \tH - H\>
 + \<\tr^\bot\I^*, {\cal E}_i\> \< H, {\cal E}_j\> \big)\\
 && +\sum B({\cal E}_i, E_b) \big(\<\I^*_i E_b, \tH - H \>  + \<\tr^\bot\I^*, E_b\> \< H, {\cal E}_i\>
 - \<\tr^\bot\I^*, {\cal E}_i\> \<\tH, E_b\> \big) .
\end{eqnarray*}
Next, we shall use equations (20) and (21) from \cite{rz-2}:
\begin{eqnarray*}
 \<\dt\tilde H, X\> \eq \<\,2\<\theta, X^\top\>,\,B\> -\frac12\,X^\top(\tr_{\mD}B),\\
 \<\dt H, X\> \eq \Div(B^\sharp (X^\perp) )^\top +\< B^\sharp (X^\perp), \tilde H\> -\<B^\sharp (X^\perp), H\>
 -\< {\tilde \delta}_{X^\perp}, B \> \nonumber \\
 && -\< \< {\tilde \alpha} -{\tilde \theta}, X^\perp \>, B \> - B(H, X^\top).
\end{eqnarray*}
We have
\begin{eqnarray*}
 && \< \tr^\top(\I^* -\I), \dt \tH\>
 = 2 \sum B({\cal E}_i, E_b) \<\T_i E_b , \tr^\bot(\I^* -\I)\> \\
 && -\frac{1}{2} \sum B({\cal E}_i, {\cal E}_j) \<{\cal E}_i, {\cal E}_j\> \Div((\tr^\bot(\I^* -\I))^\top)
 -\Div (\frac12\, (\tr_{\mD}B) (\tr^\bot(\I^* -\I) )^\top) .
\end{eqnarray*}
Finally,
\begin{eqnarray*}
 &&\qquad \<\dt H, \tr^\top(\I^* -\I)\>
 =\Div(( B^\sharp((\tr^\top(\I^* -\I))^\perp) )^\top) \\
 && +\sum \big[ B({\cal E}_i, E_b) \<{\cal E}_i, \tr^\top(\I^* -\I)\> \<\tilde H, E_b\>
 - B({\cal E}_i, {\cal E}_j) \<H, {\cal E}_j\> \<\tr^\top(\I^* -\I), {\cal E}_i\> \\
 && - B({\cal E}_i, E_b) \<H, {\cal E}_i\> \<\tr^\top(\I^* -\I), E_b\>
 - B({\cal E}_i, E_b) \<\nabla_b((\tr^\top(\I^* -\I) )^\perp ), {\cal E}_i\> \\
 && - B({\cal E}_i, E_b) \<(\tA_b -\tT_b){\cal E}_i, \tr^\top(\I^* -\I)\>
 - B({\cal E}_i, E_b) \< H, {\cal E}_i\> \<\tr^\top(\I^* -\I), E_b\> \big],\\
 && \qquad \<\dt H, \tr^\bot(\I^* -\I) \> = \Div(( B^\sharp((\tr^\bot(\I^* -\I))^\perp) )^\top) \\
 && +\sum \big[ B({\cal E}_i, E_b) \<{\cal E}_i, \tr^\bot(\I^* -\I)\> \<\tilde H, E_b\>
 - B({\cal E}_i, {\cal E}_j) \<H, {\cal E}_j\> \<\tr^\bot(\I^* -\I), {\cal E}_i\> \\
 && - B({\cal E}_i, E_b) \<H, {\cal E}_i\> \<\tr^\bot(\I^* -\I) , E_b\>
 - B({\cal E}_i, E_b) \<\nabla_b((\tr^\bot(\I^* -\I))^\perp ), {\cal E}_i\> \\
 && - B({\cal E}_i, E_b) \<(\tA_b -\tT_b){\cal E}_i, \tr^\bot(\I^* -\I)\>
 - B({\cal E}_i, E_b) \< H, {\cal E}_i\> \<\tr^\bot(\I^* -\I), E_b\> \big].
\end{eqnarray*}
Summing
$\dt \< \tr^\top(\I^* -\I), \tH - H \>$ and $\dt \<\tr^\bot(\I^* -\I), \tH - H\>$,
we obtain \eqref{dtIEaH} and \eqref{dtIeiH}.
\end{proof}



We have the following results for critical metric connections and 
$g^\perp$-variations (see Definition \ref{defintionvariationsofg}), that can be considered as a special case of Lemma~\ref{L-dT-3}.


\begin{lemma} \label{L-dT-metric}
Let $\widetilde{\mD}$ and $\mD$ be both totally umbilical distributions on $(M,g)$.
Let $g_t$ be a
$g^\perp$-variation of metric $g$ and $\nabla +\I$ be a metric connection: $\I^* = -\I$.
If $\I$ is a critical point for \eqref{actiongSmix} with fixed $g$, then, up to divergences of compactly supported vector fields, the following 
formulas~hold:
\begin{subequations}
\begin{eqnarray}\label{metriccon1}
 && \dt \<\tr^\top(\I^* -\I),\ \tH - H\> = \< B,\ {3}\,H^\flat \odot(\tr^\top \I)^{\perp \flat}
 +\frac{p-1}{p}\, (\Div\tH) g^\perp \> ,\\
\label{metriccon2}
 && \dt \<\tr^\bot(\I^* -\I), \tH{-}H\> =\< B,\, 3\frac{n{-}1}{n} H^\flat\otimes H^\flat -\frac{1}{2}\<\phi, \tH{-}H\> +\Div((\tr^\perp \I)^\top) g^\perp\> , \\
 \label{metriccon3}
 && \dt \<\tr^\top \I, \tr^\perp \I^*\> = 0 ,\\
\label{metriccon4}
 && \dt \<\tr^\top \I^*, \tr^\perp \I\> =  \< B,\ \frac{1}{2}\,\<\,\phi, \tr^\top \I\,\>\, \> , \\
 \label{metriccon5}
 && \dt \< \Theta, A \> = \< B,\ \frac{2}{n}\,H^\flat \odot (\tr^\top \I)^{\perp \flat} \> , \\
 \label{metriccon6}
 && \dt \< \Theta, {\tilde A} \> =  \< B,\ \frac{1}{p}\, \< \phi, {\tilde H} \> + 2 \Div L^\top + 8 \chi + 8 \widetilde{\cal{T}}^\flat \>,\\
\label{metriccon7} 
 && \dt \< \Theta, T^\sharp \> = \< B,\, \,\Upsilon_{T,T} \>, \\
 \label{metriccon8}
 && \dt \< \Theta, {\tilde T}^\sharp \> = \< B,\ 12\,\widetilde{\cal{T}}^\flat + 2 \chi  \> ,\\
 \label{metriccon9}
 && \dt \< \I^*, \I^\wedge \>_{\,|\,V} = \< B,\, \frac12\,\Upsilon_{T,T} -2 \widetilde{\cal{T}}^\flat -\chi \> .
\end{eqnarray}
\end{subequations}
\end{lemma}

\begin{proof}
First we adapt the results of Lemma~\ref{L-dT-3} to the case of
$g^\perp$-variation and totally umbilical distributions $\widetilde{\mD}$ and $\mD$. Then we shall use the Euler-Lagrange equations 
(\ref{ELconnectionNew1}-j), which for a metric connection have the following form:
\begin{subequations}
\begin{eqnarray} \label{metricconcrit1}
 && (\I_V\, U -\I_U\, V)^\top = 2\,{\tilde T}(U, V), \\ \label{metricconcrit2}
 && \<(\I_U-\T_U) X,\, Y\> = 0, \\ \label{metricconcrit3}
 && (\tr^\perp \I)^\perp = \frac{n-1}{n} H, \\ \label{metricconcrit4}
 && (\I_Y\, X -\I_X\, Y)^\perp = 2\, T(X, Y), \\ \label{metricconcrit5}
 && \<(\I_X-\T_X)\,U,\, V\> = 0, \\
\label{metricconcrit6}
 && (\tr^\top \I)^\top = \frac{p-1}{p}\,\tH,
\end{eqnarray}
\end{subequations}
for all $X,Y\in\widetilde\mD$ and $U,V\in\mD$, 
and
\begin{equation*}
 (\tr^\perp \I)^\top = -\tH\quad {\rm for}\quad n>1,\qquad
 (\tr^\top \I)^\perp = -H\quad {\rm for}\quad p>1.
\end{equation*}
The last two equations require special assumptions on dimensions of the distributions -- we shall not use them in this proof.
%
For metric connections we also have
\[
\Theta = \Theta^\wedge = 2\,(\I + \I^\wedge).
\]

For metric connections, 
$g^\perp$-variations of metric and totally umbilical distributions,
using \eqref{metricconcrit6}, we obtain 
\begin{eqnarray*}
 && \dt \<\tr^\top(\I^* -\I), \tH -H\>
 = \sum B({\cal E}_i, {\cal E}_j) \big(\,3\<H, {\cal E}_j\> \<\tr^\top\I, {\cal E}_i\> \\
 && +\,\frac{p-1}{p} \delta_{ij} \Div\tH\,\big)
  +\Div\big(\frac{p-1}{p}(\tr_{\mD}B)\tH - 2( B^\sharp (\tr^\top\I)^\perp )^\top\big) .
\end{eqnarray*}
Writing divergence of compactly supported vector field as $\Div Z$, we finally get
\begin{equation*}
 \dt \<\tr^\top(\I^* -\I), \tH -H\> = \sum B({\cal E}_i, {\cal E}_j) \big(\,3\,\<H, {\cal E}_j\> \<\tr^\top\I, {\cal E}_i\>
  +\frac{p-1}{p}\, \delta_{ij} \Div\tH \big) +\Div Z .
\end{equation*}
Without explicitly using the orthonormal frame, we can write the above as \eqref{metriccon1}.

For metric connections, 
$g^\perp$-variations of metric and totally umbilical distributions,
using \eqref{metricconcrit3}, 
we have:
\begin{eqnarray*}
 && \dt \<\tr^\bot(\I^* -\I), \tH -H\>
 =\sum B({\cal E}_i, {\cal E}_j) \big(\,3\,\frac{n-1}{n}\,\<H, {\cal E}_j\> \<H, {\cal E}_i\>
 -\frac{1}{2}\,\<\I_j {\cal E}_i, \tH -H\>  \\
 && -\,\frac{1}{2}\,\<\I_i {\cal E}_j, \tH -H\> +\delta_{ij} \Div((\tr^\bot\I)^\top)\,\big)
  +\Div\big((\tr_{\mD}B) (\tr^\bot\I)^\top -2( B^\sharp (\tr^\bot\I)^\perp )^\top\big) .
\end{eqnarray*}
Writing divergence of compactly supported vector field as $\Div Z$, we finally get
\begin{eqnarray*}
 && \dt \<\tr^\bot(\I^* -\I), \tH - H\> =
  \sum B({\cal E}_i, {\cal E}_j) \big(\,3 \frac{n-1}{n}\,\<H, {\cal E}_j\> \<H, {\cal E}_i\> \\
 && -\frac{1}{2}\,\<\I_j {\cal E}_i +\I_i {\cal E}_j, \tH - H\>
 +\delta_{ij} \Div( (\tr^\bot\I)^\top)\big) +\Div Z.
\end{eqnarray*}
Without explicitly using the orthonormal frame, we can write the above as \eqref{metriccon2}.

For metric connections, 
$g^\perp$-variations of metric and totally umbilical distributions:
\begin{equation*}
 \dt \<\tr^\top \I, \tr^\perp \I^*\>
 = \frac{1}{2} \sum B({\cal E}_i, {\cal E}_j )\<\tr^\top\I,\ \I^*_i {\cal E}_j -\I^*_j {\cal E}_i\> = 0,
\end{equation*}
as $B({\cal E}_i, {\cal E}_j )$ is symmetric and $\I^*_i {\cal E}_j -\I^*_j {\cal E}_i$ is antisymmetric in $i,j$.

For metric connections, 
$g^\perp$-variations of metric and totally umbilical distributions:
\begin{equation*}
 \dt \<\tr^\top \I^*, \tr^\perp \I\>
 =\frac{1}{2} \sum B({\cal E}_i, {\cal E}_j) \<\phi({\cal E}_i, {\cal E}_j ), \tr^\top \I \> .
\end{equation*}
Without explicitly using the orthonormal frame, we can write the above as \eqref{metriccon4}.

For metric connections, 
$g^\perp$-variations of metric and totally umbilical distributions,
using \eqref{metricconcrit2},
we have:
\begin{eqnarray*}
 \dt \<\Theta, A \>
 \eq \sum B({\cal E}_i, {\cal E}_j) \big(\,2\,\<{\cal E}_j, \Hn\> \<\tr^\top\I, {\cal E}_i\>
 - 4 \<{\cal E}_j, \Hn\> \<\T_i E_a, E_a\> \,\big)\\
\eq 2 \sum B({\cal E}_i, {\cal E}_j) \<{\cal E}_j, \Hn\> \<\tr^\top\I, {\cal E}_i\> .
\end{eqnarray*}
Without explicitly using the orthonormal frame, we can write the above as \eqref{metriccon5}.

For metric connections, 
$g^\perp$-variations of metric and totally umbilical distributions:
\begin{eqnarray*}
 && \dt \< \Theta, {\tilde A} \> = -2 \Div^\top \< B, L \> + 2 \< B, \Div^\top L \> \\
 && +\,4 \sum  B({\cal E}_i, {\cal E}_j) \big(\,
 \<\I_k {\cal E}_i + \I_i {\cal E}_k, E_a\> \<{\tilde T}^\sharp_a  {\cal E}_j, {\cal E}_k\>
 -2 \<\I_j E_a, {\cal E}_i\> \<\tHp, E_a\> \big) .
\end{eqnarray*}
Using symmetry $B(X,Y) = B(Y,X)$ for $X,Y \in \mathfrak{X}^\bot$, 
we obtain:
\begin{eqnarray*}
 && \dt\< \Theta, {\tilde A}\> =  -2 \Div^\top \< B, L \> + 2 \< B, \Div^\top L \>\\
 && + \sum B({\cal E}_i, {\cal E}_j) \big[\,\frac{1}{p}\,\<\I_j {\cal E}_i +\I_i {\cal E}_j, {\tilde H}\>
 + 4 \<\I_k {\cal E}_i +\I_i {\cal E}_k, {\tilde T}({\cal E}_j, {\cal E}_k)\> \,\big].
\end{eqnarray*}
Note that
\begin{equation*}
 \<\I_k {\cal E}_i +\I_i {\cal E}_k, {\tilde T}({\cal E}_j, {\cal E}_k)\>
 =
 \<\I_k E_a, {\cal E}_i\> \<\tT_a {\cal E}_k, {\cal E}_j\>
 -\<\I_i E_a, \tT_a {\cal E}_j\>
  .
\end{equation*}
By the above, we can write $\dt \< \Theta, {\tilde A} \>$
as
\begin{equation}\label{dttildeAmetricadaptedlemma}
 \dt \< \Theta, {\tilde A} \>  = -2 \Div^\top \< B, L \>
 +\< B,\ \frac{1}{p} \< \phi, {\tilde H} \> + 2 \Div^\top L - 4\,\psi
 + 4\sum\nolimits_{a,j} (\I_j E_a)^{\perp \flat} \odot (\tT_a {\cal E}_j)^{\perp \flat}
  \>,
\end{equation}
where 
\begin{equation*}
 \psi(X,Y) = \frac{1}{2}\sum\nolimits_{\,a} \big(\<\I_{X^\perp}\, E_a,\, \tT_a (Y^\perp)\> +\<\I_{Y^\perp}\, E_a,\, \tT_a (X^\perp)\> \big) .
\end{equation*}
 We claim that $\psi$ can be written in terms of tensor $\chi$ introduced in \eqref{E-chi}. Indeed, for arbitrary symmetric (0,2)-tensor $B : {\cal D} \times {\cal D} \rightarrow \mathbb{R}$ we have
\begin{equation*}
	\< B, \psi \>
	= \< B, \ -2 \widetilde{\cal{T}}^\flat - \sum\nolimits_{a,j} (\I_j E_a)^{\perp \flat} \odot (\tT_a {\cal E}_j)^{\perp \flat}
	\>.
\end{equation*}
Using \eqref{E-chi}, 
we obtain
\begin{equation}\label{psiandchi}
 \psi = -2\,\widetilde{\cal{T}}^\flat -\chi.
\end{equation}
Using the following computation:
\begin{eqnarray*}
	&& \< B, \Div^\top L \> = \< B, \Div^\top L^\top +\Div^\top L^\perp \>
	= \< B, \Div L^\top \> +\< B, \< L^\top, {\tilde H} \> \> -\< B, \< L^\perp, H \> \> ,\\
	&& \Div^\top \< B, L \> = \Div \< B, L \> -\Div^\perp \< B, L \>
	= \Div \< B, L \> +\< B, \< L^\top, {\tilde H} \> \> -\Div^\perp \< B, L^\perp \> ,
\end{eqnarray*}
we obtain
\begin{equation*}
	 -\Div^\top \< B, L \> +\< B, \Div^\top L \>
	%
	%
	= -\Div \< B, L^\top \> +\< B, \Div L^\top \> ,
\end{equation*}
which, together with \eqref{dttildeAmetricadaptedlemma}--\eqref{psiandchi}, up to divergence of a compactly supported vector field, yields~\eqref{metriccon6}.

For metric connections, 
$g^\perp$-variations of metric and totally umbilical distributions we have
\begin{equation*}
 \dt \< \Theta, T^\sharp \> = 2 \sum B({\cal E}_i, {\cal E}_j) \< T (E_a, E_b), {\cal E}_i\> \<\I_a {\cal E}_j, E_b\> .
\end{equation*}
Using \eqref{metricconcrit4}, 
we obtain:
\begin{equation*}
 \dt \< \Theta, T^\sharp \> = 2 \sum B({\cal E}_i, {\cal E}_j) \< T (E_a, E_b), {\cal E}_i\> \<T(E_a,E_b), {\cal E}_j\> .
\end{equation*}
Without explicitly using the orthonormal frame, we can write the above as \eqref{metriccon7}.

For metric connections, 
$g^\perp$-variations of metric and totally umbilical distributions we have
\begin{equation*}
 \dt\<\Theta, {\tilde T}^\sharp\> = -\sum B({\cal E}_i, {\cal E}_j)\<{\tilde T}({\cal E}_i,{\cal E}_k), E_a\>
 \big(\,4\<\I_a{\cal E}_j, {\cal E}_k\> + 4\<\I_j E_a, {\cal E}_k\> -2\,\<\I_k E_a, {\cal E}_j\> \,\big).
\end{equation*}
Using \eqref{metricconcrit5}, 
we obtain:
\begin{equation*}
 \dt\< \Theta, {\tilde T}^\sharp\> = \sum B({\cal E}_i, {\cal E}_j) \big(\,4\<(\tT_a)^2 {\cal E}_j, {\cal E}_i\>
 - 4 \<\I_j E_a, \tT_a {\cal E}_i\> -2 \<\tT_a {\cal E}_j, {\cal E}_i\> \<\I_j E_a, {\cal E}_k\> \,\big).
\end{equation*}
Next, we have
\begin{equation*}
 \dt \< \Theta, {\tilde T}^\sharp \> = \< B,\ 4 \widetilde{\cal{T}}^\flat - 4 \psi
 -2\sum\nolimits_{a,j} (\I_j E_a)^{\perp \flat} \odot (\tT_a {\cal E}_j)^{\perp \flat} \> .
\end{equation*}
 For metric connections, 
 $g^\perp$-variations of metric and totally umbilical distributions we have:
\[
 \dt \<\I^*, \I^\wedge\>_{\,|\,V} = \sum B({\cal E}_i, {\cal E}_j) \<\I_j E_a, \I_a {\cal E}_i\> .
\]
Using (\ref{metricconcrit2},d,e), 
we obtain the following:
\begin{equation*}
 \dt \<\I^*, \I^\wedge\>_{\,|\,V} = \sum B({\cal E}_i, {\cal E}_j)\,\big(\,\< T(E_a, E_b),{\cal E}_j\> \< T(E_a,E_b), {\cal E}_i\>
 +\<\I_j E_a, \tT_a {\cal E}_i\> \,\big).
\end{equation*}
We can write,
 $\dt \<\I^*, \I^\wedge\>_{\,|\,V} = \< B,\, \frac12\,\Upsilon_{T,T} +\psi\>$ ,
which, together with \eqref{psiandchi}, yields \eqref{metriccon9}.
\end{proof}

\begin{lemma}
\label{dtQadapted}
Let $g_t$ be a $g^\perp$-variation of $g\in{\rm Riem}(M,\widetilde{\mD},{\mD})$, let $\I$ be the contorsion tensor 
of a metric connection that is critical for \eqref{actiongSmix} with fixed $g$, and let $\widetilde{\mD}$ and $\mD$ be totally umbilical distributions.
Then, up to divergences of compactly supported vector fields, for $Q$ given by \eqref{E-defQ} we have
\begin{eqnarray*}
 && -\dt Q = \big\<
 \< \phi,\, \frac{p+2}{2\,p}\,\tH -\frac{1}{2} H +\frac{1}{2} \tr^\top \I \> -2 \Div \phi^\top + 7 \chi
 +\frac{3n+2}{n}\,H^\flat \odot (\tr^\top \I)^{\perp \flat} \\
 &&\quad -\Div( (\tr^\perp \I )^\top)\,g^\perp
 +\frac{p-1}{p} (\Div \tH) g^\perp - 3 \frac{n-1}{n} H^\flat \otimes H^\flat
+ 2 \widetilde{\cal{T}}^\flat +\frac{3}{2}\,\Upsilon_{T,T},\  B\big\> .
\end{eqnarray*}
\end{lemma}

\begin{proof}

Recall that
\begin{eqnarray*}
\nonumber
 && L(X,Y) = \frac{1}{4} (\Theta^*_{X^\perp} Y^\perp +\Theta^{\wedge*}_{X^\perp} Y^\perp +
 \Theta^*_{Y^\perp} X^\perp +\Theta^{\wedge*}_{Y^\perp} X^\perp) ,
\end{eqnarray*} 
and let
 $L^\perp(X,Y) =(L(X,Y) )^\perp$
 and
 $L^\top(X,Y) =(L(X,Y) )^\top$ for
 $X,Y \in \mathfrak{X}_M$.
We have $L = L^\top + L^\perp$.
Note that $\<L^\perp(X,Y), Z\> = \<L^\perp(X^\perp,Y^\perp), Z^\perp\>$ and for metric connections
\begin{eqnarray*}
	&& \<\I^{\wedge*}_X Y, Z\> = \<\I^\wedge_X Z, Y\> = \<\I_Z X, Y\> = -\<\I_Z Y, X\> \\
	&& = -\<\I^\wedge_Y Z, X\> = -\<\I^{\wedge*}_Y X, Z\> = -\<\I^{\wedge*\wedge}_X\, Y,Z\>,
\end{eqnarray*}
for all $X,Y,Z \in \mathfrak{X}_M$, so
\begin{equation*}
	4 \<L(X,Y), Z) = \< Z, \Theta^*_{X} Y +\Theta^{\wedge*}_{X} Y +\Theta^*_{Y} X +\Theta^{\wedge*}_{Y} X\>
	= -4\<Z, \I_X Y +\I^\wedge_X Y\>.
\end{equation*}
Hence,
 $L^\perp = -(\I +\I^\wedge)^\perp$ and $L^\top = -(\I +\I^\wedge)^\top$
and for metric connections we obtain $L = - \phi$, see \eqref{E-chi}, which together with Lemma~\ref{L-dT-metric} yields the claim.
\end{proof}

\begin{lemma} \label{lemmasemisymmetric}
Let $\bar\nabla$ be a semi-symmetric connection on $(M,g,\mD)$.
a)~Then \eqref{E-defQ} reduces to
\begin{equation} \label{QforUconnection}
 Q = (n-p) \< U , H - \tH \> + np \< U , U \> - n \< U^\perp , U^\perp \> - p \< U^\top , U^\top \>.
\end{equation}
b)~For any $g^\pitchfork$-variation of metric $g$ and $Q$ given by \eqref{QforUconnection} we have
\begin{eqnarray}\label{dtQgforUconnection}
\nonumber
 \dt Q(g_t) |_{\,t=0} \eq \<\, B ,\
 -(n-p) {\tilde \delta}_{U^\perp} - (n-p) \< {\tilde \alpha} - {\tilde \theta} , U^\perp \>
 +2(p-n) \< {\theta}, U^\top \> \nonumber \\
 &-& \!\frac12\,(p-n)( \Div U^\top ) g^\perp
 +n(p-1) U^{\perp \flat} \otimes U^{\perp \flat}
 +2p(n-1) U^{\top \flat} \odot U^{\perp \flat}\, \> .
\end{eqnarray}
\end{lemma}

\begin{proof}
a) From \eqref{Uconnection} we obtain
\begin{equation}\label{UconnectiontrtopI}
 \tr^\top \I
 = \sum\nolimits_a \<U, E_a\> E_a - \sum\nolimits_a \<E_a , E_a\> U = U^\top - n\,U.
\end{equation}
Similarly,
 $\tr^\perp \I = U^\perp - p\,U$.
We also have
\begin{equation}\label{UconnectionImixed}
 \I_a {\cal E}_i = \<U, {\cal E}_i\> E_a ,\quad \I_i E_a = \<U , E_a\> {\cal E}_i,
\end{equation}
so we obtain $\< \I , \I^\wedge \>_{| V} =0$. Next, we have
\begin{eqnarray*}
 \< \tr^\top \I - \tr^\perp \I , H - \tH \> \eq (p-n-1) \< U^\perp , H \> + (n-p-1) \< U^\top , \tH \>.
\end{eqnarray*}
We have
 $(\I + \I^\wedge)_i E_a
 = \<U, E_a\>{\cal E}_i + \<U, {\cal E}_i\> E_a$.
Also
\begin{eqnarray*}
 \< \tr^\top \I , \tr^\perp \I \>
 \eq np \< U , U \> - n \< U^\perp , U^\perp \> - p \< U^\top , U^\top \>.
\end{eqnarray*}
Thus,
 $\< \I + \I^\wedge, \tA - \tT + A - \T \>
 =\< H + \tH ,\ U \>$.
\ b)~By \cite[Lemma~3]{rz-2}, we have:
\begin{eqnarray*}
 \<U^\perp,\ \dt U^\perp\> \eq \<U^\perp,\ - (B^\sharp (U^\perp) )^\top\> =0,\\
 \< U^\top,\ \dt U^\top\> \eq \<U^\top,\,B^\sharp (U^\perp)\>
 = \< B,\, U^{\top \flat} \odot U^{\perp \flat}\, \> .
\end{eqnarray*}
Similarly, by \cite[Eq.~(20) and Eq.~(21)]{rz-2}, we have
\begin{eqnarray*}
 \<\dt \tH,\, U \> \eq \Div ( (\tr_{\cal D} B ) U^\top )
 + \<\,B,\, 2 \< {\theta}, U^\top \> - \frac{1}{2}\,(\Div U^\top ) g^\perp \>,\\
 \<\dt H , U \> \eq \Div((B^\sharp(U^\perp))^\top) + \<B, U^\perp\odot(\tH -H) - U^\top\odot H
 - {\tilde\delta}_{U^\perp} - \<{\tilde\alpha} - {\tilde\theta}, U^\perp\>\,\>.
\end{eqnarray*}
Omitting divergences of compactly supported vector fields and using $B|_{\widetilde{\cal D} \times \widetilde{\cal D}} =0$, we obtain
\begin{eqnarray*}
 \dt Q (g_t) |_{\,t=0} \eq (n-p) B(U, H-\tH) + (n-p)\<U , \dt H\>
 - (n-p)\<\dt \tH , U\> + np\, B(U,U)\\
 &-& n B(U^\perp , U^\perp) - 2n \<\dt U^\perp , U\> - p B(U^\top , U^\top) - 2p \<\dt U^\top , U^\top\>,
\end{eqnarray*}
that reduces to \eqref{dtQgforUconnection}.
\end{proof}


\baselineskip=12.7pt

\begin{thebibliography}{}


\bibitem{bdrs}
E. Barletta, S. Dragomir, V. Rovenski and M. Soret,
{Mixed gravitational field equations on globally hyperbolic space-times}, Classical and Quantum Gravity, 30\,:\,8, (2013)
085015, 26\,pp.

\bibitem{bdr}
E. Barletta, S. Dragomir and V. Rovenski,
{The mixed Einstein-Hilbert action and extrinsic geometry of foliated manifolds},
Balkan J. of Geometry and Its Appl. 22\,:\,1, (2017), 1--17

\bibitem{bee}
J. Beem, P. Ehrlich and K. Easley, \textit{Global Lorentzian geometry}. New York, Dekker, 1996

\bibitem{bf}
A. Bejancu and H. Farran, \textit{Foliations and geometric structures}. Springer-Verlag, 2006

\bibitem{bs}
A.\,N. Bernal and M. S\'{a}nchez, {Smoothness of time functions and the metric splitting of globally hyperbolic
space-times}, Commun. Math. Phys. {257} (2005), 43--50

\bibitem{besse}
A. L. Besse, \textit{Einstein manifolds}. Springer, 1987

\bibitem{Blairsurvey}
D.E. Blair, {A Survey of Riemannian Contact Geometry}, Complex Manifolds, 6 (2019), 31--64

\bibitem{cb19}
S. Capozziello and F. Bajardi, Gravitational waves in modified gravity. International J. of Modern Physics D, 28\,:\,05, 1942002 (2019)

\bibitem{ch12}
M. Crasmareanu and C.-E. Hretcanu, Statistical structures on metric path spaces, Chin. Ann. Math. Ser. B, {33} 
(2012), 889--902

\bibitem{FR}
Gh. Fasihi-Ramandi, Semi-symmetric connection formalism for unification of gravity and electromagnetism.
J. Geom. Phys. 144, (2019) 245--250

\bibitem{fs}
A. Fathi, Time functions revisited, International J. of Geometric Methods in Modern Physics, 12\,:\,8 (2015) 1560027 (12 pages)

\bibitem{FriedrichIvanov}
T. Friedrich and S. Ivanov, Parallel spinors and connections with skew-symmetric torsion in string theory, Asian J. of Mathematics, 6\,:\,2, (2002), 303--336

\bibitem{gps}
I. Gordeeva, V.I. Pan'zhenskii and S. Stepanov, Riemann-Cartan manifolds, J. of Math. Sci. New-York, 169\,:\,3
(2010), 342--361

\bibitem{g1967}
A. Gray, {Pseudo-Riemannian almost-product manifolds and submersions}, J. Math. Mech., 16\,:\,7, (1967), 715--737

\bibitem{mikes}
J. Mike\v{s}, et al. \textit{Differential geometry of special mappings},
Palack\'{y} Univ., Olomouc, 2019.


\bibitem{op2016}
B. Opozda, A sectional curvature for statistical structures, Linear Algebra and its Applications, 497 (2016), 134--161

\bibitem{pss-2020}
V.I. Panzhensky, S.E. Stepanov and M.V. Sorokina, Metric-affine spaces, J. of Math. Sciences, 245\,:\,5 (2020), 644--658

\bibitem{rov-m}
V. Rovenski, \textit{Foliations on Riemannian Manifolds and Submanifolds}. Birkh{\"a}user, 1998


\bibitem{r-affine}
V. Rovenski, Integral formulas for a metric-affine manifold with two complementary orthogo\-nal distributions,
Global J. Adv. Research Class. and Modern Geom., 6(1), (2017), 7--19

\bibitem{r2018}
V. Rovenski, {Einstein-Hilbert type action on space-times},
Publications de l'Institut Math\'{e}ma\-tique, Issue: (N.S.) 103 (117) (2018), 199--210

\bibitem{rov-5}
V. Rovenski, {Prescribing the mixed scalar curvature of a foliation}.
Balkan J. of Geometry and Its Applications, 24\,:\,1, (2019), 73--92

\bibitem{RWa-1}
V. Rovenski and P. Walczak, \textit{Topics in extrinsic geometry of codimension-one foliations}.
Springer Briefs in Mathematics, Springer-Verlag, 2011

\bibitem{rz-1}
V. Rovenski and T. Zawadzki, {The Einstein-Hilbert type action on foliated pseudo-Riemannian manifolds},
 J. of Math. Physics, Analysis and Geometry, 15\,:\,1 (2019), 86--121
	
\bibitem{rz-2}
V. Rovenski and T. Zawadzki, {Variations of the total mixed scalar curvature of a distribution},
Ann. Glob. Anal. Geom. 54 (2018), 87--122

\bibitem{RZconnection}
V. Rovenski and T. Zawadzki, {The mixed scalar curvature of almost product metric-affine manifolds}, Results in Math.,
(2018) 73:23

\bibitem{rze-27b}
V. Rovenski and L. Zelenko, Prescribing mixed scalar curvature of foliated Riemann-Cartan spaces,
J.~of Geometry and Physics, 126 (2018), 42--67

\bibitem{topp}
P. Topping, {\em Lectures on the Ricci flow}, LMS Lecture Notes 325, Cambridge Univ. Press, 2006

\bibitem{tr}
 A. Trautman, Einstein-Cartan theory. In: Fran\c{s}oise, J.P., Naber, G.L., Tsun, T.S. (eds.)
 Encyclopedia of Mathematical Physics, vol. 2, pp. 189--195. Elsevier, Amsterdam, 2006

\bibitem{wa1}
P. Walczak, {An integral formula for a Riemannian manifold with two or\-thogonal  complementary distributions}. Colloq. Math. {58} (1990), 243--252

\bibitem{Yano}
K. Yano, On semi-symmetric metric connection. Rev. Roum.
Math. Pures Appl. 15 (1970), 1579--1586


\end{thebibliography}
\end{document}